\documentclass{article}

\usepackage[letterpaper,top=2cm,bottom=2cm,left=3cm,right=3cm,marginparwidth=1.75cm]{geometry}\hfuzz=4pt
\usepackage[T1]{fontenc}
\usepackage{lmodern}
\usepackage{textcomp}
\usepackage{graphicx}
\usepackage{amsfonts,amsmath,amssymb,amsthm,amscd}
\usepackage{mathtools}
\usepackage{tikz}
\usetikzlibrary{decorations.pathmorphing,patterns,arrows.meta}
\usepackage{soul}
\usepackage{xcolor}
\usepackage{comment}
\usepackage{url}
\usepackage{bm}
\usepackage[toc,page]{appendix} 
\usepackage{fancyvrb} 
\usepackage{mathtools,nccmath}
\usepackage{cancel} 
\usepackage{enumitem} 
\usepackage[normalem]{ulem} 
\usepackage{fullpage}
\usepackage{dsfont} 

\usepackage[pdfencoding=auto,psdextra]{hyperref} 
\hypersetup{
  colorlinks   = true, 
  urlcolor     = black, 
  linkcolor    = blue, 
  citecolor   = blue 
}

\usepackage{booktabs} 
\usepackage{makecell}

\newcommand{\R}{\mathbb{R}}
\newcommand{\eps}{\varepsilon}

\newcommand{\littleO}{o}
\newcommand{\bigO}{O}

\theoremstyle{definition}
\newtheorem{defi}{Definition}[section]

\theoremstyle{remark}
\newtheorem{rem}[defi]{Remark}
\usepackage{etoolbox} 
\AtEndEnvironment{rem}{\qed}
\AtEndEnvironment{example}{\qed}

\theoremstyle{plain}
\newtheorem{theorem}[defi]{Theorem}
\newtheorem{lemma}[defi]{Lemma}
\newtheorem{corollary}[defi]{Corollary}
\newtheorem{prop}[defi]{Proposition}

\def\E{\mathbb{E}}

\def\Re{\mathrm{Re}}
\def\tr{\mathrm{Tr}}

\title{
Realisation of constraints in underdamped Langevin dynamics}
\author{Carsten Hartmann\thanks{Institut f\"ur Mathematik, Brandenburgische Technische Universit\"at Cottbus-Senftenberg, Konrad-Wachsmann-Allee 1, D-03046 Cottbus, Germany.\ Email:\href{mailto:carsten.hartmann@b-tu.de}{carsten.hartmann@b-tu.de}},  
Lara Neureither\thanks{Institut f\"ur Mathematik, Brandenburgische Technische Universit\"at Cottbus-Senftenberg, Konrad-Wachsmann-Allee 1, D-03046 Cottbus, Germany.\ Email:\href{mailto:neurelar@b-tu.de}{neurelar@b-tu.de}}, 
Upanshu Sharma\thanks{School of Mathematics and Statistics, University of New South Wales, Sydney 2052, Australia.\ Email:\href{mailto:upanshu.sharma@unsw.edu.au}{upanshu.sharma@unsw.edu.au}}}

\allowdisplaybreaks
\usepackage{tocloft}

\begin{document}

\newpage

\maketitle

\begin{abstract}
This article deals with the realisation of constraints in underdamped Langevin dynamics via soft-constrained dynamics. Specifically, we study systems with a large (or small) parameter that controls the constraint mechanisms, e.g. the strength of confinement forces, mass or friction coefficients, and we derive quantitative convergence results for both the constrained variables and the softly constrained dynamics on the limiting subspace. The latter can be either a spatial or a momentum or velocity subspace, depending on the underlying soft constraint mechanism; in this paper we treat only holonomic constraints, i.e.~all momentum- or velocity-level constraints are integrable. We explicitly include the initial conditions so that it is clear whether they must satisfy the constraint or not in order to realise the desired constrained dynamics. We discuss the implications of these results as well as questions related to the sampling of the corresponding conditional probability measures.  
\end{abstract}

\tableofcontents

\section{Introduction}


The realisation of constraints is a classical topic in mechanics; see \cite[Ch.~1.6]{arnoldBook} and the references therein. This is due to the fact that constraints are often understood as idealisations of some underlying physical mechanism that is possibly too complicated to be modelled in detail. For example, the relatively stiff bases Adenine, Guanine, Cytosine, and Thymine of which DNA is composed are often modelled as rigid bodies; rolling of a body on a solid surface with friction can be modelled as rolling without slipping, etc. In molecular dynamics simulations, bond constraints are routinely used as an approximation of stiff atomic bonds, with the aim of increasing the stable numerical time step \cite{barth1995algorithms,Leimkuhler-Matthews-MD-book,manning2025laplace}. Constraints also play a role in Markov Chain Monte Carlo (MCMC) for the sampling of statistical distributions with complex geometry \cite{girolami2011riemann,diaconis2013sampling,goodman-submanifold}, for the computation of conditional expectations \cite{carter1989constrained,CiccottiKapralEijnden05,hartmann2008ergodic}, the training of deep neural networks \cite{leimkuhler2022better,bellavia2025fast,kong2022momentum}, or the efficient computation of free energy profiles \cite{sprik1998free,den1998calculation,LelievreRoussetStoltz10}.

Historically, constrained mechanical systems have been classified into holonomic and non-holonomic systems where the former can be regarded as a special case of the latter, in that holonomic constraints impose restrictions on the positions (assuming that the velocities stay in the corresponding tangent space) whereas non-holonomic constraints impose restriction on positions \emph{and} velocities in such a way that the restriction is non-integrable (i.e.~the space of admissible velocities is not the tangent space of the configuration space).  
While holonomically constrained mechanical systems behave essentially like unconstrained mechanical systems on a lower-dimensional state space or phase-space, there can be subtle differences when it comes to the numerical discretisation; see e.g. \cite{leimkuhler1994symplectic} or \cite[Ch.~VII.1]{hairer2006geometric} and references given there. 

Another aspect is the realisation of holonomic constraints by actual physical mechanisms, such as  strong confining forces due to a stiff spring or large friction. This is sometimes called a soft or real constraint, in contrast to the ideal constraint that is imposed on the equations of motion using Lagrange multipliers or by switching to local coordinates \cite[Ch.~3.6]{gallavotti2013elements}. Depending on the forcing mechanism and the initial conditions, the resulting constrained dynamics may contain extra terms that are not present in an ideally constrained system. In the absence of resonances, these extra terms can often be expressed in terms of adiabatic invariants. We refer to \cite{RubinUngar57,bornemann1997homogenization,takens2006motion} for an in-depth discussion of this aspect. 

\paragraph{Relevant previous work}

While the realisation of holonomic constraints is now well understood for deterministic mechanical systems, the field of constrained stochastic dynamical systems is less developed. For deterministic systems, convergence of the dynamics with soft constraints to a system with hard (also ``real'' or ``ideal'') constraints in the limit of infinitely strong forcing can be often based on energy arguments \cite{benettin1989realization,kozlov1990realization} and homogenisation theory \cite{bornemann1997homogenization,froese2001realizing}. 
Related arguments have been exploited in the numerical analysis of highly oscillatory Hamiltonian systems, e.g.~\cite{LeBrisLegoll10,DobsonLeBrisLegoll13} and stochastic differential equations (SDEs) \cite{tao2010nonintrusive}, with the aim of eliminating the stiff parts from the equations of motion and to allow for larger time steps. 

Stochastic dynamics driven by Brownian motion is in general lacking the passivity and stability of Hamiltonian systems, due to the nature of the Brownian motion having unbounded total variation. As a consequence, the stochastic limit dynamics may show other features than its deterministic counterpart. 
For dissipative (``overdamped'') Langevin dynamics, the problem of realising holonomic constraints by adding strong gradient forces is now relatively well understood, starting from the works by Hinch \cite{hinch1994brownian} and Öttinger \cite{ottinger1994brownian} on rigid polymers and including more recent works, such as \cite{marbach2026brownian,waszkiewicz2025trimer} on Langevin equations with space-dependent mobility; convergence results for overdamped Langevin equations with constant diffusion coefficient can be found in, e.g. \cite{projection_diffusion}.

Results for underdamped Langevin equations or general SDEs are rarer than for the overdamped case. Langevin equations with hard (holonomic) constraints have been studied in Lagrange multiplier formulation \cite{vanden2006second,leimkuhler-constraint-regularization-nn}, within the so-called impetus-striction framework \cite{WalterHartmannMaddocks11}, or its dual formulation -- the Dirac bracket framework \cite{LelievreRoussetStoltz12}. For further results on numerical schemes for  non-reversible perturbations of overdamped Langevin systems on manifolds, see \cite{SharmaZhang21,zhang2022efficient}. 
To our knowledge, limits of underdamped systems under soft constraints have not been systematically studied yet. An exception is the paper \cite{Reich00} that analyses highly oscillatory Langevin equations using action-angle formulations and formal asymptotic expansions, drawing on similar arguments for deterministic Hamiltonian systems \cite{schuetteBornemann1997} and nearly integrable stochastic Hamiltonian systems \cite{kifer2001stochastic}. It has been conjectured, using statistical equipartition arguments, that the system under soft constraints converges to a system with hard constraints and an additional correction potential (``Fixman potential'') that depends on the noise coefficient, but that is independent of the initial conditions; see also \cite[Ch.~3.4]{Hartmann2007}. This should be contrasted with the deterministic case, in which the limit leads to a constrained system with a correction potential that depends on the initial data. It turns out, as we will discuss below, that the correction potential in the underdamped Langevin case may also depend on the initial data, depending on the physical mechanism by which the constraint is realised. Moreover, it is not always of the asserted ``Fixman'' form.

We mention two studies on soft constrained limits that do not fall under the Langevin category. The seminal work \cite{Katzenberger91} by Katzenberger features general results on soft constraints for semi-martingales; this work shows how SDEs on manifolds can be represented as limits of randomly perturbed dynamical systems with an asymptotically stable manifold. Yet, the proofs are non-constructive, in that the limit dynamics is not explicitly characterised, and provides no quantitative estimates. Quantitative   estimates for linear SDEs with soft constraints, with an explicit algebraic characterisation of the constrained limit SDEs and their steady states (or: invariant measures) are given in \cite{HartmannNeureitherSharma25a}.

\paragraph{Contributions of this article} 

Here we study and compare two different classes of physical mechanisms for the realisation of (holonomic) constraints for underdamped Langevin equations. The first class comprises constraint realisation by adding strong confinement terms to the system Hamiltonian that penalises deviations from the constraint, in both configuration and phase space. 
The second class of approaches is based on studying the physical parameter limits of small and large masses and the closely related limit of large damping. While configuration space (``spatial'') confinement is the traditional approach in classical mechanics (e.g. \cite{RubinUngar57,bornemann1997homogenization}) there are few works on imposing hard phase-space constraints on Langevin dynamics (e.g.~\cite{WalterHartmannMaddocks11,LelievreRoussetStoltz12,DLPSS18}), and we are not aware of any systematic accounts of phase-space confinement for Langevin equations (that suppresses the momentum oscillations of the spatial confinement); the same goes for the physical parameter limits that have a solid theoretical foundation for deterministic Hamiltonian systems (see \cite{arnoldBook} are the references therein), but  to our knowledge, have not been studied in the context of underdamped Langevin dynamics.   

Following our recent work~\cite{HartmannNeureitherSharma25a}, we confine ourselves to constraints that can be written as coordinate projections. Under suitable Lipschitz and growth conditions on the conservative part of the (nonlinear) drift terms, we present a complete quantitative analysis of the underdamped Langevin dynamics for all the aforementioned soft constraint limits, including explicit convergence rates as functions of the time horizon and the initial data. We moreover include a quantitative analysis of the orthogonal (i.e.~constrained) variables, akin to the analysis in~\cite[Appendix C]{projection_diffusion} for the overdamped Langevin equation.

\paragraph{Guiding example} 
We will now illustrate the different physical mechanisms that can be used to realise constraints on Langevin systems, and discuss connections with the relevant approaches for deterministic mechanical systems. To this end, we consider the following example that has been adapted from the classical mechanics textbook \cite[Ch.~1.6.1]{arnoldBook}: a linear SDE for $Q=(Q^1,Q^2)$ and $P=(P^1,P^2)$ that models a system of two bodies on the real axis. The two bodies are coupled by springs where one mass is attached to a wall (by a spring), and one mass is subject to damping and noise (see Figure \ref{fig:2bodyexample}):    
\begin{equation}\label{eq:2bodyexample}
    \begin{aligned}
        \begin{aligned}
            dQ^1_t & = M^{-1}P^1_t\,dt\\
            dQ^2_t & = m^{-1}P^2_t\,dt\\
            dP^1_t & = -\left(k Q^1_t + \alpha(Q^1_t-Q^2_t)\right) dt\\
            dP^2_t & = -\left(\alpha(Q^2_t-Q^1_t) + \gamma m^{-1}P^2_t \right) dt + \sigma dW_t\,.
        \end{aligned}
    \end{aligned}
\end{equation}
Here $W_t$ is a standard one-dimensional Brownian motion.
The system is an underdamped Langevin equation with quadratic Hamiltonian
\[
H(Q,P) =  \frac{1}{2}Q^TKQ + \frac{1}{2}P^\top G^{-1}P
\]
where the mass and stiffness matrices
\[
G = \begin{pmatrix}
    M & 0 \\ 0 & m  
\end{pmatrix}\,,\quad K = \begin{pmatrix}
    k+\alpha & -\alpha \\ -\alpha & \alpha  
\end{pmatrix}
\]
are easily seen to be positive definite, which implies that the Hamiltonian is strictly convex and bounded below by zero for any choice of $m,M,\alpha,k>0$.

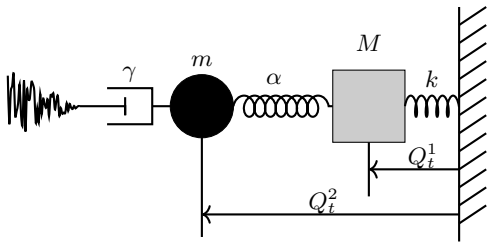
\begin{figure}[ht!]
    \centering
\begin{tikzpicture}[scale=1.2, every node/.style={font=\small}]

  \draw[thick] (7,1.1) -- (7,-1.5);
  \foreach \y in {-1.3,-1.1,...,0.9}
    \draw[thick] (7,\y) -- ++(0.3,0.2);

  \draw[fill=black!20] (5.6,-0.4) rectangle (6.4,0.4);
  \node at (6,0.7) {$M$};

  \draw[thick,decorate,decoration={coil,aspect=0.3, segment length=4pt, amplitude=4pt}]
        (6.4,0) -- (7,0);
  \node at (6.7,0.3) {$k$};

  \draw[thick] (2.8,0) -- (3.3,0) -- (3.3,-0.1) -- (3.3, 0.1);
  \draw[thick] (3.1,0.2) -- (3.6,0.2) -- (3.6,-0.2) -- (3.1,-0.2);
  \draw[thick] (3.6,0) -- (4.5,0);
  \node at (3.35,0.35) {$\gamma$};
  \draw[thick,decorate,decoration={coil,aspect=0.75, segment length=4pt, amplitude=4pt}] (4.5,0) -- (5.6,0);
  \node at (4.95,0.3) {$\alpha$};

      \draw[thick] plot [domain=2:2.8,samples=50,smooth] (\x,{0.5*(\x-2.8)*rand});

  \draw[fill=black] (4.15,0) circle (0.35);
  \node at (4.15,0.5) {$m$};

  \draw[thick] (4.15,-0.25) -- (4.15,-1.45);
  \draw[->, thick] (7,-1.2) -- (4.15,-1.2) node[anchor=north] {};
  \node at (5.5,-1.05) {$Q^2_t$};

  \draw[thick] (6,-0.4) -- (6,-1.0);
  \draw[->, thick] (7,-0.7) -- (6,-0.7) node[anchor=north] {}; 
  \node at (6.6, -0.55) {$Q^1_t$};

\end{tikzpicture}
    \caption{Two coupled point masses subject to friction and noise}
    \label{fig:2bodyexample}
\end{figure}

Let us consider the case that the spring that is attached to the wall becomes infinitely stiff, i.e. we take the limit $k\to\infty$. We start by considering the deterministic case and suppose that $\gamma=\sigma=0$. It can be shown (e.g. \cite{RubinUngar57}) that 
$Q^1_t=\mathcal{O}(1/k)$, moreover the limit motion of the second body is a harmonic oscillation of frequency $\omega=\sqrt{\alpha/m}$, provided that the initial conditions of the full system  satisfy $Q^1_0=P^1_0=0$. This implies that the limit dynamics agree with the constrained dynamics $\dot{q}=p/m$, $\dot{p}=-\alpha q$ under the constraint $Q^1=0$ (which implies the velocity constraint $\dot{Q}^1=P^1/M=0$). The total energy of the limit system,
\[
\bar{H}(q,p) = \frac{\alpha}{2}q^2 + \frac{1}{2m}p^2  \,,
\]
equals the total energy of the original system under the initial conditions $Q^1_0=0$ and $P^1_0=0$. (Note that the total energy is preserved under the deterministic dynamics.)  

Now, if $\gamma>0$ and $\sigma=\sqrt{2\gamma\beta^{-1}}$ for some $\beta>0$, then the law of the dynamics at time $t$ converges to a unique Gaussian probability measure with density $\rho\propto\exp(-\beta H)$, independently of the initial conditions. In the limit $k\to \infty$, the Gaussian  measure converges weakly to a degenerate Gaussian with density $\rho|_{Q^1=0}$, and therefore we expect the limit dynamics to be a three-dimensional system for $(Q^2,P^1,P^2)$, rather than the two-dimensional one for $(Q^2,P^2)$ as in the deterministic case. Precise statements will be given in the next section. 

The situation is different  for the noisy system if we let simultaneously $k\to\infty$ and $M\to 0$, in which case $(Q^1_t,P^1_t)\to 0$ for any $t>0$ and independently of the initial conditions. The joint Gaussian invariant measure thus becomes singular in both $Q^1$ and $P^1$ and weakly converges to a degenerate Gaussian with the non-degenerate part having the density $\bar{\rho}\propto \exp(-\beta \bar{H})$. This Gaussian density turns out to be the unique stationary density of the resulting limit dynamics
\begin{equation}\label{eq:2bodyexamplePhaseSpace}
    \begin{aligned}
        \begin{aligned}
            dq_t & = m^{-1}p_t\,dt\\
            dp_t & = -\left(\alpha q_t + \gamma m^{-1}p_t \right) dt + \sigma dW_t\,,
        \end{aligned}
    \end{aligned}
\end{equation}
which suggests that the approximation of in $(Q^2_t,P^2_t)$ by the constrained dynamics $(q_t,p_t)$ is uniform in time. 

An alternative mechanism to realise a holonomic (or even a non-holonomic) constraint is by increasing friction \cite{kozlov1990realization,eldering2016realizing}. When we add a term $-c P^1_t$ to the third equation in (\ref{eq:2bodyexample}) and take the limit $c\to\infty$, it is easy to see that $P^1_t=\mathcal{O}(e^{-ct})$ for every $t>0$. 
Therefore, by integrating the first equation,  it follows that $|Q^1_t-Q^1_0| = \mathcal{O}(1/c)$ on every finite time interval. As a consequence, $Q^1_t$ hardly departs from its initial value, which implies that $Q^1$ becomes an integral rather than a constrained variable. The resulting limit dynamics  
\begin{equation}\label{eq:2bodyexampleHF}
    \begin{aligned}
        \begin{aligned}
            dq_t & = m^{-1}p_t\,dt\\
            dp_t & = -\left(\alpha (q_t - Q^1_0) + \gamma m^{-1}p_t \right) dt + \sigma dW_t\,,
        \end{aligned}
    \end{aligned}
\end{equation}
agrees with (\ref{eq:2bodyexamplePhaseSpace}) if and only if $Q^1_0=0$. For $\gamma=\sigma=0$ the dynamics is Hamiltonian with the energy
\[
\tilde{H}(q,p;Q^1_0) =  \frac{\alpha}{2}(q-Q^1_0)^2 + \frac{1}{2m}p^2 \,.
\]
Note, however, that $\tilde{H}(q,p;\xi)\neq H(\xi,q,0,p)$. If $\gamma>0$ and $\sigma=\sqrt{2\gamma\beta^{-1}}$, then the limit dynamics has a unique Gaussian invariant measure with density $\tilde{\rho}\propto\exp(-\beta \tilde{H})$, which has no obvious relation to the original Gaussian density $\rho\propto\exp(-\beta H)$ in that it cannot be obtained by marginalising over or conditioning on $Q^1$. 

Interestingly, equation (\ref{eq:2bodyexampleHF}) can also be formally derived from (\ref{eq:2bodyexample}) by sending the mass of the first particle to infinity. Indeed, it is easy to see that $Q^1_t=\mathcal{O}(1/M)$ as $M\to\infty$. Yet, there is no damping acting on the first particle that would force $P^1_t\to 0$, and it turns out that the limit dynamics comprises $Q^2,P^2$ and $P^1$. This has consequences for the invariant measure as 
\[
\lim_{M\to\infty} \exp\left(-\beta H(Q^1,Q^2,P^1,P^2)\right) = \exp\left(-\beta\left(\frac{1}{2}Q^\top K Q + \frac{1}{2m}|P^2|^2 \right)\right)
\] 
where the right-hand side not only is not invariant under the limit dynamics, but also not a normalisable probability density function in the variables $Q^2,P^2$, and $P^1$. (The marginal distribution in $P^1$ tends to a Gaussian with infinite variance.) Physically, an infinitely large mass $M$ has the effect of preventing the corresponding particle with position $Q^1_t$ to be moved away from its initial value $Q^1_0$. Clearly $\dot{Q}^1_t\to 0$ as $M\to\infty$, still this does not imply that $P^1_t\to 0$, since $P^1_t=M\dot{Q}^1_t$ with $M\to\infty$. The large mass limit has connections to what is known as vakonomic mechanics and the method of adjoint masses, which we will briefly discuss at the end of this paragraph (see Remark \ref{rem:adjointMasses} below).  

Rather than preventing $Q^1_t$ to move away from $Q^1_0$ by making it infinitely heavy, we can also consider the small-mass limit that is in some sense dual to the former (see \cite{maddocks1995unconstrained}). 
Sending $M$ to zero will have the effect of making the first particle oscillate at a high frequency $\tilde{\omega}=\mathcal{O}(1/\sqrt{M})$ as $M\to 0$. While this implies that $P^1_t$ converges to a zero-mean Gaussian with variance $M$ as $t\to\infty$, the effect on $Q^1_t$ on finite time intervals is not obvious. Under suitable conditions, $Q^1_t$ undergoes an overdamped diffusion as will be discussed below. For deterministic Hamiltonian systems, the small mass limit has links to port-Hamiltonian descriptor systems \cite{beattie2018linear}, but also to Dirac's theory of constraints as has been pointed out in \cite{deryabin2000dirac}.   

\begin{rem}\label{rem:adjointMasses}
    The idea of realising constraints by friction can be generalized; see, e.g.~\cite{kozlov1990realization}: if a term $-c \bm{M}^{-1}A^\top A \bm{M}^{-1} P_t$  is added to the equation where $A=A(Q)$ is assumed to have a  non-trivial kernel and $\bm{M}$ is a  symmetric positive definite mass matrix, then the limit dynamics as $c\to\infty$ will evolve under the constraint $A(Q_t)\dot{Q}_t=0$; depending on whether $A$ can be written as the Jacobian of some function $\xi$ or not, this gives rises to a holonomic constraint  $\xi(Q)=0$ or a non-holonomic constraint $A(Q)\dot{Q}=0$. 

    The high friction and large mass limits have strong connections to the theory of Vakonomic mechanics \cite[Ch.~1.4]{arnoldBook} and the impetus-striction formulation of constrained mechanical systems \cite{maddocks1995unconstrained,dichmann1996hamiltonian}, which both lead to formulations with degenerate Hamiltonians of the form
    \[
    \mathcal{H}(Q,P) = H(Q, \mathcal{P}_Q P)\,,\quad \mathcal{P}_Q = I - A^\top(A \bm{M}^{-1}A^\top)^{-1}A\bm{M}^{-1}
    \] 
    where $A=\nabla\xi$ in case of a holonomic constraint of the form $\xi(Q)=0$. Generalisations of the impetus-striction formulation for underdamped Langevin equations can be found in \cite{latorre2010free,WalterHartmannMaddocks11}.
\end{rem}

\paragraph{Outline} The rest of the paper is organised as follows: The next Section \ref{sec:mainresults} introduces the Langevin set-up considered in this paper, the main results and underlying assumptions. Section \ref{sec:InfConf} analyses the realisation of constraints by strong confinement, whereas Section \ref{sec:PhysicalLimits} is devoted to a discussion of friction and mass parameter limits (``physical parameter limits''). Possible generalisations of the results are discussed and contextualised in Section \ref{sec:discussion}. The main proofs are deferred to the  Appendix that also records a couple of auxiliary results and examples for the spatial confinement case.

\section{Main results}\label{sec:mainresults}

We briefly introduce the set-up of this paper and discuss the main finding of this work. To this end, we introduce the underdamped Langevin equation (we will drop the prefix underdamped from here onwards)
\begin{align}\label{eq:intro-genLang}
    \begin{aligned}
    dQ_t &= \nabla_p H(Q_t,P_t)dt\\
    dP_t &= -\nabla_q H(Q_t,P_t)dt - \gamma \nabla_p H(Q_t,P_t)dt + \sqrt{2\gamma\beta^{-1}}dW_t,
    \end{aligned}
\end{align}
where $(Q_t,P_t)\in \R^{d}\times \R^d$ are the position and momenta, $\gamma>0$ is the friction coefficient, $\beta$ is the inverse temperature, and $W_t$ is the standard Brownian motion in $\R^d$. Here $H:\R^{2d}\to\R$ is the \emph{Hamiltonian} 
\begin{equation*}
    H(Q,P)=V(Q) + \frac{1}{2}|P|^2
\end{equation*}
where, for simplicity, we have assumed that the particles have unit mass, and $V:\R^d\to \R$ is the potential. The Langevin equation is often the canonical choice for sampling the Boltzmann density $\rho=Z^{1}\exp(-\beta H)$, where $Z$ is a normalisation constant that makes $\rho$ a probability density function. 
Now let 
\begin{equation*}
    \xi:\R^d\to \R^k, \ \ Q\mapsto \xi(Q),
\end{equation*}
be a smooth function, with regular value $0\in\R^k$, such that the level set $\xi^{-1}(0)=\{Q\in\R^d\colon \xi(Q)=0\}$ is a smooth submanifold of codimension $k$ in $\R^d$. We call  $\xi^{-1}(0)$ the constraint manifold.

A key challenge in this context is to sample the Boltzmann distribution conditioned on the level sets of $\xi$, for example to compute free energy profiles using thermodynamic integration or so to disintegrate an otherwise intractable probability measure. 
In numerical simulations, this challenge is often addressed by adding a confining force to the potential (e.g. see \cite{kastner2011umbrella}), considering 
a Hamiltonian of the form
\begin{equation}\label{eq:stuff-Ham-spat}
    H(Q,P) = V(Q) + \frac{1}{2\eps} | \xi(Q)|^2 + \frac12 |P|^2,
\end{equation}
where $0< \eps \ll 1$. The idea here is that the confinement potential $(2\eps)^{-1}|\xi|^2$ forces the dynamics to the zero level set of $\xi$ as $\eps\to 0$. Clearly, if the constraint is of the form $\xi^{-1}(a)\subset\R^d$, we can replace the confinement potential by $(2\eps)^{-1}|\xi(Q)-a|^2$, assuming that $a\in\R^k$ is a regular value of $\xi$. Nevertheless, we focus on the choice $a=0$ in this article since $a\neq 0$ does not change the fundamental behaviour of what is going to follow.

\paragraph{Strong confinement of Langevin dynamics}

In this paper we study the $\eps\to 0$ limit of the Langevin dynamics with stiff Hamiltonian \eqref{eq:stuff-Ham-spat}, and show that the corresponding momenta oscillate in the $\eps\to 0$ limit (see Appendix~\ref{app:oscP1} for a simple example). This behaviour is reminiscent of oscillating momenta in the deterministic, purely Hamiltonian setting that can be treated by weak convergence techniques \cite{bornemann1997homogenization,klar2021second}. While the derived limiting Langevin dynamics does sample the correct steady state on the level sets of $\xi$ (cf.~the introductory example), the oscillating momentum implies that the stochastic dynamics itself never stabilizes. 
To deal with this issue, in this paper we study a second, so-called \emph{phase-space} constraint which corresponds to a Hamiltonians of the form
\begin{equation}\label{eq:stuff-Ham-phase}
    H(Q,P) =V(Q) + \frac{1}{2\eps} | \xi(Q)|^2 + \frac12 |P|^2 + \frac{1}{2\eps} |\nabla\xi^\top(Q) P|^2.
\end{equation}
The idea of this construction is that, if we want to constrain a mechanical system to the set $\xi(Q)=0$, then the corresponding velocity, which for a system with unit mass matrix is formally given by $\frac{d}{dt}\xi(Q) = \nabla\xi^\top(Q)\dot Q =\nabla\xi^\top(Q)P$, should also be zero since the motion will be constrained to the corresponding (co)tangent space. This is sometimes called a \emph{hidden constraint}.   

With this additional term, one would expect that oscillations in momenta as $\eps\to 0$ do not occur, and therefore we have a well-defined soft-constraining limit for the Langevin dynamics. This indeed turns out to be the case, as will be discussed below. We point out that such momentum-based constraints have appeared in the context of free energy calculation \cite{LelievreRoussetStoltz10,LelievreRoussetStoltz12} and coarse-graining \cite{DLPSS18,nateghi2025consistent}. Yet, to our knowledge, this work is the first attempt to address the problem of highly oscillatory momenta in Langevin dynamics, beyond the formal asymptotics approach \cite{Reich00}.  


As a starting point, this paper focusses on the $\eps\to 0$ limit of the underdamped Langevin dynamics in the case of linear/affine constraints. 
It has been demonstrated in the mechanics literature (e.g.~\cite{kozlov1990realization,eldering2016realizing}) that constraints can also be realised via the infinite scale-separation limit of certain physical parameters such as friction and mass. As will be discussed below, in this work we also analyse the limits of such physical parameters for the Langevin dynamics. 
%
We will follow a pathwise approach  which allows us to handle the restrictive case of purely spatial constraint described above. We first analyse the limit of the constrained variables, characterised by the linear map $\xi(Q)$. This limit is then used to study the dynamics of the remaining unconstrained variables, including a discussion of the invariant measures arising out of the aforementioned constraining procedures (either via adding a stiff potential or passing to the limit in a physical parameter).
It should be pointed out that the only literature on soft-constraint limits for Langevin dynamics~\cite{Reich00} analyses the behaviour of the unconstrained variables  only. Alternatively, the general martingale approach of Katzenberger~\cite{Katzenberger91}  provides qualitative results only (some quantitative results have recently been derived in a companion paper~\cite{HartmannNeureitherSharma25a}), while in the present paper we provide quantitative pathwise estimates for the evolution of the unconstrained variables and pointwise in time estimates for the evolution of the constrained variables. Compared to related results  for overdamped systems, such as \cite[Appendix C]{projection_diffusion} that assumes bounded gradients for the potential $V$, we can slightly relax the assumptions on the coefficients; see Subsection \ref{sec:Ass-Pot} below. 


\subsection{Set-up: coordinate projection constraints and physical parameter limits} \label{sec:limit-results}

In this section we introduce notation and give an overview of our results briefly discussed above. 
Throughout this article we will work with \emph{coordinate-projection} constraints, i.e.\ $\xi:\R^d\to \R^k$ defined as
\begin{equation*}
      \R^d \ni Q=(Q^1,Q^2) \mapsto \xi(Q) = Q^1 \in \R^k, \ \ \text{ with } Q^1\in \R^k, \ Q^{2}\in \R^{d-k}.  
\end{equation*}
In line with the notation above, throughout this article we will use $X=(X^1,X^2) \in \R^{k}\times \R^{d-k}$. In order to avoid confusion, we write $|\cdot|^2$ whenever we raise something to the power of 2. We point out that while the results in this article generalise to affine constraints (see the discussion in Section \ref{sec:affine}), we restrict to the coordinate-projection for simplicity of presentation. We use $\nabla$ for the gradient in $\R^d$ and $\nabla_{q^1}$, $\nabla_{q^2}$ for the gradients in $\R^k$ and $\R^{d-k},$ respectively. 

To get an overview of the results, we introduce an explicit  Langevin equation (recall the general form in~\eqref{eq:intro-genLang}) which captures all considered constraint mechanisms and which is given by
\begin{subequations}\label{eq:abc}
    \begin{align}
        \begin{split}\label{eq:abc-fast}
        dQ^{1}_t &=  a P^{1}_t dt \\
        dP^{1}_t &= -\nabla_{q^1} V(Q_t) dt -b  Q^{1}_t dt - c \gamma P^{1}_t dt + \sqrt{2\gamma\beta^{-1}} \,dW^1_t 
        \end{split}\\
        \begin{split}\label{eq:abc-slow}
        dQ^{2}_t &=   P^{2}_t dt    \\
        dP^{2}_t &= -\nabla_{q^2} V(Q_t) dt - \gamma P^{2}_t dt + \sqrt{2\gamma\beta^{-1}} \,dW^2_t.
        \end{split}
    \end{align}
\end{subequations}
Here $a,b,c>0$ are parameters that will be specified in the following section and that can change according to the chosen constraint mechanism. 
Moreover, we introduce the limit dynamics for $(q_t,p_t) \in \R^{d-k} \times \R^{d-k}$:
\begin{align}\label{eq:limit-gen}
\begin{aligned}
    dq_t &= p_t dt\\
    dp_t &= - \nabla_{q^2} V(\hat q, q_t) dt - \gamma p_t dt + \sqrt{2  \gamma\beta^{-1}} dW_t^2\,. 
\end{aligned}
\end{align}

Here $W^1$ and $W^2$ are standard Brownian motions in $\R^k$ and $\R^{d-k}$ respectively and $V,\beta,\gamma$ are as before. The variable $\hat q$ in \eqref{eq:limit-gen} can be constant or evolve according to yet another limit dynamics. Note that the dynamics of $(Q^1,P^1)\in\R^{2k}$ and $(Q^2,P^2)\in \R^{2(d-k)}$ only see each other via the force $\nabla V(Q^1,Q^2)$. 

In this article we are interested in the behaviour of $(Q_t,P_t)$ in the limit of infinite scale-separation, typically indicated by some small parameter  $0<\eps\ll 1$ which will enter the dynamics through the choice of $a,b,c$ in \eqref{eq:abc-fast}. We drop the explicit dependence on $\eps$ of these capitalised variables to simplify notation. Table~\ref{tab:summary} below summarises the various limits studied in this article. A detailed explanation will be provided in the next section.  

We assume that the force $\nabla V$ is Lipschitz, which ensures the existence of unique strong solutions to the Langevin dynamics. Furthermore, to avoid technical compactness arguments, we assume that the $q^1$-gradient of the non-quadratic part of the potential is bounded. For a detailed discussion, about the assumption on the potential $V$ see Section \ref{sec:Ass-Pot}.  
We point out that in order to keep the presentation of our results as clear as possible, we present them for initial conditions, which are deterministic and do not depend on $\eps$. Yet, our proofs are written for the general case.

\begin{table}[ht]
\centering
{\renewcommand{\arraystretch}{1.5}
\resizebox{\textwidth}{!}{%
\begin{tabular}{@{} l l l l @{}} 
\toprule
\makecell[l]{Section} &
\makecell[l]{Pre-limit dynamics} & 
\makecell[l]{Limit for constrained variables } & 
\makecell[l]{Limit dynamics for unconstrained variables} \\
\midrule
\ref{sec:spatial} & Spatial confinement~\eqref{eq:intro-Spat-Cons-Fast}  & $Q^1_t\to 0$ if $Q^1_0=0$ else oscillates    &  if $Q^1_t\to 0$, $(Q^2_t,P^2_t)\to$\,\eqref{eq:limit-gen} with $\hat q=0$; $P^1_t$ oscillates \\
\ref{sec:phase-space} & Phase-space confinement~\eqref{eq:intro-Phase-Cons-Fast} & $(Q^1_t,P^1_t)\to (0,0)$    & $(Q^2_t,P^2_t)\to$\,~\eqref{eq:limit-gen} with $\hat q=0$\\
\ref{sec:zeromass} & Zero mass~\eqref{eq:intro-0Mass-Fast}  & $P^1_t\to 0$  & $Q^1_t\to$\,\eqref{eq:0Mass-q1}, $(Q^2_t,P^2_t)\to$\, \eqref{eq:limit-gen} with $\hat q=Q^1_t$ \\
\ref{sec:highmass} & Infinite mass~\eqref{eq:intro-InfMass} & $Q^1_t\to Q^1_0$  & $P^1_t\to$\,\eqref{eq:InfMass-p1}, $(Q^2_t,P^2_t)\to$\,\eqref{eq:limit-gen} with $\hat q=Q^1_0$ \\
\ref{sec:highfric} & Infinite friction ~\eqref{eq:intro-InfFriction-withFD-Fast}
& $Q^1_t\to Q^1_0$ &  $(Q^2_t,P^2_t)\to$\,\eqref{eq:limit-gen} with $\hat q=Q^1_0$ \\
\bottomrule
\end{tabular}
}
}
\caption{ This table summarises the various asymptotic limits studied in this article. 
The pre-limit dynamics in the second column refers to the $(Q^1_t,P^1_t)$ dynamics \eqref{eq:abc-fast} for each of the cases studied in this article -- the first column indicates the respective sections where the precise error bounds can  be found. 
In all the aforementioned cases the dynamics of $(Q^2_t,P^2_t)$ is given by~\eqref{eq:abc-slow}. The third column collects the \emph{constrained} variables 
and their corresponding limits as $\eps\to 0$. The fourth column collects the remaining \emph{unconstrained} variables and the corresponding limit dynamics.}
\label{tab:summary}
\end{table}

\subsubsection{Confinement via stiff potential}\label{sec:intro-StiffConf}
Adding stiff forces to the system by changing the Hamiltonian is the classical approach to constraining SDEs to manifolds. In the following, we discuss two such choices: the first approach involves adding confinement in the position variable only and is referred to as  \emph{spatial confinement};  the second involves adding a confinement potential in both position and momentum variables and is referred to as \emph{ phase-space confinement} in the following. For a deeper connection between these two choices and possibly non-linear coarse-graining maps, see the discussion in Section~\ref{sec:Disc-GenForm}. 
\paragraph{Spatial confinement.}
Using the Hamiltonian~\eqref{eq:stuff-Ham-spat} with $\xi(Q)=Q^1$, we have $a=1, \  b=\frac1\eps, \ c=1$, and therefore 
the $(Q^1,P^1)$-dynamics~\eqref{eq:abc-fast} is (note that $(Q^2,P^2)$ remains unchanged) 
\begin{align}\label{eq:intro-Spat-Cons-Fast}
\begin{aligned}
     dQ^{1}_t &=  P^{1}_t dt   \\
    dP^{1}_t &= -\nabla_{q^1} V(Q_t) dt -\frac1\eps  Q^{1}_t dt -  \gamma P^{1}_t dt + \sqrt{2\gamma\beta^{-1}} \,dW^1_t.
\end{aligned}
\end{align}
Here and in what follows, we drop the explicit dependence of $(Q^1,P^1)$ on $\eps$. 
As expected, $Q^1_t\to 0$ as $\eps\to 0$ (for the exact mode of convergence see Theorem~\ref{prop:SpatCons}), however this only holds for well-prepared initial datum $Q^1_0=0$. Furthermore, the constrained momentum $P^1_t$ fluctuates for any choice of initial data -- see Proposition~\ref{prop:SpatCons}  for the exact result and Appendix~\ref{app:oscP1} for a simple example. This fluctuation or oscillation is in line with similar findings in the Hamiltonian literature~\cite{bornemann1997homogenization}; see Remark \ref{rem:SpatCompLit} for further references. 

The unconstrained dynamics $(Q^2_t,P^2_t)$ in~\eqref{eq:abc-slow} converges, as expected, to the limit dynamics $q_t,p_t$ in~\eqref{eq:limit-gen}. However, this convergence only holds if $Q^1_t\to 0$, which in turn holds if $Q^1_0=0$. Let us already mention that under certain conditions on the potential, the constrained as well as the unconstrained variables have a \emph{surprising uniform-in-time convergence}, see Remark \ref{rem:Unif-InTime} for details. This result uses Gronwall's inequality and is therefore remarkable as usually these estimates grow exponentially with the considered time horizon. 

These results indicate that the spatially-confined Langevin dynamics has a reasonable spatial limit only for well-prepared initial data while the corresponding momentum always oscillates. Consequently, using this dynamics for sampling steady states can lead to numerical errors. This motivates the phase-space constrained Langevin dynamics which we discuss in the sequel.

\paragraph{Phase-space confinement.}
Using the Hamiltonian~\eqref{eq:stuff-Ham-phase} with $\xi(Q)=Q^1$, we have $a=1+ \frac1\eps, \  b=\frac1\eps, \ c=1+\frac1\eps$ in \eqref{eq:abc-fast}, and therefore 
the $(Q^1,P^1)$ dynamics is (note that $(Q^2,P^2)$ stays unchanged) 
\begin{align}\label{eq:intro-Phase-Cons-Fast}
\begin{aligned}
     dQ^{1}_t &=  \bigg(1+ \frac{1}{\eps}\bigg)P^{1}_t dt   \\
    dP^{1}_t &= -\nabla_{q^1} V(Q_t) dt -\frac1\eps  Q^{1}_t dt -  \gamma\bigg(1+ \frac1\eps\bigg) P^{1}_t dt + \sqrt{2\gamma\beta^{-1}} \,dW^1_t.
\end{aligned}
\end{align}
In stark contrast to the spatially-constrained case, for \emph{any} initial datum both constrained state-variables $(Q^1_t,P^1_t)\to 0$ as $\eps\to 0$, see Theorem~\ref{prop:PS-Cons}. In fact, as outlined in Sec.~\ref{sssec:PS-Init}, these convergence results even hold for initial conditions, for which the total energy diverges as $\eps\to 0$. This analysis suggests that the dynamics is attracted to the constraint manifold $(Q^1_t,P^1_t)=0$ for any initial datum. This behaviour is different from the behaviour of spatially confined mechanical systems, for which the limiting orbit in such cases may not even lie on the constraint manifold (see \cite{bornemann1997homogenization,froese2001realizing}).

Similar to the spatially-constrained setting, the unconstrained variables $(Q^2_t,P^2_t)$ also converge to the limit $(q_t,p_t)$ as given in~\eqref{eq:limit-gen}, but now for \emph{any} initial datum of the full dynamics $(Q_t,P_t)$. Moreover, the \emph{surprising uniform-in-time convergence} for the constrained as well as the unconstrained variables holds, under the same assumptions as in the spatially-constrained case. 

To the best of our knowledge, the $\eps\to 0$ limit (which we call the soft-constrained limit) of the phase-space confined Langevin dynamics has not yet been studied in the literature. 

\paragraph{Steady states.} Sampling invariant probability densities or steady states (typically the Boltzmann or Gibbs-Boltzmann density) on manifolds is one of the key motivation for working with soft- and hard-constrained Langevin dynamics. Consequently, it is important to understand the steady states for the limiting dynamics discussed above. 
The Langevin dynamics~\eqref{eq:intro-genLang} admits the unique steady state 
\begin{equation*}
     \mathcal P(\R^{2d})\ni \mu(dQ,dP) = Z^{-1} e^{-\beta H(Q,P) }\,dQdP,  
\end{equation*}
where $Z$ is the normalisation constant and the Hamiltonian $H$ is given by \eqref{eq:stuff-Ham-spat} for spatially-constrained dynamics or \eqref{eq:stuff-Ham-phase} for the phase-space constrained dynamics. 

For both spatial or phase-space confinement that come with uniform-in-time estimates (i.e. they hold for all $t\in[0,\infty)$), it turns out that the limits $\eps\to 0$ and $t\to\infty$ commute. Consequently, by formally letting $\eps\to 0$ in the full Boltzmann measure we arrive at the following steady states for the unconstrained variables: 
In the spatial constraining case, under suitable initial conditions, $Q^1_t$ converges to zero, and in the limit the remaining unconstrained variables $P^1_t$ and $(Q^2_t,P^2_t)$ admit the following steady state in the limit $\eps\to 0$, 
\begin{equation*}
    \mu(dq^2,dp^1,dp^2) = \frac{1}{Z}\exp\biggl(-\beta\biggl[V(0,q^2) +\frac12 |p|^2\biggr]\biggr)
\end{equation*}
with normalisation constant $Z$ that is different from $Z$ above. (In an abuse of notation we use the same variable $Z$ for different normalisation constants.) This can be seen by formally sending $\eps\to 0$ in the invariant measure for the full dynamics. Note that, even though $P^1_t$ oscillates, it still admits an invariant measure.  

A similar analysis for the phase-space constrained case, where $(Q^1_t,P^1_t)\to 0$, reveals that in the limit $\eps\to 0$, the remaining unconstrained variables $(Q^2_t,P^2_t)$ admit the steady state 
\begin{equation*}
    \mu(dq^2,dp^2) = \frac{1}{Z}\exp\biggl(-\beta\biggl[V(0,q^2) +\frac12 |p^2|^2\biggr]\biggr)
\end{equation*}
with (another) normalisation constant $Z$. A detailed analysis of the steady states is given in Section~\ref{sec:Steady}.

\subsubsection{Constraining via physical parameter limits}\label{sec:intro-PhyLimits}
In the following we take a different perspective on realising constraints, in that we consider limits of certain physical parameters. For instance, the constraint $Q^1_t \equiv Q^1_0 $, can be realised by sending the corresponding friction to infinity, i.e. over-damping parts of the system. We discuss several such limits below.   

\paragraph{Zero mass limit} 
We first study the zero mass limit, which is some form momentum confinement. The zero mass limit corresponds to considering the Hamiltonian 
\begin{align}\label{eq:Ham-ZeroMass}
    H(Q,P) = V(Q) + \frac{1}{2} \biggl( |P|^2 + \frac{1}{\eps} |P^1|^2 \biggr),
\end{align}  
that includes a stiff potential in $P^1$ that penalises deviations from $P^1=0$ in the limit $\eps\to 0$. 
The Hamiltonian in \eqref{eq:Ham-ZeroMass} can be written as 
\[
H(Q,P) = V(Q) + \frac{1}{2} P^\top M^{-1} P\,, \
\text{ with } 
M = \begin{pmatrix}
    \frac{\eps}{\eps +1}I_k & 0 \\ 0 &I_{d-k}
\end{pmatrix} \in \R^{d\times d}\quad 
\] 
and the entries correspond to the mass of the first $k$ particles, which tends to zero as $\eps \to 0.$ In the formulation of~\eqref{eq:abc-fast} this corresponds to $a=c=1+\frac1\eps$, $b=0$ so that
\begin{align}\label{eq:intro-0Mass-Fast}
\begin{aligned}
     dQ^{1}_t &= \bigg(1+\frac1\eps\bigg) P^{1}_t dt  \\
    dP^{1}_t &= -\nabla_{q^1} V(Q_t) dt -   \gamma\bigg(1+\frac1\eps \bigg) P^{1}_t dt + \sqrt{2\gamma\beta^{-1}} \,dW^1_t.
\end{aligned}
\end{align}
As expected from the momentum confinement perspective $P^1_t \to 0$ as $\eps \to 0.$ Then, in accordance with classical results (e.g.~\cite[Ch.~10]{Nelson01}) on overdamped limits for underdamped Langevin dynamics,  $Q^1_t$ converges to $\hat q_t$ that evolves according to the overdamped Langevin equation 
\begin{align}\label{eq:0Mass-q1}
    d\hat q_t = - \frac{1}{\gamma} \nabla_{q^1} V(\hat q_t,q_t)  dt + \sqrt{\frac{2\beta^{-1}}{\gamma}} dW_t^1
\end{align}
where $q_t$ is the solution to \eqref{eq:limit-gen}. Overall, we can show that $(Q^1_t,Q^2_t,P^2_t) \to (\hat q_t,q_t,p_t)$, where $\hat q_t$ is given above and $(q_t,p_t)$ are the solution to \eqref{eq:limit-gen}, as $\eps \to 0$  -- see Proposition \ref{prop:Light} for details. 
Here, in contrast to the phase-space confinement result, we can only prove a pointwise in time estimate. Moreover, the convergence is not uniform in time, but the error bound grows exponentially in time.

\paragraph{Infinite mass limit}
Next, we study the infinite mass limit, in which the mass of the particles with position and momentum given by $Q^1,P^1$ is increased to infinity. Intuitively, this should lead to a constant position $Q^1$ and is somewhat similar to the infinite friction case, which we discuss below. From a Hamiltonian perspective, this corresponds to the Hamiltonian 
\begin{align}\label{eq:Ham-InfMass}
    H(Q,P) = V(Q) + \frac{1}{2} P^\top M P, \quad \text{where } M = \begin{pmatrix}
        \frac{1}{\eps} I_k & 0 \\ 0 & I_{d-k}
    \end{pmatrix} \in \R^{d \times d}
\end{align}
meaning the mass of the particles with position and momenta $(Q^1,P^1)$ tends to infinity as $\eps \to 0.$
In~\eqref{eq:abc-fast} this corresponds to the choices $a=c=\eps, b = 0$  and so the dynamics of the constrained variables reads
\begin{align}\label{eq:intro-InfMass}
\begin{aligned}
     dQ^{1}_t &= \eps P^{1}_t dt   \\
    dP^{1}_t &= -\nabla_{q^1} V(Q_t) dt -   \gamma\eps P^{1}_t dt + \sqrt{2\gamma\beta^{-1}} \,dW^1_t.
\end{aligned}
\end{align}
Indeed, and as expected, $Q^1_t \to Q^1_0$ as $\eps \to 0$. 
Nevertheless, $P^1_t \nrightarrow 0$ , instead $P^1_t$ converges for any $t>0$ to a process $p^1_t$ which is the solution to 
\begin{align}\label{eq:InfMass-p1}
dp^1_t = - \nabla_{q^1}V(Q^1_0,Q^2_t) dt + \sqrt{2\gamma\beta^{-1}} dW_t^1\,.
\end{align}
The unconstrained variables $Q^2_t,P^2_t$ converge to the solution of \eqref{eq:limit-gen} with $\hat q = Q^1_0$ for any initial condition and in a pathwise sense, again not independent of the considered time-horizon; see Proposition \ref{prop:heavy} for details. 

\paragraph{Infinite friction limit and the fluctuation-dissipation relation}
We first study the infinite friction (i.e. large damping) limit \emph{with}  fluctuation-dissipation relation, which corresponds to the parameters  $a=1, b = 0, c = \frac{1}{\eps}$ in~\eqref{eq:abc-fast}, and an additional $\eps^{-\frac12}$-scaling of the noise term (i.e.\ the temperature for the constrained variables also scales like $\eps$), which leads to 
\begin{align}\label{eq:intro-InfFriction-withFD-Fast}
\begin{aligned}
     dQ^{1}_t &=  P^{1}_t dt  \\
    dP^{1}_t &= -\nabla_{q^1} V(Q_t) dt -   \frac\gamma\eps P^{1}_t dt + \sqrt{\frac{2\gamma\beta^{-1}}{\eps}} \,dW^1_t.
\end{aligned}
\end{align}
while $(Q^2_t,P^2_t)$ stay unchanged. Note that $\eps$ appears in both the linear damping term (second term in $P^1_t$ evolution) and the noise term, such that the fluctuation-dissipation relationship holds for all $\eps$. 

In this case, as the physical intuition of infinite damping suggests, $Q^1$ remains constant, i.e. $Q^1_t \to Q^1_0$ as $\eps \to 0$ for any $t>0$; see Proposition~\ref{prop:PartialMom}. We expect $P^1_0$ to have zero mean, but non-trivial variance as $\eps\to 0$ since $P^1_t \sim \mathcal{N}(0, \beta)$ for small $\eps$, which is captured by our results as well. Moreover, the unconstrained variables $Q^2,P^2$ converge pathwise  to the solution of \eqref{eq:limit-gen} with $\hat q = Q^1_0$ for any initial condition. 

A related setup is that of infinite  friction \emph{without} fluctuation-dissipation,  which corresponds to $a=1, b = 0, c = \frac{1}{\eps}$ in~\eqref{eq:abc-fast}, without any change in the noise term, which leads to 
\begin{align}\label{eq:intro-InfFriction-NoFD-fast}
\begin{aligned}
     dQ^{1}_t &=  P^{1}_t dt  \\
    dP^{1}_t &= -\nabla_{q^1} V(Q_t) dt -   \frac\gamma\eps P^{1}_t dt + \sqrt{2\gamma\beta^{-1}} \,dW^1_t,
\end{aligned}
\end{align}
while $(Q^2_t,P^2_t)$ stay unchanged as before. Note that in this case the noise is not scaled by $\eps$, i.e.\ the fluctuation-dissipation in violated here. In this case, $Q^1\to Q^1_0$ and $P^1\to 0$ where the latter follows since the noise is $O(\eps)$ and therefore does not contribute to the limit.

\subsection{Main assumptions and consequences}\label{sec:Ass-Pot}

We prove quantitative estimates for the asymptotic limits discussed earlier. In addition to requiring well-posedness of strong solutions to the Langevin dynamics~\eqref{eq:intro-genLang} we need certain growth conditions on $V$ for our analysis. We state the assumptions here already for two reasons: first, these results will help outline our proof techniques in the next section, and second, there is a crucial difference between the convergence results for constrained and unconstrained variables and this difference depends on the particular form of the potential
\begin{equation*}
    V\in C^2(\R^d,\R), \ \ \R^{k}\times \R^{d-k} \ni (q^1,q^2)\mapsto V(q^1,q^2)\,.
\end{equation*}
Specifically, we assume that $V$ is of the form
\begin{equation}\label{eq:Valpha}
    V(q) = \frac{\alpha}{2}|q|^2 + U(q) \geq 0,
\end{equation}
where $\alpha\geq 0$ and $U\in C^2(\R^d,\R).$  Additionally, the following is assumed throughout this article:
\begin{enumerate}[label=(\Alph*)]    \item\label{item:assVLip} $\nabla U:\R^d\to \R^{d}$ is globally Lipschitz, i.e.\ there exists a constant $L_U$ such that for any $q,\tilde q\in\R^d$ we have 
    \begin{equation}\label{ass:VLip}
        \bigl|\nabla U(q) - \nabla U(\tilde q)\bigr| \leq L_U|q-\tilde q|.
    \end{equation}
    \item\label{item:assVbound} $\nabla_{q^1} U:\R^d\to \R^k$ is bounded, i.e.\ there exists a constant $C_U$  such that for any $q\in\R^d$ we have
    \begin{equation}\label{ass:Vbound}
        \bigl|\nabla_{q^1} U(q) \bigr| \leq C_U.
    \end{equation}
\end{enumerate}

Let us comment on these assumptions. Classical literature dealing with Langevin equations (especially when dealing with steady states) assumes that there exist constants $c_1,c_2>0$ such that 
\begin{equation}\label{eq:genVAss}
     c_1(1+|q|)\leq | V(q)|\leq c_2(1+|q|^2)
\end{equation}
where the super-linear growth ensures existence of a steady state proportional to $e^{-\beta V}$ while the sub-quadratic growth ensures that the Langevin SDE has Lipschitz coefficients, which guarantees existence of a unique strong solution. 
This article deals with strong solutions constructed via variation of constants, and therefore we require that $q\mapsto \nabla V(q) =\alpha q + \nabla U(q)$ is Lipschitz, which explains Assumption~\ref{item:assVLip}. 

Employing variation of constants, we use the linear part of the drift, which explicitly involves $\alpha$ (note that there are other linear terms in Langevin dynamic~\eqref{eq:intro-genLang} even if $\alpha=0$). It turns out that the parameter $\alpha$ can lead to significant changes in the quantitative results, in particular for the unconstrained variables. Specifically, for constraining via spatial and phase-space confinement, the convergence turns out to be,~\emph{uniform-in-time} even though the proof employs a Gronwall-type argument. This is remarkable, since most Gronwall-type quantitative estimates for pathwise asymptotic limits in SDEs explode exponentially in time. The qualitative $\alpha$-dependence of the convergence behaviour is the reason to include $\alpha$ explicitly in the potential $V$. Note, that we could allow for different quadratic growth in the different $q$ variables, i.e. $V(q) = \frac{1}{2}\sum_{i=1}^d \alpha_i q_i^2,\ \alpha _i \geq 0$ which would not change the results. For simplicity of presentation we restrict ourselves to $\alpha_i \equiv \alpha$. For a detailed explanation of the uniform in time behaviour see Remark \ref{rem:Unif-InTime}.

Assumption~\ref{item:assVbound} is a technical requirement made to simplify the analysis. In particular, this assumption is required when dealing with the constrained variables, for instance $Q^1$, where we need to bound terms of the type 
\begin{equation*}
    \int_0^t f_1(s,\eps)|\nabla_{q^1} U(Q^\eps_s)|ds
\end{equation*}
as $\eps\to 0$  (see the general proof strategy outlined in Section~\ref{sec:Strategy-constrained} and the proof of Proposition~\ref{prop:SpatCons}). A more general approach to deal with these integrals would be to discuss compactness properties of the sequence $(Q^{\eps}_s)_{\eps>0}$ as $\eps\to 0$ and to then pass to the limit in these integrals.  We choose the simpler approach by assuming~\ref{item:assVbound} to avoid technical details as this does not change the nature of the soft-constraint limits themselves, but only the underlying analysis. We should point out that while this compactness analysis is involved, it is well understood (see for instance~\cite[Sec.~4]{Katzenberger91} which studies compactness properties in the context of stiff drifts in semi-martingales). We also point out that Assumption~\ref{item:assVbound} (with $\alpha=0$) is used in~\cite[Eq.~(C.3)]{projection_diffusion} which deals with soft-constraining limits for the overdamped Langevin dynamics. In fact, the aforementioned compactness analysis would allow for the more general assumption~\eqref{eq:genVAss} on $V$.

Finally, we point out that Assumptions~\ref{item:assVLip}-\ref{item:assVbound} along with $\alpha>0$ imply~\eqref{eq:genVAss}, and therefore are sufficient to ensure that the corresponding Boltzmann-Gibbs distribution is a probability measure, i.e.\ $e^{-\beta V}$ is integrable. This is another reason we have included $\alpha$ in the assumption on $V$.

\subsubsection{Proof strategy}

In this section, we outline the proof strategy and implications of the error estimates presented in this work. To this end, introduce the notation $Q^{1}_t = \bigl( (Q^{1}_t)^1,\ldots,(Q^{1}_t)^k\bigr)$, $P^{1}_t = \bigl( (P^{1}_t)^1,\ldots,(P^{1}_t)^k\bigr)$ for the coordinates of $Q^{1
}_t \in\R^k$, $P^{1}_t\in\R^k$ (similarly for $(Q^2_t,P^2_t)\in \R^{2(d-k)}$). Observe that each pair $((Q^{1}_t)^i,(P^{1}_t)^i)$ in $\R^2$ for $i\in\{1,\ldots,k\}$  satisfies a simple $2d$-Langevin equation of form \eqref{eq:abc-fast}, where the different spatial-momenta coordinates only interact via the potential $V$ (and therefore only via $U$) (similarly for $(Q^2_t,P^2_t)$ in  \eqref{eq:abc-slow}). Throughout this article, we call $(Q^1,P^1)$ the constrained variables, and $(Q^2,P^2)$ the unconstrained variables.  

    \paragraph{General solution} Calling $X=((Q^{1}_t)^i,(P^{1}_t)^i)$  and applying variation-of-constant (see Appendix \ref{app:VarConst}), we arrive at the solution
\begin{align}\label{eq:intro-VarCons}
    X_t = e^{A t} X_0 + \int_0^t e^{A(t-s)} \begin{pmatrix} 0 \\ \nabla_i U(Q_s)\end{pmatrix}ds + \int_0^t e^{A(t-s)} C dB_s, 
\end{align}
where the linear part of the dynamics driven by $A\in \R^{2\times 2}$ and the noise matrix $C\in\R^{2\times 1}$  are of the form 
\begin{equation} \label{eq:intro-driftdiffmat}
A= \begin{pmatrix}
    0 & a \\ \alpha + b & c
\end{pmatrix}, \quad 
C = \begin{pmatrix}
    0 \\ \sqrt{2\gamma\beta^{-1}}
\end{pmatrix}.
\end{equation}
Here $B_t$ is a one-dimensional Brownian motion.
The same idea works for any $i$ as well as for $((Q^2)^i,(P^2)^i) \in \R^2$.
\paragraph{Constrained variables} For the constrained variables, the parameters $a,b,c\geq 0$ defining $A$ in \eqref{eq:intro-driftdiffmat} involve $\eps$ and therefore control the convergence. They are precisely defined  in Section~\ref{sec:limit-results} where the different limits are discussed. As stated above, Assumption~\ref{item:assVbound} is used to control the first integral term in~\eqref{eq:intro-VarCons}. 
A careful analysis of the explicit solution  \eqref{eq:intro-VarCons}, in particular the corresponding matrix exponentials, allows for proving convergence statements of the type 
\begin{align*}
    \E\Bigl[\bigl|X_t - \bar X_t\bigr|^2 \Bigr] \leq f(\eps,t), 
\end{align*}
where $t\geq0$ and the limit $\bar X_t$ could possibly be a constant. These results for $X_t=((Q^1)^i,(P^1)^i)$ are then combined to provide estimates for $(Q^1_t,P^1_t)$. 
\paragraph{Unconstrained variables} The bounds for the constrained variables, are then used to provide pathwise bounds for the unconstrained variables $(Q^2_t,P^2_t)$. These bounds are derived using the explicit solution \eqref{eq:intro-VarCons} together with a Gronwall-type argument. For the unconstrained variables the linear part of the drift, given by the matrix $A$ in \eqref{eq:intro-driftdiffmat}, is independent of $\eps$ and reads \[ A = \begin{pmatrix}
    0 & 1 \\ - \alpha  & -\gamma 
\end{pmatrix}.\]  Depending on $\alpha$ the overall bound will either grow exponential in the time horizon $T$ (case $\alpha=0$) or, if $\alpha > 0$ be either polynomial in $T$ or even independent of $T$ leading to a uniform in time estimate. For more details see Section~\ref{app:LimitProofs} and in particular Theorem \ref{cor:gronwall}. 
%

\subsubsection{Further remarks}\label{sec:genremarks}

The quantitative results also imply \emph{pointwise-in-time} estimates, \emph{convergence in probability} as well as the convergence of the corresponding time-marginal distributions in the \emph{Wasserstein-2 distance} (see Corollary \ref{cor:point-time-conv-prob}).
We should point out that several results for Langevin equation, for instance the overdamped limit of Langevin dynamics~\cite[Proposition 2.14]{LelievreRoussetStoltz10}, provide almost sure convergence on bounded time intervals.
In Appendix~\ref{app:AlmostSure} we show that quantitive estimates in almost-sure and $L^1$-pathwise sense also hold in the phase-space confinement limit. As our focus in this article is on proving quantitative estimates in (pathwise) $L^2$-sense and hence we do not pursue this route here. 

\paragraph{Role of $\alpha$ in \eqref{eq:Valpha}.} The choice of $\alpha$ considerably changes the $T$-scaling of the estimates of the unconstrained variables as stated above. On the other hand, the choice of $\alpha$ plays no role in the estimates of the constrained variables in the spatial and phase-space confinement case, see Propositions \ref{prop:SpatCons}, \ref{prop:PS-Cons}. This is due to the fact that in this case there is the quadratic confinement potential in $Q^1$ which scales like $\eps^{-1}$. For the physical limits the choice of $\alpha$ does make a difference, since in this case there is no quadratic potential  due to the constraining method.

\paragraph{Time horizon.} The final time horizon $T$ is assumed to be fixed. In particular in Section \ref{sec:PhysicalLimits} dealing with physical limits, when discussing competing terms in the error estimates, for instance $\eps T^2$ and $\eps T$ which can either decay or blow up as $\eps\to 0$ \emph{and} $T\to\infty$, we always present the lower-order terms in $\eps$, to simplify presentation. This will make the final estimates in the results considerably more readable. Readers should refer to the proofs for the precise scaling in $\eps$ and $T$.

\paragraph{Initial data.} In the presentation of the results we choose deterministic and $\eps$-independent initial conditions, to make the presentation as clear as possible. Nevertheless, all results hold for random and $\eps$-dependent initial datum and the general statements can be found in the proofs. Moreover, for the unconstrained variables, we choose the same initial datum for the pre-limit as well as the limit dynamics throughout the presentation, but all results hold for general initial datum.  See Sec.~\ref{sssec:Spat-InitDatum} for a discussion.

\section{Constraining via confinement}\label{sec:InfConf}

In this section we prove convergence results for the soft-constrained dynamics with spatial confinement~\eqref{eq:abc-slow},\eqref{eq:intro-Spat-Cons-Fast} and with phase-space confinement~\eqref{eq:abc-slow},\eqref{eq:intro-Phase-Cons-Fast}, which corresponds to the Hamiltonian~\eqref{eq:stuff-Ham-spat} and~\eqref{eq:stuff-Ham-phase} respectively. Throughout the section we work under the assumptions \ref{item:assVLip}-\ref{item:assVbound}. 

\subsection{Spatial confinement} \label{sec:spatial}
The following two results investigate the behaviour of the spatially confined Langevin dynamics~\eqref{eq:abc-slow},\eqref{eq:intro-Spat-Cons-Fast} as $\eps\to 0$. The first result deals with the behaviour of the constrained variables $(Q^1_t,P^1_t)$, which we expect to converge to zero due to the infinite confinement of $Q^1$, while the second result discusses the behaviour of the unconstrained variables $(Q^2_t,P^2_t)$, which we expect to converge to \eqref{eq:limit-gen} with $\hat q=0$. 
Making the connection to~\eqref{eq:abc-fast}, here we study the limit of equation \eqref{eq:abc-fast} with $b \to \infty$ and $a=c=1$.

\begin{prop}\label{prop:SpatCons}
Given $\eps>0$, let $(Q^1_t,P^1_t)$ for $i=1,2$, be the solution of~\eqref{eq:intro-Spat-Cons-Fast} with corresponding initial datum $(Q^1_0,P^1_0)$. 
\begin{enumerate}[label=(\roman*)]
\item\label{item:spat-Q1} 
For any $t\geq 0$
    \begin{align}
        \E\Bigl[\bigl|Q_t^{1}\bigr|^2\Bigr] \leq C_1 \eps +  C_2 e^{-\gamma t} \Bigl(\bigl|Q^1_0\bigr|^2 + \eps \bigl|P^1_0\bigr|^2\Bigr), 
        \label{eq:spat-point-bound}
    \end{align}
    where the constants $C_1,C_2>0$ are independent of $\eps$ and $t$.

\item\label{item:spat-P1} 
For any $T>0$ 
\begin{equation*}
    \lim\limits_{\eps \to 0 }  \int_0^T P^{1}_t \, dt = 0 \text{ \ in probability}.
\end{equation*}
\end{enumerate}
\end{prop}
    For proof see Appendix~\ref{proof:spat}.

The result shows that the constraint $Q^1=0$ is satisfied in the limit $\eps \to 0$ if $Q^1_0=0$, i.e.\ if the initial datum also satisfies the constraint. For the corresponding constrained momenta $P^1$ it is not possible to derive an estimate of the form \eqref{eq:spat-point-bound}, but only a time integrated form of convergence as given in \ref{item:spat-P1}, which holds for any choice of initial condition. This suggests that $P^1_t$ oscillates and these oscillations become faster as $\eps\to 0$ --  see the proof for precise details. In fact, one can make precise statements about the behaviour of $P^1$ both as $\eps\to 0$ and $t\to\infty$ in a simple linear setting -- this is discussed in Appendix~\ref{app:oscP1}. Note that $\alpha$ in the definition of $V$~\eqref{eq:Valpha} does not play any role in the estimates above (which is not the case in many of the forthcoming results) -- recall the observations from Section~\ref{sec:genremarks}. 

The next result concerns the convergence of the unconstrained variables $(Q^2,P^2)$. Note that in the result below the initial datum for the unconstrained (pre-limit, $\eps$-dependent) variables is the same as the limit. We make this choice here (and throughout this article) for the sake of notational simplicity and all our results for unconstrained variables generalise to general initial datum; see Sec.~\ref{sssec:Spat-InitDatum} below for details. Connections to related literature on spatial confinement  Remark~\ref{rem:SpatCompLit}.
\begin{prop}\label{prop:Behav-Q2P2}
    Given $\eps>0$, let $(Q_t,P_t)$ be the solution of the spatially confined dynamics~\eqref{eq:abc-slow},\eqref{eq:intro-Spat-Cons-Fast} with corresponding initial datum $(Q^i_0,P^i_0)$ for $i=1,2$. 
    Let $(q_t,p_t)\in \R^{2(d-k)}$ with the same initial data $(q_0,p_0)=(Q^2_0,P^2_0)$ evolve according to~\eqref{eq:limit-gen} with $\hat q=0 \in \R^k$. Recall the definition of the potential $V$ from~\eqref{eq:Valpha} and the parameter $\alpha\geq 0$. 
    We have  
\begin{align}\label{eq:Spat-fast}
  \mathbb{E} \biggl[\sup_{t\in[0,T]}  \biggl|\begin{pmatrix}
       q_t - Q^{2}_t \\ p_t-P^{2}_t  
    \end{pmatrix} \biggr|^2 \biggr] \leq
    \begin{cases}
    C_3 e^{C_4 T} \Bigl(
          T \bigl|Q_0^{1}\bigr|^2 + \eps T \bigl|P_0^{1}\bigr|^2 + \eps T^2 \Bigr) \ \  &\text{if $\alpha=0$}\\
     C_5 \left( \eps T + |Q^1_0|^2 + \eps |P^1_0|^2 \right) &\text{if $\alpha>0$}. 
    \end{cases}      
\end{align}
For $\alpha>0$ (i.e. $V$ has a quadratic part) we additionally have the uniform-in-time ($L^1$-)estimate 
\begin{align}\label{eq:Spat-fast-2}
\mathbb{E} \biggl[\sup_{t \geq 0}  \biggl|\begin{pmatrix}
       q_t - Q^{2}_t \\ p_t-P^{2}_t  
    \end{pmatrix} \biggr| \biggr] &\leq
    C_6  \Bigl(\sqrt{\eps} + \sqrt{\eps}  \bigl|P_0^{1}\bigr|  +\bigl|Q_0^{1}\bigr|\Bigr).
    \end{align}
The constants $C_3,C_4, C_5, C_6>0$ above are independent of $\eps$ and $T$.    
\end{prop}
    The proof can be found in Appendix~\ref{proof:spatslow}.

Note that $\alpha>0$ provides a \emph{surprising} uniform-in-time estimate~\eqref{eq:Spat-fast-2} for the unconstrained variables $(Q^2,P^2)$ -- see Remark~\ref{rem:Unif-InTime} below for details on its derivation. This directly implies convergence in probability in $C([0,\infty],\R^{2(d-k)})$ of $(Q^2,P^2) \to (q,p)$, see~\eqref{eq:conv-prob-spat} below. 

The following Corollary states that pointwise-in-time estimates as well as convergence in probability can be extracted from the results above. Similar pointwise results also hold throughout the article for every other asymptotic limit with proofs exactly as below. 

\begin{corollary}[Convergence in probability and pointwise in time bounds] \label{cor:point-time-conv-prob}
    Let $(Q_t,P_t)$ and $(q_t,p_t)$ be as above. For any $t \geq 0$ we have  
    \begin{align*}
        &\mathbb{E} \biggl[  \left|\begin{pmatrix}
       q_t - Q^{2}_t \\ p_t-P^{2}_t  
        \end{pmatrix} \right|^2 \biggr] \leq 
        \begin{cases}
        C_3 e^{C_4 t} \Bigl(
          t |Q_0^{1}|^2 + \eps t |P_0^{1}|^2 + \eps t^2 \Bigr) \ \  &\text{if $\alpha=0$}\\
          C_5  \left( \eps t + |Q^1_0|^2 + \eps |P^1_0|^2 \right) &\text{if $\alpha>0$},
          \end{cases}
\end{align*}
Additionally if $\alpha>0$, we have the time independent $L^1$-estimate
\begin{equation*}
    \mathbb{E} \biggl[ \biggl|\begin{pmatrix}
       q_t - Q^{2}_t \\ p_t-P^{2}_t  
    \end{pmatrix} \biggr| \biggr] \leq
    C_6  \Bigl(\sqrt{\eps} + \sqrt{\eps}  \bigl|P_0^{1}\bigr|  +\bigl|Q_0^{1}\bigr|\Bigr)
\end{equation*}
where $C_i >0, i\in\left\{1,\ldots,6 \right\}$ are independent of $t$ and $\eps$.
In particular, if $Q^1_0=0$, then for any $t\geq 0$
\begin{align*}
     Q^1_t \xrightarrow{\eps\to 0} 0 \text{ in probability}, 
\end{align*}
and in $C([0,\infty],\R^{2(d-k)})$ we have 
\begin{align}\label{eq:conv-prob-spat}
     (Q^2,P^2) \xrightarrow{\eps\to 0} (q,p) \text{ in probability}.  
\end{align}
 
\end{corollary}

\begin{proof}
It turns out that our pathwise results for the unconstrained variables in Proposition~\ref{prop:SpatCons} (and most other estimates for unconstrained variables in the rest of this article as well) are of the form 
\begin{align}
    \E \biggl[ \sup\limits_{t \in [0,T]} |X_t - \bar X_t|^2 \biggr] \leq f(\eps,T)\,, \label{eq:pathest-1}
\end{align}
where $X=(q-Q^2,p-P^2)^\top$ as in \eqref{eq:Spat-fast}-\eqref{eq:Spat-fast-2}.  
Then (\ref{eq:pathest-1}) immediately implies the pointwise-in-time estimate via
\begin{align*}
    \E \left[ |X_t - \bar X_t|^2 \right] \leq  \E \left[  \sup\limits_{r \in [0,t]} |X_r - \bar X_r|^2 \right]   \leq f(\eps,t)
\end{align*}
for any $t>0$.
For convergence in probability, note that by applying Jensen's inequality to the pointwise bound above we find 
$\E \left[ |X_t - \bar X_t| \right] \leq \sqrt{f(\eps,t)}\,$
and using Markov's inequality leads to 
\begin{align*}
    \Pr\bigl(|X_t - \bar X_t| > \eta \bigr) \leq \frac{1}{\eta} 
   \E \left[ |X_t - \bar X_t| \right] \leq \frac{\sqrt{f(\eps,t)} }{\eta}.
\end{align*}
for any $\eta>0$. 
In other words, if $f(\eps,t) \to 0$ as $\eps \to 0$, then $X_t$ converges to $\bar X_t$ in probability as $\eps \to 0$. For the pathwise convergence in $C([0,\infty],\R^{2(d-k)})$ we use the uniform-in-time estimate~\eqref{eq:Spat-fast-2} together with Markov's inequality. 
and the argument is the same as above.
Note that similar ideas as above also work for the $L^1$-estimate~\eqref{eq:Spat-fast-2}. 
\end{proof}

The following remarks discuss the uniform-in-time estimates for $\alpha>0$, related literature and the role of initial data. 

\begin{rem}[Uniform in time estimate]\label{rem:Unif-InTime}
 The estimates for the unconstrained variables $(Q^2,P^2)$ are derived  using Gronwall's inequality together with the quantitative results obtained for the constrained variable $Q^1_t$. Consequently, as is typical when using  Gronwall's inequality, the estimate involves a constant $C=C(T)$ which is typically exponential, where $T$ is the final time of the considered path. Nevertheless, if the potential energy function $V$ is quadratic in $Q^2$, i.e. $\alpha>0$ (recall the assumptions on the potential in Section~\ref{sec:Ass-Pot}),  we are able to prove the  $L^1$-bound~\eqref{eq:Spat-fast-2} which is uniform-in-time. This is remarkable, but in line with the limit invariant measure (see Proposition \ref{prop:invmeas-spat}), 
 which also requires $V$ growing at least quadratically at infinity. This is due to the fact that the $L^1$ bound allows us to directly estimate the integrals, while the $L^2$-estimate~\eqref{eq:Spat-fast} requires using the Cauchy-Schwarz inequality, which introduces the dependence on the considered time interval, see Theorem~\ref{cor:gronwall} and its proof for details. Note that by Jensen's inequality the $L^2$-bound implies convergence in $L^1$ sense with a similar bound. 

 Nevertheless, there are a few caveats for this surprising uniform-in-time estimate. First, we were only able to prove a uniform-in-time $L^1$-estimate and not a uniform-in-time $L^2$-estimate when $\alpha>0$.  

 Second, this  bound works when dealing with linear confinement forces (i.e. $\xi$ introduced in~\eqref{eq:stuff-Ham-spat} is linear), which leads to same additive noise in the pre-limit unconstrained variables $(Q^2,P^2)$ and the limit $(q_t,p_t)$. Consequently, the noise terms cancel when dealing with the difference of these two, which allows us to prove an $L^1$-estimate. However, if we were working with more complex settings (such as nonlinear confinement), then the noise terms would not cancel, and we not expect an $L^1$-estimate to work. A more detailed study of such cases is left to future work.
\end{rem}

\begin{rem}[Comparison to related literature]\label{rem:SpatCompLit} 
Such systems have been studied for \emph{nonlinear} constraints in the Hamiltonian setting \cite{RubinUngar57,bornemann1997homogenization,froese2001realizing,takens2006motion,klar2021second}, and the full Langevin setting \cite{Reich00} using formal asymptotic expansions (often) based on transformations of the system to action-angle variables and an analysis of adiabatic invariants. A specific feature of the deterministic case is that the resulting correction potential depends on the initial conditions via the  action variables that are conserved quantities under the highly oscillatory dynamics. While these works provide a characterisation of the constrained limit dynamics, they do not provide quantitative convergence rates.   
\end{rem}

\subsubsection{Initial conditions and unbounded energy}\label{sssec:Spat-InitDatum}

If $Q^1_0=0$ then, as expected, $Q^1_t\to 0$ as $\eps\to 0$. Nevertheless, by Proposition~\ref{prop:SpatCons}, the constrained momentum $P^1_t$ fluctuates for any choice of initial condition; see also the example in Appendix~\ref{app:oscP1}.

This is in line with similar findings for deterministic Hamiltonian systems, e.g. \cite{bornemann1997homogenization,froese2001realizing}: Both  $Q^1_t$ and $P^1_t$ oscillate with increasing frequency if $Q^1_0 \neq 0$ as can be seen from \eqref{eq:Q1-sol}--\eqref{eq:P1-sol}. The position variable $Q^1_t$ has decreasing amplitude as $t \to \infty$, while the amplitude of $P^1_t$ is non-decreasing for any choice of initial data. For deterministic systems, under our ``flatness'' assumption on the constraint manifold $\xi^{-1}(0)$, the spatial confinement realises the constrained motion $(Q^2_t,P^2_t)$ if and only if the initial energy 
\[
H(Q_0,P_0)= V(Q_0) + \frac{1}{2\eps} |Q^1_0|^2 + \frac12 |P_0|^2 
\]
remains bounded as $\eps\to 0$; see \cite[Thm.~2.1]{froese2001realizing} and \cite[Ch.~II.3.1]{bornemann1998homogenization}; this condition that is also enforced in the work \cite{Reich00} on highly oscillatory Langevin dynamics, requires $Q^1_0=\bigO(\sqrt{\eps})$. In general, for diverging initial energy, the limit motion of a deterministic Hamiltonian system ceases to satisfy the constraint as has been demonstrated in \cite[Ch.~II.1.9]{bornemann1997homogenization}. 

In our case, under the assumptions on friction and noise coefficients (specifically: the fluctuation-dissipation relation), we only need that $Q^1_0 \to 0$ as $\eps \to 0$; the analysis even reveals that the convergence holds for initial momentum $P^1_0= \littleO(\eps^{-1/2})$. As a consequence, constrained motion with ``flat'' coordinate projection constraints can be realised even for infinite initial energy. 

We stress that, even though the aforementioned and all subsequent results are formulated for deterministic initial conditions, our results straightforwardly generalise to random $\eps$-dependent initial conditions. In this case $|Q^1_0|^2, |P^1_0|^2$ in the above results will be replaced by $\E[|Q^{1,\eps}_0|^2],\E[|P^{1,\eps}_0|^2]$ respectively. The result in Proposition \ref{prop:SpatCons} \ref{item:spat-P1}  also holds as long as  $Q^{1,\eps}_0 = \littleO(\eps^{-1})$  and $P^{1,\eps}_0 = \littleO(\eps^{-\frac12})$ in probability. Moreover, if we allow $(q_0,p_0) \neq (Q^2_0,P^2_0) := \lim\limits_{\eps \to 0} (Q^{2,\eps}_0,P^{2,\eps}_0)$ in Proposition \ref{prop:Behav-Q2P2} then \eqref{eq:Spat-fast} and \eqref{eq:Spat-fast-2} respectively read
\begin{align*}
    \mathbb{E} \biggl[\sup_{t\in[0,T]}  \biggl|\begin{pmatrix}
       q_t - Q^{2}_t \\ p_t-P^{2}_t  
    \end{pmatrix} \biggr|^2 \biggr] \leq
    C_3 e^{C_4 T}\left\{
        \E\left[\left|
        \begin{pmatrix}
           q_0 - Q^{2,\eps}_0 \\ p_0 - P^{2,\eps}_0
        \end{pmatrix} \right|^2\right] + 
          T \E\left[\bigl|Q_0^{1,\eps}\bigr|^2\right] + \eps T \E\left[\bigl|P_0^{1,\eps}\bigr|^2 \right]+ \eps T^2 \right\} 
    \end{align*}
    for $\alpha=0$, and 
    \begin{align*}
    \mathbb{E} \biggl[\sup_{t\in[0,T]}  \biggl|\begin{pmatrix}
       q_t - Q^{2}_t \\ p_t-P^{2}_t  
    \end{pmatrix} \biggr|^2 \biggr] \leq
     e^{- \eta r}\E\left[\left|
        \begin{pmatrix}
           q_0 - Q^{2,\eps}_0 \\ p_0 - P^{2,\eps}_0
        \end{pmatrix} \right|^2\right] +
    C_5  \left(T \eps + \eps  \E\left[\bigl|P_0^{1,\eps}\bigr|^2\right]  +
           \E\left[\bigl|Q_0^{1,\eps}\bigr|^2\right]\right)
    \end{align*}
    for $\alpha>0$. Additionally, in the latter case, we have the uniform in time   $L^1$ bound 
    \begin{align*}
        &\mathbb{E} \biggl[\sup_{t \geq r}  \biggl|\begin{pmatrix}
       q_t - Q^{2}_t \\ p_t-P^{2}_t  
    \end{pmatrix} \biggr| \biggr] \leq e^{- \mfrac{\eta}{2}r}\E\left[\left|
        \begin{pmatrix}
           q_0 - Q^{2,\eps}_0 \\ p_0 - P^{2,\eps}_0
        \end{pmatrix} \right|\right] +
    C_6  \left(\sqrt{\eps} + \sqrt{\eps}  \E\left[\bigl|P_0^{1}\bigr|\right]  +
           \E\left[\bigl|Q_0^{1}\bigr|\right]\right),
\end{align*}
where $\eta \coloneqq  \left(\gamma - \Re(\sqrt{\gamma^2-4\alpha})\right) \in (0,\gamma]$ and $C_i >0, i \in \left\{1,\ldots,6 \right\}$ are independent of $\eps, T.$

\subsubsection{Steady state under spatial confinement}

As stated in the introduction, constraints are often used to sample probability measures, so we will now discuss the steady state of the limiting dynamics under spatial confinement. (We use $\mathcal P(\R^{m})$ to denote probability measures on $\R^m$.) 
    The spatially confined Langevin dynamics~\eqref{eq:abc-slow},\eqref{eq:intro-Spat-Cons-Fast} admits the steady state
    \begin{align}\label{eq:SpatPreLimSte}
     d\mu^\eps(q,p) \coloneqq \frac{1}{Z^\eps} \exp\left( - \beta\left(V(q) + \frac{1}{2\eps}|q^1|^2 + \frac{1}{2}|p|^2\right)\right)dqdp
    \end{align}
where $Z^\eps$ is the normalisation constant which ensures that $\mu^\eps\in \mathcal P(\R^{2d})$. 
Define $\xi:\R^d\to \R^k$ as $\xi(q)=q^1$. The probability measure $\mu^\eps\in \mathcal P(\R^{2d})$ converges weakly to $\mu\in \mathcal P(\xi^{-1}(0)\times\R^d)$ given by 
\begin{align*}
    d\mu(q^2,p)\coloneqq \frac{1}{Z} \exp\left( - \beta\left(V(0,q^2) +\frac{1}{2}|p|^2\right)\right) dq^2 dp
\end{align*}
where $Z$ is the normalisation constant.  
In particular, for $f\in C_b(\R^d)$ with $f=f(q)$ we have
\begin{align*}
    \lim_{\eps\to 0} \int_{\R^{2d}}f(q)d\mu^\eps(q,p) = \int_{\R^{d-k}} f(0,q^2) \frac{1}{\hat Z} \exp\left(-\beta V(0,q^2)\right)dq^2,
\end{align*}
with $\hat Z\coloneqq \int
\exp(-\beta V(0,q^2))dq^2$ as detailed in Proposition~\ref{prop:invmeas-spat} in Appendix~\ref{sec:Steady}. 

While the spatially confined dynamics has the correct constrained position marginal, the oscillatory momentum component $P^1_t$ that, in contrast to the damped spatial component $Q^1_t$, is not decaying makes the dynamics numerically stiff and hence difficult to simulate. 

\begin{rem}
    The fact that $p^1$ is not constrained implies that, if $\xi(q)$ is nonlinear in $q$, such that $\nabla\xi$ is varying when restricted to the constraint manifold $\xi^{-1}(0)$, then marginalisation over the momenta will lead to an extra correction potential $-(2\beta^{-1})\log \det(\nabla\xi^\top\nabla\xi)$ in addition to the Fixman potential that is the result of spatial confinement; see \cite[Sec.~3.4]{Hartmann2007} or \cite[Rem 3.3]{LelievreRoussetStoltz12}. The pathwise analysis would be beyond the scope of this paper, but we briefly mention that the correction potential due to the oscillatory momenta will not be there in the phase-space confinement case that is discussed next for coordinate projections.  
\end{rem}


\subsection{Phase-space confinement} \label{sec:phase-space}
The following result investigates the behaviour of the phase-space confined Langevin dynamics~\eqref{eq:abc-slow},\eqref{eq:intro-Phase-Cons-Fast} as $\eps\to 0$. The first part of this result deals with the behaviour of the constrained variables $(Q^1_t,P^1_t)$, which we both expect to converge to $0$ due to confinement in both these variables. The second part of this result discusses the behaviour of $(Q^2_t,P^2_t)$, which converge to~\eqref{eq:limit-gen} as in the spatial confinement setting. Making the connection to~\eqref{eq:abc-fast}, here we study the limit of equation \eqref{eq:abc-fast} with $a,b,c\to\infty$.

\begin{prop}[Phase-space constrained limits]\label{prop:PS-Cons}
    Given $\eps>0$, let $(Q_t,P_t)$ be the solution to~\eqref{eq:abc-slow},\eqref{eq:intro-Phase-Cons-Fast} with corresponding initial datum $(Q^i_0,P^i_0)$ for $i=1,2$.
    \begin{enumerate}[label=(\roman*)]
        \item\label{item:phase-Q1} 
        For any $t\geq 0$
        \begin{align}\label{eq:PhasePathBound}
             \max\biggl\{\E\Bigl[|Q^{1}_t|^2\Bigr] , \E\Bigl[ |P^{1}_t|^2\Bigr]\biggr\} &\leq C_1 \left(\eps +  e^{- \frac{t}{\eps} C_2} \Bigl(   \bigl|Q^{1}_0\bigr|^2+ \bigl|P^{1}_0\bigr|^2  \Bigr)\right), 
        \end{align}
        where $C_1, C_2>0$ are independent of $\eps$ and $r$.

        \item\label{item:phase-Q2P2} Let $(q_t,p_t)\in \R^{2(d-k)}$ with initial data $(q_0,p_0)=(Q^2_0,P^2_0)$ evolve according to~\eqref{eq:limit-gen} with $\hat q=0 \in \R^k$. Recall the definition of the potential $V$ from~\eqref{eq:Valpha} and the parameter $\alpha\geq 0$. In general (i.e. for any $\alpha\geq 0$) we have 
        \begin{align}
            \label{eq:Phase-fast}
            \mathbb{E} \biggl[ \sup_{t\in [0,T]}\left|\begin{pmatrix}
            q_t - Q^{2}_t \\ p_t-P^{2}_t  
            \end{pmatrix} \right|^2 \biggr] \leq 
            \begin{cases}
            C_3 e^{C_4  T} \Bigl( \eps T^2 + \eps T \Bigl( |Q_0^{1}|^2 +  |P_0^{1}|^2 \Bigr)\Bigr) \ \ &\text{if $\alpha=0$,}\\
            \eps C_5 \left( T + |Q^1_0|^2 + |P^1_0|^2 \right) &\text{if $\alpha>0$}.
            \end{cases}
        \end{align}
        In the case when $\alpha>0$ (i.e. $V$ has a quadratic part) and $\eps$ small enough we additionally have the uniform-in-time $L^1$- estimate
        \begin{align}\label{eq:phase-fast-L1}
            \mathbb{E} \biggl[ \sup_{t \geq r}\left|\begin{pmatrix}
            q_t - Q^{2}_t \\ p_t-P^{2}_t  
            \end{pmatrix} \right| \biggr] \leq \sqrt{\eps} C_6  + 
            \eps C_7 e^{- C_8 r}\Bigl( |Q_0^{1}| +  |P_0^{1}| \Bigr)       
            \end{align}
        Here 
        the constants $C_i>0$ for $i\in \{3,\ldots,8\}$ are independent of $\eps>0$ and $T>0$.
    \end{enumerate}
\end{prop}
    For proof see Appendix~\ref{App:PS-cons-proof}.    

Note that, as in the spatial-confinement setting, we arrive at the uniform-in-time estimate for $\alpha>0$, with an improved prefactor; for details see Remark~\ref{rem:Unif-InTime}.  Furthermore, the phase-space confined Langevin converges for any choice of initial conditions, which is in stark contrast to the spatial-confinement setting discussed in the previous discussion. We discuss the initial conditions in Sec.~\ref{sssec:PS-Init} below. 

As in Corollary~\ref{cor:point-time-conv-prob}, pointwise-in-time estimates also follow in this setup. In particular, following the same proof as Corollary~\ref{cor:point-time-conv-prob}, it follows that for any $t>0$ we have
\begin{align*}
    Q^1_t\to 0, \ P^1_t\to 0, \text{ in probability as } \eps\to 0,
\end{align*}
for any choice of initial data. 
Moreover,
\begin{align}\label{eq:conv-prob-phasespace}
     (Q^2_t,P^2_t) \to (q_t,p_t)  \text{ in probability as } \eps\to 0 \text{ in } C([0,\infty],\R^{2(d-k)}) \text{ for any } Q^1_0,P^1_0 \in \R^k.
\end{align}
This result can be directly compared with the result of~\cite{Katzenberger91} and we refer to Section~\ref{sec:Disc-GenForm} for more details. \\
Following the discussion in Sec.~\ref{sssec:Spat-InitDatum}, our results in Propositions~\ref{prop:SpatCons} and~\ref{prop:Behav-Q2P2} can be readily generalised to the setting of $\eps$-dependent and random initial data, the same goes for initial conditions $(q_0,p_0) \neq (Q^2_0,P^2_0)$ that do not lie on the constraint subspace.

Next, we compare the pathwise and pointwise estimates discussed above as well as the assumptions on the initial conditions in the spatial and phase-space constrained setting.  

\subsubsection{Initial conditions and comparison with spatial confinement}\label{sssec:PS-Init}

    Estimate~\eqref{eq:PhasePathBound} states that the constrained variables $(Q^1_t,P^1_t)$ converge to zero for any positive time, regardless of the choice of initial condition $(Q^1_0,P^1_0)$. This is in stark contrast to the spatial confinement estimates, which require a well-prepared initial datum $Q^1_0=0$  
    and in which case $P^1_t$ oscillates with non-decreasing amplitude. 
The behaviour of the unconstrained variables $(Q^2_t,P^2_t)$ does not depend on the initialisation of the constrained variables $(Q^1_0,P^1_0)$ as becomes clear in the corresponding estimate~\eqref{eq:Phase-fast}--\eqref{eq:phase-fast-L1}.

Allowing for $\eps-$dependent initial data, the initial energy  
\[
H(Q_0,P_0)= V(Q_0) + \frac{1}{2\eps} |Q^1_0|^2 + \frac12\left(1+ \frac{1}{\eps}\right) |P^1_0|^2    + \frac12 |P^2_0|^2
\]
can explode if $P^1_0 = \littleO(\eps^{-\frac12})$ or $Q^1_0 = \littleO(\eps^{-\frac12})$ while all results remain valid. 
Before we conclude this subsection, we briefly discuss the steady state of the limiting dynamics under phase-space confinement.

\subsubsection{Steady state under phase space confinement}

    The phase-space confined Langevin dynamics~\eqref{eq:abc-slow},\eqref{eq:intro-Phase-Cons-Fast} 
     admits the steady state
\begin{align*}
    d\mu^\eps(q,p) = \frac{1}{Z^\eps} \exp\left( - \beta\left(V(q) + \frac{1}{2\eps}|q^1|^2 + \frac{1}{2}|p|^2 + \frac{1}{2\eps}|p^1|^2\right)\right)dqdp
\end{align*}
where $Z^\eps$ is the normalisation constant which ensures that $\mu^\eps\in \mathcal P(\R^{2d})$. Define $\Xi:\R^{2d}\to\R^{2k}$ as $\Xi(q,p)=(q^1,p^1)$. 
Following the proof of Proposition~\ref{prop:invmeas-spat} which discusses the steady state for the spatially confined case, it can be shown that $\mu^\eps\in \mathcal P(\R^{2d})$ converges weakly to $\mu\in \mathcal P(\Xi^{-1}(0))$ given by 
\begin{align*}
    d\mu(q^2,p^2)\coloneqq \frac{1}{Z} \exp\left( - \beta\left(V(0,q^2) +\frac{1}{2}|p^2|^2\right)\right)dq^2dp^2
\end{align*}
where $Z$ is the limiting normalisation constant. 
In particular, for $f\in C_b(\R^d)$ with $f=f(q)$ we have
\begin{align*}
    \lim_{\eps\to 0} \int_{\R^{2d}}f(q)d\mu^\eps(q,p) = \int_{\R^{d-k}} f(0,q^2) \frac{1}{\hat Z} \exp(-\beta V(0,q^2))dq^2,
\end{align*}
with $\hat Z\coloneqq \int
\exp(-\beta V(0,q^2))dq^2$.
Furthermore, the limit~\eqref{eq:limit-gen} with $\hat q=0$ of the unconstrained variables $(Q^2,P^2)$  as described in Proposition~\ref{prop:PS-Cons} admits $\mu$ as a steady state.

This should be contrasted with the steady state of the position confined system: For the case of a coordinate projection constraint the $q$-marginal of the steady states are the same as for the phase-space confined dynamics. Nevertheless, the oscillatory momentum component $P^1_t$ is dissipative, which is advantageous in terms of numerical discretisation \cite{kane2000variational}. 

\begin{rem}
    The stabilisation of the constraint subspace by penalising higher order (also: hidden) constraints is a well-known phenomenon for deterministic Hamiltonian systems with Lagrange multipliers (e.g. \cite{barth1995algorithms}) or impetus-striction (e.g. \cite{gonzalez2001multi}); see also \cite{kunkel2006differential} and the references therein. 
\end{rem}

\section{Constraining via (physical) parameter limits}\label{sec:PhysicalLimits}

In this section we prove quantitative convergence results for physical parameters in four settings outlined in the introduction: zero mass, infinite mass, and infinite friction with/without fluctuation dissipation. Throughout this section we assume that the end-time $T>1$ to avoid lower order terms and simplify the final form of the estimates. The complete form of the estimates for any $T>0$ is available in the proofs. 
Furthermore, to simplify presentation, when discussing competing terms, for instance $\eps T^2$ and $\eps T$ which scale differently as $\eps\to 0$ and $T\to\infty$, we always present the lower-order terms in $\eps$. This will make the final estimates in the results considerably more readable. Interested readers can refer to the proofs for precise scaling in $\eps$ and $T$.

\subsection{Zero mass limit} \label{sec:zeromass}

We now discuss the $\eps\to 0$ limit of the zero mass setting~\eqref{eq:abc-slow},\eqref{eq:intro-0Mass-Fast}. As discussed in the introduction, sending mass of the first $k$-particles to zero corresponds to the classical overdamped limit, wherein the corresponding position converges to the overdamped Langevin dynamics while the momenta converges to zero (as the first $k$ particles lose momentum due to vanishing mass). Making the connection to~\eqref{eq:abc-fast}, here we study the limit of~\eqref{eq:abc-fast} with $a=\eps^{-1}$, $b= 0$ and $c=\eps^{-1}$.

\begin{prop}\label{prop:Light}
    Given $\eps>0$, let $(Q_t,P_t)$ be the solution to~\eqref{eq:abc-slow},\eqref{eq:intro-0Mass-Fast}, with corresponding initial data $(Q^i_0,P^i_0)$ for $i=1,2$. Recall the definition of the potential $V$ from~\eqref{eq:Valpha} and the parameter $\alpha\geq 0$. 
    \begin{enumerate}[label=(\roman*)]
        \item\label{item:Light-P1} 
        For any $t\geq 0$ 
        \begin{align}\label{eq:Light-PathBound}
            \E\Bigl[|P_t^{1}|^2\Bigr] \leq 
            \begin{dcases} C_1 \bigl(\eps +  e^{-\frac{C_2}{\eps}t} |P_0^{1}|^2\bigr) \  &\text{if $\alpha=0$}, \\
            C_3 \left(\eps +\eps |Q^1_0|^2  + (e^{-\frac{C_4}{\eps}t} + \eps)  |P^1_0|^2    \right) &\text{if $\alpha>0$},
             \end{dcases}
        \end{align}
        where the constants $C_1, C_2, C_3, C_4>0$ are independent of $\eps$ and $t$.

        \item\label{item:LightSlow} Let $(\hat q_t,q_t,p_t)\in \R^{k}\times \R^{d-k} \times \R^{d-k}$ with initial data $(\hat q_0,q_0,p_0)=(Q^1_0,Q^2_0,P^2_0)$  evolve according to~\eqref{eq:0Mass-q1} and~\eqref{eq:limit-gen}. 
        We have for any $t\leq 0$
        \begin{align}\label{eq:SmallMass-Fast}
        \E\biggl[\Bigl|(Q^{1}_t,Q^{2}_t,P^{2}_t)^\top - (\hat q_t,q_t,p_t)^\top\Bigr|^2\biggr] \leq \begin{dcases}C_5 e^{C_6 t^2}\Bigl(\eps + 
        \bigl|P_0^{1}\bigr|^2  \Bigr) \ &\text{if $\alpha=0$}, \\
        C_7 e^{C_8 t^2}\Bigl(\eps t^2+\eps^2 t^2 \bigl|Q_0^{1}\bigr|^2 + |P^1_0|^2\Bigr)  &\text{if $\alpha>0$},
        \end{dcases}
        \end{align}
        where the constants $C_5,C_6, C_7,C_8>0$ are independent of $\eps>0$ and $t>0$.
    \end{enumerate}
\end{prop}
    For proof see Appendix~\ref{app:PartCons-Limit-Proof}. Note that the zero-mass limit is the only constraining mechanism in this article  for which the quantitative estimate~\eqref{eq:SmallMass-Fast} is pointwise in time and not pathwise (i.e. missing $\sup_{t\in [0,T]}$ inside the expectation). This follows since in this scaling, the diffusion coefficient for the unconstrained pre-limit variables $(Q^1,Q^2,P^2)$ does not coincide with the diffusion coefficient for the limiting variables $(\hat q_t, q_t,p_t)$, and therefore a pathwise estimate in this case leads to exponential growth as $\eps\to 0$ following the techniques used throughout this article. Details can be found at the end of the proof for Proposition~\ref{prop:Light}. 

Following the proof of Corollary~\ref{cor:point-time-conv-prob} and Sec.~\ref{sssec:Spat-InitDatum}, we also have pointwise-in-time estimates, such as, for any $t>0$ and initial datum $P^1_0=0$
\begin{align*}
    P^1_t\to 0, \ (Q^1_t,Q^2_t,P^2_t) \to (\hat q_t,q_t,p_t) \ \ \text{in probability as \ } \eps\to 0,
\end{align*}
and, furthermore our results generalise to $\eps$-dependent, possibly exploding and random initial datum.

\subsubsection{Initial conditions}
    
Unlike other estimates on the slow variables (see for instance~\eqref{eq:Phase-fast}), the error estimate~\eqref{eq:SmallMass-Fast} for the unconstrained variables in this limit requires that the initial datum for the constrained momentum $P^{1}_0=0$ (or  converges to zero if it is $\eps$-dependent). This is in accordance with the usual small-mass, see e.g.~\cite[Proposition~2.14]{LelievreRoussetStoltz10} which is formulated for the dimensionless variable $v =\frac{P}{\eps}$. Therefore, constant initial condition $v_0$ refers to $P_0= \eps v_0$ in our case, which then yields the corresponding convergence.


\subsubsection{Steady state in the zero mass limit}

    In this case the pre-limit dynamics~\eqref{eq:abc-slow},\eqref{eq:intro-0Mass-Fast} admits the steady state
\begin{align*}
    d\mu^\eps(q,p) = \frac{1}{Z^\eps} \exp\left( - \beta\left(V(q) + \frac{1}{2}|p|^2 + \frac{1}{2\eps}|p^1|^2\right)\right)dqdp
\end{align*}
where $Z^\eps$ is the normalisation constant which ensures that $\mu^\eps\in \mathcal P(\R^{2d})$. 
Using $\eta:\R^d\to \R^k$ defined by $\eta(p^1,p^2)=p^1$, the probability measure $\mu^\eps\in \mathcal P(\R^{2d})$ converges weakly to $\mu_{3}\in \mathcal P(\R^d\times \eta^{-1}(0))$, given by 
\begin{align*}
    d\mu(q,p^2)\coloneqq \frac{1}{Z} \exp\left( - \beta\left(V(q) + \frac{1}{2}|p^2|^2\right)\right)dqdp^2
\end{align*}
where $Z$ is the normalisation constant. Note that $\mu$ is the steady state for the limit $(\hat q,q^2,p^2)$. 
In particular, for $f\in C_b(\R^d)$ with $f=f(q)$ we have
\begin{align*}
    \lim_{\eps\to 0} \int_{\R^{2d}}f(q)d\mu^\eps(q,p) = \int_{\R^{d}} f(q) \frac{1}{\bar Z} \exp(-\beta V(q))dq,
\end{align*}
with $\bar Z\coloneqq \int
\exp(-\beta V(q))dq$. 

\begin{rem}
    The paper \cite{LelievreRoussetStoltz12} discusses a Dirac bracket formulation of constrained Langevin dynamics for a nonlinear map $\xi$, for which a splitting method in combination with a Metropolisation step is introduced that samples the properly constrained steady state, which is different from the one obtained by us in the zero mass limit. This is, however, not a contradiction: By Proposition \ref{prop:Light}, the constrained position variable $Q^1_t$ follows an overdamped motion, and, formally, the Dirac formulation can be obtained if we suppose that $\nabla_{q^1}V(q)=0$ and send $\beta\to\infty$ in the equation for $Q^1_t$. The decoupling of the motion for $Q^1_t$ and for $Q^2_t$ in the zero mass limit is akin to temperature accelerated dynamics \cite{maragliano2006temperature} or the temperature-separated Langevin dynamics \cite{breiten2021stochastic}, even though the physical limit here is a different one. 
\end{rem}

\subsection{Infinite mass limit} \label{sec:highmass}

We now study the $\eps\to 0$ limit of the partial-infinite mass~\eqref{eq:abc-slow},\eqref{eq:intro-InfMass}, which corresponds to a Langevin dynamics with Hamiltonian of the form~\eqref{eq:Ham-InfMass}, which physically corresponds to the setting where the first $k$-particles with position and momentum $(Q^1,P^1)$ are $\mathcal O(\eps^{-1})$ times heavier than the remaining $d-k$ particles described by $(Q^2,P^2)$. 
Making the connection to~\eqref{eq:abc-slow}, here we study the limit of~\eqref{eq:abc-slow} with $a=\eps$, $b=0$ and $c=\eps$. In the following result, we discuss the limit $\eps\to 0$ in this setting. 

As expected $Q^1_t$ converges to the initial datum $Q^1_0$, but the limit dynamics for $P^1_t$ as $\eps\to 0$ is given by~\eqref{eq:InfMass-p1}. We point out the overall limiting dynamics $(Q^1_t=Q^1_0,p^1_t,q^2_t,p^2_t)$ is a Langevin dynamics~\eqref{eq:intro-genLang} with the \emph{impetus-striction Hamiltonian} $\mathcal H(q,p)=\lim_{\eps\to 0} H(q,p)$, where the pre-limit Hamiltonian $H$ is defined in~\eqref{eq:Ham-InfMass}. In particular, this Hamiltonian introduced in~\cite[Eq.~(19)]{WalterHartmannMaddocks11}, satisfies
\begin{equation*}
    \mathcal H(q,p)\coloneqq H(q,\mathcal{P}p) = V(q)+\frac12 |\mathcal Pp|^2,
\end{equation*}
where $\mathcal P=I_{d\times d} - \nabla\xi^\top(\nabla\xi \nabla\xi^\top)^{-1}\nabla\xi$ for general maps $\xi$ (see introduction). The definition $p_\xi=\mathcal{P}p$ is used in Vakonomic mechanics \cite[Ch.~1.6.4]{arnoldBook} and the impetus-striction formulation mechanics \cite{maddocks1995unconstrained} to define the constrained momenta.  
In the setting of this article $\xi(q)=q^1$ and hence $\mathcal Pp=p^2$. For detailed discussion of Langevin dynamics in impetus-striction form, see~\cite{WalterHartmannMaddocks11} and references therein.

\begin{prop}\label{prop:heavy}
    Given $\eps>0$, let $(Q_t,P_t)$ be the solution to~\eqref{eq:abc-slow},\eqref{eq:intro-InfMass} with corresponding initial data $(Q^i_0,P^i_0)$ for $i=1,2$. Recall the definition of the potential $V$ from~\eqref{eq:Valpha} for some parameter value $\alpha\geq 0$. Furthermore, let $T\geq 1$ be fixed. 
    \begin{enumerate}[label=(\roman*)]
        \item\label{item:heavy-Q1}
       We have for any$ t\geq 0$
       \begin{align}
           \E\Bigl[\bigl|Q_t^{1}-Q^{1}_0\bigr|^2\Bigr]
            \leq \begin{dcases}  \eps^2 C_1 t^2 \Bigl( t^2+|P_0^{1}|^2\Bigr) \ \  &\text{if $\alpha=0$}, \\
             \eps C_2 \Bigl( t+ \eps t^4 |Q^1_0| +  |P^1_0|^2  \Bigr) &\text{if $\alpha>0$},
            \end{dcases}
        \end{align}
     where the constants $C_1,C_2>0$, are independent of $\eps$ and $t$.   
        \item\label{item:heavySlow} Let $(q_t,p_t)\in \R^{2(d-k)}$ with initial data $(q_0,p_0)=(Q_0^2,P_0^2)$ evolve according to~\eqref{eq:limit-gen} with $\hat q=Q^1_0 \in \R^k$. 
         Then
        \begin{align}\label{eq:heavySlow} &\E\biggl[\sup_{t\in [0,T]}\biggl|\begin{pmatrix}q_t - Q^{2}_t \\ p_t - P^{2}_t\end{pmatrix}\biggr|^2\biggr] \leq \begin{dcases}
         \eps^2 C_4 T^4    e^{C_5 T}\Bigl(T^2+|P_0^{1}|^2  \Bigr) & \text{if } \alpha=0\\
            \eps C_6  T  \left(T+ \eps T^4 |Q^1_0|^2 +  |P^1_0|^2 \right) & \text{if } \alpha>0
            \end{dcases}
        \end{align}
        where the constants  $C_4, C_5, C_6>0$ are independent of $\eps$ and $T$. 
        \item \label{item:heavyP1} Let $p^1_t\in\R^k$ with initial condition $p^1_0=P^1_0$ evolve according to~\eqref{eq:InfMass-p1}. We have  for any $t\geq 1$
        \begin{align}
        \E\biggl[\bigl|p^{1}_t - P_t^{1}\bigr|^2\biggr] \leq  \begin{dcases}
\eps C_7  t^2 \left( e^{C_8 t} \left(t^2+\eps t^2 |P^1_0|^2  \right) + |P^1_0|^2 \ \right) & \text{ if } \alpha=0 \\
C_9 \left( (t^2 + \eps^2  t^7)  |Q_0^1|^2 + \eps^2 t^4  |P^1_0|^2 +   \eps t^4\right) &\text{ if } \alpha>0,
        \end{dcases}
        \end{align}
        where $C_7,C_8,C_9>0$ are independent of $\eps$ and $t$.
    \end{enumerate}
\end{prop}
\begin{proof}
    The proof can be found in Appendix~\ref{proof:impetus}.
\end{proof}
Finally, as in Corollary~\ref{cor:point-time-conv-prob} and Sec.~\ref{sssec:Spat-InitDatum}, we also have pointwise-in-time estimates, for instance, for any $t>0$ and initial datum 
\begin{align*}
   Q^1_t\to 0, \ (Q^2_t,P^2_t) \to (q_t,p_t) \ \ \text{in probability as \ } \eps\to 0.
\end{align*}
If $Q^1_0=0$ or $\alpha =0$ we also have 
\begin{align*}
    P^1_t \to p^1_t,  \ \ \text{in probability as \ } \eps\to 0,
\end{align*}
and, furthermore our results generalise to $\eps$-dependent, possibly exploding and random initial datum. 

\subsubsection{Steady state in the infinite mass limit}\label{ssec:Infmass-Steady}

We briefly discuss the steady state of the infinite mass limit. Recalling~\eqref{eq:Ham-InfMass}, it follows that the pre-limit dynamics~\eqref{eq:abc-slow},\eqref{eq:intro-InfMass} for $\eps>0$ admits the steady state
\begin{align*}
     d\mu^\eps(q,p)= \frac{1}{Z^\eps} \exp\left( - \beta\left(V(q) +\frac{\eps}{2}|p^1|^2+\frac{1}{2}|p^2|^2\right)\right)dqdp
\end{align*}
where $Z^\eps$ is the normalisation constant which ensures that $\mu^\eps\in \mathcal P(\R^{2d})$. The same goes for the unnormalised Boltzmann-Gibbs density $\exp(-\beta\mathcal{H})$ that is a stationary solution of the Fokker-Planck equation associated with \eqref{eq:limit-gen}, for any constant $\hat{q}=q^1$. 

Nevertheless, the limit measure as $\eps\to 0$ cannot be a probability measure on the Borel $\sigma$-algebra $\mathcal{B}(\R^d\times \R^d)$ as the density becomes constant in $p^1$, and the Lebesgue measure in $p^1$ is not finite. This property of the invariant measure can be related to the lack of coercivity of the $P^1$-dynamics, since the right-hand side of \eqref{eq:InfMass-p1} is independent of $P^1$. On the other hand, when integrated against any bounded continuous function $f$ of $q$ and $p^2$, the probability measure $\mu^\eps$  has the property 
\[
\int_{\R^d\times\R^d} f(q,p^2) d\mu^\eps(q,p) = \underbrace{\left(\frac{\beta \eps}{2\pi}\right)^{k/2}\int_{R^k}\exp\left( - \beta\frac{\eps}{2}|p^1|^2\right)dp^1}_{=1}\int_{\R^d\times\R^{d-k}}f(q,p^2)\,d\mu(q,p^2) 
\]
where 
\begin{align*}
    d\mu(q,p^2) = \frac{1}{\tilde Z} \exp\left( - \beta\left(V(q)+\frac{1}{2}|p^2|^2\right)\right) dq dp^2.
\end{align*}
with normalisation constant $\tilde Z=\int
\exp(-\beta (V+\frac{1}{2}|p^2|^2))dqdp^2$. As a consequence, the $(q,p^2)$-marginal of $\mu^\eps$ weakly converges to the probability measure $\mu$ on $\mathcal{B}(\R^d\times \R^{d-k})$. 
This is in line with existing results for impetus-striction formulations of the Langevin equation \cite[Sec.~5]{WalterHartmannMaddocks11}. 

Note, however, that the limit \emph{dynamics} cannot sample from the limiting marginal probability measure $\mu$, since $Q^1$ becomes a conserved quantity under the dynamics in the limit $\eps\to 0$. Therefore, setting $Q^1_0=w$, a candidate for the marginal probability measure of the limit dynamics \eqref{eq:abc-slow},\eqref{eq:intro-InfMass} in $(q,p^2)$ is 
\[
    d\nu_{w}(q,p^2) = \frac{1}{\bar{Z}(w)} \exp\left( - \beta \left(V(q)+\frac{1}{2}|p^2|^2\right)\right)\delta_{w}(dq^1)dq^2dp^2.
\]
where $\bar{Z}(w)=\int
\exp(-\beta (V(w,q^2)+\frac{1}{2}|p^2|^2))dq^2dp^2$ denotes the normalisation constant. (Note that there is a difference between the invariant measure of the limiting equation and the limit of the invariant measure, a difference that is related to a non-commutativity of the limits $\eps\to 0$ and $t\to\infty$.)

\subsection{High friction limit}\label{sec:highfric}
The following result states the $\eps\to 0$ limit of the infinite-friction setting with fluctuation dissipation~\eqref{eq:abc-slow},\eqref{eq:intro-InfFriction-withFD-Fast}; the case~\eqref{eq:abc-slow},\eqref{eq:intro-InfFriction-NoFD-fast} without fluctuation-dissipation is discussed in Remark~\ref{rem:InfFric-withouFD} below. Physically speaking, assuming that friction and noise coefficient are balanced by the fluctuation dissipation relation, we are increasing the drag in the system and, proportionally, the thermal fluctuation via the Brownian motion.  Intuitively, one would expect that in this setting the motion of the particle is damped out, and thus the noisy particle position stays, on average, constant. We also expect that the constrained momentum $P^1_t$ converges to a zero-mean Gaussian $\mathcal N(0,\beta^{-1})$ as $\eps \to 0$, which follows from the Fokker-Planck equation corresponding to~\eqref{eq:intro-InfFriction-withFD-Fast}, formally letting $\eps\to 0$. This behaviour is captured by our variance estimates below.

Making the connection to \eqref{eq:abc-fast}, here we study the limit of~\eqref{eq:abc-fast} with $a=1, b =\alpha\geq 0$, with the potential~\eqref{eq:Valpha}, and $c\to \infty$ with the noise term scaled accordingly. 
\begin{prop}\label{prop:PartialMom}
     Given $\eps>0$, let $(Q_t,P_t)$ be the solution to~\eqref{eq:abc-slow},\eqref{eq:intro-InfFriction-withFD-Fast} with corresponding initial data $(Q^i_0,P^i_0)$ for $i=1,2$. Recall the definition of the potential $V$ from~\eqref{eq:Valpha} for any  parameter value $\alpha\geq 0$. Further let $T\geq 1$ be fixed.
    \begin{enumerate}[label=(\roman*)]
        \item\label{item:PartialMom-P1}
        For any $t\geq 0$ we have
        \begin{align}
        \E\Bigl[|Q_t^{1} - Q_0^{1}|^2 \Bigr] &\leq 
        \begin{dcases}
        C_1 \Bigl(  \eps t+ \eps^2 |P_0^{1}|^2   \Bigr) \ \ &\text{if $\alpha=0$,}\\
         C_2 \Bigl( \eps t+ \eps^2 \Bigl(   (1+t^2) |Q_0^{1}|^2 + |P_0^{1}|^2  \Bigr) \Bigr) \ \ &\text{if $\alpha>0$,}
        \end{dcases}
        \label{eq:InfFric-FD-Q1}\\
        \E \Bigl[  \bigl| P^1_t \bigr|^2 \Bigr] & \leq 
        \begin{dcases}
        C_3 \Bigl( \eps^2+  e^{-\frac{C_4}{\eps}t}  |P^1_0|^2   +\beta^{-1} \Bigr) \ \ &\text{if $\alpha=0$,}\\
        C_5 \Bigl(\eps^2+ \eps^2 \left|Q^1_0\right|^2 +  \Bigl(e^{-\frac{C_6}{\eps}t} + \eps^4\Bigr) |P^1_0|^2   +\beta^{-1} \Bigr) &\text{if $\alpha>0$,}
        \end{dcases}
        \label{eq:InfFric-FD-P1}
        \end{align}
    where the constants $C_i>0$ for $i\in\{1,\ldots 6\}$ are independent of $\eps$ and $t$.
        \item\label{item:PartialMomSlow} 
    Let $(q_t,p_t)\in \R^{2(d-k)}$ with initial data $(q_0,p_0)=(Q_0^2,P_0^2)$ evolve according to~\eqref{eq:limit-gen} with $\hat q=Q^1_0 \in \R^k$. We have 
    \begin{align}\label{eq:PartialMom-Slow}
        \E\biggl[\sup_{t\in [0,T]}\biggl|\begin{pmatrix}q_t - Q^{2}_t \\ p_t - P^{2}_t\end{pmatrix}\biggr|^2\biggr] \leq 
        \begin{dcases}
             \eps C_7 e^{C_8 T} T^2 \Bigl( T+\eps |P_0^{1}|^2\Bigr) \ \ &\text{if $\alpha=0$}, \\
             \eps C_9 T  \Bigl(T+\eps  \left|Q^1_0\right|^2  + \eps |P_0^{1}|^2  \Bigr)   &\text{if $\alpha>0$},
        \end{dcases} 
    \end{align}
        where $C_7,C_8,C_9>0$ are independent of $\eps$ and $T$.
    \end{enumerate}
\end{prop}
\begin{proof}
    See Appendix~\ref{proof:inf-fric}. 
\end{proof}
The results show that the constrained variable $Q^1_t$ becomes constant in the limit $\eps\to 0$, i.e.\ $Q^1_t=Q^1_0$ for any $t\geq 0$, which is in accordance with Kozlov's approach to the realisation of constraints by high friction \cite{kozlov1990realization}. Note that this is different from the strong confinement setting of the previous section where $Q^1_t$ converges to the constrained manifold $Q^1=0$.

Following the proofs of Corollary~\ref{cor:point-time-conv-prob} and Sec.~\ref{sssec:Spat-InitDatum}, we also have pointwise-in-time estimates, for instance, for any $t>0$ and any initial datum 
\begin{align*}
    Q^1_t \to Q^1_0,  \ \ (Q^2_t,P^2_t) \to (q_t,p_t) \ \ \text{in probability as \ } \eps\to 0,
\end{align*}
and, furthermore our results generalise to $\eps$-dependent, possibly exploding and random initial datum.

\subsubsection{Steady state in the high friction limit} 

If the fluctuation-dissipation relation holds, then the dynamics~\eqref{eq:abc-slow},\eqref{eq:intro-InfFriction-withFD-Fast} admits a unique invariant measure 
\begin{align*}
    d\mu(q,p) = \frac{1}{Z} \exp\left( - \beta\left(V(q) + \frac{1}{2}|p|^2 \right)\right)dqdp
\end{align*}
that is independent of $\eps$. As a consequence, the limit of $\mu$ as $\eps\to 0$ does not coincide with the invariant measure of the limit dynamics, which is not surprising as the convergence result is on finite time only, so it does not imply any assertion whatsoever about the long term behaviour. Assuming that $Q^1_0=w$, then, by construction, the finite-time limit dynamics for $(Q^2,P^2)$ has the steady state
\begin{align*}
    d\nu_{w}(q^2,p^2)\coloneqq \frac{1}{\bar{Z}(w)} \exp\left( - \beta\left(V(w,q^2) +\frac{1}{2}|p^2|^2\right)\right)dq^2dp^2
\end{align*}
that agrees with the marginal probability measure of the infinite mass case.

\begin{rem}\label{rem:InfFric-withouFD}
    Recall the constrained dynamics~\eqref{eq:intro-InfFriction-NoFD-fast} under the `no fluctuation-dissipation relation' setting. In this case, the scaling behaviour of the constrained-unconstrained variables is as in Proposition~\ref{prop:PartialMom} -- see the end of Appendix~\ref{proof:inf-fric} for the detailed estimate. The key difference to the fluctuation-dissipation relation case discussed earlier is that now the noise in the $P^1$ variable does not depend on $\eps$ and therefore $P^1_t\to 0$ as $\eps\to 0$ in $L^2$. 
    (Recall that with fluctuation-dissipation relation, $P^1_t$ becomes a centred Gaussian with non-zero variance $\beta^{-1}I_{(d-k)\times(d-k)}$ as $\eps\to 0$ that is independent of $(Q^2,P^2)$, so the joint distribution in this case will have a Gaussian marginal in $ p^1$, rather than a Dirac.) 
    
    We expect that the high friction and the large mass limit will show similar behaviour in terms of numerical stability, which would be consistent with fact that both high friction and large mass limits of deterministic mechanical systems belong to the Vakonomic family \cite[Ch.~1.6.4]{arnoldBook} 
\end{rem}

\section{Discussions}\label{sec:discussion}

In this section we comment on several issues pertaining to the constraining of the Langevin dynamics which have not been covered so far in the article. First, we phrase infinite confinement via stiff potentials (recall Section~\ref{sec:InfConf}) into a general form which allows us to discuss related issues and generalisations. We start with explaining the particular form of \emph{constraint-geometry} used in this article, followed by a brief discussion regarding connections to~\cite{Katzenberger91} which contains fairly general qualitative results on stiff--non-stiff stochastic dynamics, and finally discuss the role of mass when dealing phase-space confinement (this has been ignored so far as we assume that particles have unit mass in the Langevin dynamics). 

So far in this article, in particular in Section~\ref{sec:InfConf}, we have restricted our analysis to the coordinate-projection case, i.e.\ $\xi(Q)=Q^1$. We discuss generalisations of our results to affine maps $\xi$. Finally, we reflect on nonlinear $\xi$ and outline the difficulties arising when dealing this setting. 

\subsection{General form of infinite confinement via stiff potentials.} \label{sec:Disc-GenForm}
In the introduction we introduced the general form~\eqref{eq:intro-genLang} of the Langevin dynamics and discussed the two Hamiltonians~\eqref{eq:stuff-Ham-spat} and~\eqref{eq:stuff-Ham-phase} which correspond to constraining via stiff potentials. In fact, the Langevin dynamics with these Hamiltonian are the special cases of the following general Langevin equation (in a compact notation)
\begin{equation}\label{eq:Lang-SC}
    dX_t = (J-A)\nabla H(X_t) dt +\frac{1}{2\eps} K \nabla|\Xi(X_t)|^2dt+ \sqrt{2\beta^{-1}A}\, dW_t,
\end{equation}
where we use $X_t=(Q_t,P_t)\in \R^{2d}$, and the \emph{constraint-geometry} matrix $K\in \R^{2d \times 2d}$. Furthermore, $\Xi:\R^{2d}\to\R^{2k}$ is the confinement potential. 

Note that in the stiff-potential confinement discussed in Section~\ref{sec:intro-StiffConf} and Section~\ref{sec:InfConf} can be written in the form~\eqref{eq:Lang-SC} with the choice $K=J-A$ where
$J\in\R^{2d\times 2d}$ is the canonical skew-symmetric matrix and $A\in\R^{2d\times 2d}$ is the diffusion matrix given explicitly by
\begin{equation*}
J\coloneqq\begin{pmatrix}
    0 & I_{d\times d} \\ -I_{d\times d} & 0 
\end{pmatrix}, \ \ 
A\coloneqq \begin{pmatrix}
    0 & 0 \\ 0 & \gamma I_{d\times d}  
\end{pmatrix}.   
\end{equation*}
Furthermore, $\Xi:\R^{2d}\to \R^{2k}$ encodes the stiff part of the Hamiltonian explicitly given by two choices $\Xi_1$ for spatial-confinement and $\Xi_2$ for phase-space confinement
\begin{equation}\label{eq:genXi}
    \Xi_1(Q) = \begin{pmatrix}
        \xi(Q) \\ 0
    \end{pmatrix}, \ \ \Xi_2(Q,P)=\begin{pmatrix}
        \xi(Q) \\ \nabla\xi^\top(Q)P.
    \end{pmatrix}
\end{equation}
We recall that the spatial constraint is enforced via the given map $\xi:\R^d\to\R^k$.  

\noindent\textbf{Role of constraint-geometry matrix.} The matrix $K$ in~\eqref{eq:Lang-SC} defines the geometry of descent to the zero level set of $\Xi$ as $\eps\to 0$. As pointed out above, in this paper (see Section~\ref{sec:intro-StiffConf}) we make the specific choice $K=A-J$. We now motivate this choice. 

In a recent work~\cite{HartmannNeureitherSharma25a} we study constraints (of stiff-confinement type) for OU processes of the form~\eqref{eq:Lang-SC}  by adding an extra drift term $K\nabla|\eta|^2$ where $\eta$ defines the constraint via $\eta = 0$. In~\cite[Sec.~4]{HartmannNeureitherSharma25a} we discuss the steady state of the strong confinement limit -- we show that if the unconstrained system (i.e.\ with $K=0$) admits a steady state $\mu$, then the steady state of the strong confinement limit crucially depends on the choice of $K$. In particular, with the choice $K=A-J$ the steady state of the limiting dynamics is $\mu$ constrained to the level set $\mu|_{\eta=0}$, while choosing $K=A$ leads to an entirely different steady state. 

Our choice of $K=A-J$ in this article is inspired by similar sampling questions for the underdamped Langevin dynamics, i.e.\ our goal is to sample the \emph{correct constrained steady state}, where `correct' should be understood in the sense that if $\mu$ is the steady state of~\eqref{eq:intro-genLang}, then the strong confinement limit admits the steady state $\mu|_{\Xi=0}$. Following our findings in~\cite{HartmannNeureitherSharma25a}, this suggests that we should consider constraint geometry which is consistent with the underlying (deterministic part of the) Langevin dynamics, i.e.\ $K=A-J$. This is further clarified by our discussion of steady states at the end of each of the sections above. 

We now briefly comment on how constraining via physical parameters discussed in Section~\ref{sec:intro-PhyLimits} and Section~\ref{sec:PhysicalLimits} fit into this general setting. 

The infinite-friction limit without fluctuation dissipation, i.e.\ $\eps\to 0$ limit of~\eqref{eq:abc-slow},\eqref{eq:intro-InfFriction-NoFD-fast} correspond to~\eqref{eq:Lang-SC} with the choices (recall $\xi$ is a coordinate projection in all the physical limits)
\begin{equation*}
    K=A, \ \Xi=\Xi_2(Q,P).
\end{equation*}
In particular, note that the constraint geometry in this case is not consistent with the underlying geometry of the Langevin dynamics. 
The infinite friction with fluctuation dissipation corresponds to~\eqref{eq:Lang-SC} but now also with the noise that scales like $\eps^{-\frac12}$ in the constrained variables. 

The zero mass limit, i.e.\ $\eps\to 0$ limit of~\eqref{eq:abc-slow},\eqref{eq:intro-0Mass-Fast} correspond to~\eqref{eq:Lang-SC} with the choices
\begin{equation}\label{eq:CG-0Mass}
    K=A-J, \ \ \Xi(Q,P)=\begin{pmatrix}
        0 \\ \nabla\xi^\top(Q)P.
    \end{pmatrix}
\end{equation}
Note the peculiar form of $\Xi$, which states that we only add confinement via the coupling $\nabla\xi^\top(Q)P$. However, in the case $\xi(Q)=Q^1$, we have $\nabla\xi^\top(Q)P=P^1$, i.e.\ we only confine the momentum while adding no explicit stiff potential in the spatial variables. 

Finally, we point out that the infinite mass (or impetus-striction) has no added stiff-confinement and corresponds to~\eqref{eq:Lang-SC} with $K=0$ and the Hamiltonian given by~\eqref{eq:Ham-InfMass}.

\noindent\textbf{Connections to~\cite{Katzenberger91}.} Katzenberger~\cite{Katzenberger91} studies the qualitative behaviour of general semimartingales with stiff drifts (characterised by the presence of $\eps>0$ as above). At first sight our setting of constraining via stiff-confinement (see Section~\ref{sec:intro-StiffConf}) seems to be a special case of~\cite[Sec.~8]{Katzenberger91}, but there are some crucial differences. First, the setting of spatial confinement characterised via $\Xi_1$ defined in~\eqref{eq:genXi} is not covered by~\cite{Katzenberger91} -- in particular the constraint does not have the required attracting eigenvalues to ensure convergence to the low-dimensional manifold (see~\cite[Sec.~3.3]{HartmannNeureitherSharma25a} for a detailed explanation for OU processes). Second, even though~\cite{Katzenberger91} covers a considerably larger class of problems, we take an entirely different approach (via variation of constants) which allows us to exploit the underlying semi-linear structure of the Langevin dynamics and extract precise error estimates. This enables us to prove stronger results compared to~\cite{Katzenberger91} for the unconstrained variables as \eqref{eq:conv-prob-phasespace} shows. In particular, the convergence in~\eqref{eq:conv-prob-phasespace} for the unconstrained variables is in $C([0,\infty],\R^{2(d-k)})$, i.e.\  it holds for infinite time horizon, as opposed to the result in~\cite[Sec.8]{Katzenberger91} which only holds up to a certain stopping time.

\noindent\textbf{Constraining momentum in presence of mass.} Note that throughout this article we have worked with unit-mass Langevin dynamics. However, mass of individual particles is often important in practical applications and one can incorporate mass into the Langevin dynamics~\eqref{eq:intro-genLang} by using the Hamiltonian $H(Q,P)=V(Q)+\frac12 p^\top M^{-1}P$, where $M$ is the mass matrix. Of particular interest is the question `How to constrain momentum in the presence of non-identity mass matrix?'. This is achieved by modifying $\Xi_2$ (recall~\eqref{eq:genXi}) used in phase-space confinement appropriately, which leads to two new choices
\begin{align*}
        \Xi_2(q,p) = \begin{pmatrix}
            \xi(q) \\ \nabla\xi^\top(q)M^{-1}p
        \end{pmatrix}, \ \ \Xi_2(q,p) = \begin{pmatrix}
            \xi(q) \\ \bigl(\nabla\xi^\top(q)M^{-1}\nabla\xi(q)\bigr)^{-1} \nabla\xi^\top(q)M^{-1}p
        \end{pmatrix} 
    \end{align*}
    Both these choices for phase-space coarse-graining maps appear in the context of hard-constraints, see~\cite[Sec.~2.2]{LelievreRoussetStoltz12} for physical interpretation and  properties of associated steady states.  

    In the specific setting of coordinate projection $\xi$, i.e.\ $\xi(Q)=Q^1$ we find $\nabla\xi^\top M^{-1}\nabla\xi=M_{11}^{-1}$ (where we use block notation $M=\left(\begin{smallmatrix}
        M_{11} & M_{12} \\ M_{21} & M_{22}\end{smallmatrix}\right)$\,) is the top-left block), which leads using inversion of block matrices to the following explicit forms of the two choices above respectively  
        {\small
        \begin{align*}
        \Xi_2(q,p) = \begin{pmatrix}
            \xi(q) \\ (M_{11}- M_{12}M_{22}^{-1}M_{21})^{-1}p^1 - (M_{11} - M_{12}M_{22}^{-1}M_{21})^{-1}M_{12}M_{22}^{-1}p^2
        \end{pmatrix}, \ \ \Xi_2(q,p) = \begin{pmatrix}
            \xi(q) \\ p^1 + M_{12}M_{22}^{-1}p^2
        \end{pmatrix}. 
    \end{align*}}
    In the case when $M_{12}=0$, which holds for instance when $M$ is a diagonal matrix and specifically $M=mI_{d\times d}$, the two choices become
    \begin{equation*}
        \Xi_2(q,p) = \begin{pmatrix}
            \xi(q) \\ M_{11}^{-1}p^1
        \end{pmatrix}, \ \ \Xi_2(q,p) = \begin{pmatrix}
            \xi(q) \\ p^1 
        \end{pmatrix}.
    \end{equation*}
    The analysis in this article straightforwardly generalises to this setting with $M_{12}=0$. The analysis of more complicated settings is outside the scope of this article.

\subsection[Generalisation to affine constraints]{Generalisation to affine constraints}\label{sec:affine}
While constraining via stiff potentials we restricted our analysis to 
coordinate-projection $\xi(Q)=Q^1$ (see start of Section~\ref{sec:limit-results}). We now briefly discuss how our results readily generalise to affine $\xi$, by looking at the particular setting of phase-space confinement discussed in Section~\ref{sec:phase-space}. 

To make things precise, consider the affine map
\begin{equation*}
    \xi(Q)=YQ + e, 
\end{equation*}
where $e\in\R^k$ is a given constant vector and $Y\in\R^{k\times d}$ satisfies $YY^\top>0$. The latter requirement ensures that $\xi$ is non-degenerate as $\nabla\xi^\top\nabla\xi=YY^\top>0$. Note that the  the pre-limit Langevin dynamics in the case of phase-space confinement can be derived using the general form~\eqref{eq:Lang-SC} with $K=J-A$ the mapping $\Xi_2$ in~\eqref{eq:genXi}. Furthermore, the constrained variables in this case are given by 
\[(\bar Q,\bar P)=\Xi_2(Q,P)=\Bigl(\xi(Q),\nabla\xi(Q)^\top P\Bigr)=\Bigl(YQ+e,YP\Bigr).\]
Using It\^o's lemma, these constrained variables evolve according to
\begin{align*}
    \begin{pmatrix}
        d\bar Q_t \\ d\bar P_t
    \end{pmatrix}
    =
    \begin{pmatrix}
        0 & I_{k\times k} + \frac1\eps YY^\top \\ -\frac1\eps YY^\top & -\gamma\bigl(I_{k\times k}+\frac1\eps YY^\top \bigr)  
    \end{pmatrix}
     \begin{pmatrix}
        \bar Q_t \\ \bar P_t
    \end{pmatrix} dt +
    \begin{pmatrix}
       0 \\ - Y\nabla V(Q_t)
    \end{pmatrix}dt
    + \sqrt{2\gamma\beta^{-1}}\begin{pmatrix}
        0 \\  Y
    \end{pmatrix}dW_t.
\end{align*}
Since $YY^\top>0$, there exist an orthonormal matrix such that $\tilde Y = O(YY^\top)O^\top \in \R^{k\times k}$ is a diagonal matrix. The transformed variables
\begin{equation*}
    \tilde Q = O \bar Q, \ \ \tilde P = O \bar P,
\end{equation*}
evolve according to the evolution
\begin{equation*}
    \begin{pmatrix}
        d\tilde Q_t \\ d\tilde P_t
    \end{pmatrix}
    =
    \begin{pmatrix}
        0 & I_{k\times k} + \frac1\eps \tilde Y \\ -\frac1\eps \tilde Y & -\gamma\bigl(I_{k\times k}+\frac1\eps \tilde Y \bigr)  
    \end{pmatrix}
     \begin{pmatrix}
        \tilde Q_t \\ \tilde P_t
    \end{pmatrix} dt +
    \begin{pmatrix}
       0 \\ - OY\nabla V(Q_t)
    \end{pmatrix}dt
    + \sqrt{2\gamma\beta^{-1}}\begin{pmatrix}
        0 \\  OY
    \end{pmatrix}dW_t
\end{equation*}
Comparing the evolution of these ($O$-transformed) constrained variables $(\tilde Q_t,\tilde P_t)$ to the constrained variables $(Q^1_t,P^1_t)$~\eqref{eq:intro-Phase-Cons-Fast} in the coordinate-projection setting (details in Section~\ref{sec:phase-space}), the only difference in the stiff-part is the presence of additional constants due to the diagonal matrix $\tilde Y$. Therefore, one can repeat the same procedure as in Section~\ref{sec:phase-space} to derive the $\eps\to 0$ limit of $(\tilde Q_t,\tilde P_t)$ and transform back to arrive at the limiting dynamics for the original constrained variables $(\bar Q_t,\bar P_t)$. Note that the assumptions on $V$ allows us to deal with the $\nabla V(Q_s)$ term. Similar transformation and results have been discussed in our earlier work on constrained linear-diffusions, see~\cite[Remark 3.11]{HartmannNeureitherSharma25a}.

\subsection[Towards nonlinear constraints]{Towards  nonlinear constraints}
In parts of this article we have focussed on constraining via stiff-confinement using coordinate-projection CG maps $\xi$ (which generalise to affine maps as discussed above). We now discuss the case of nonlinear maps $\xi$, which are extremely important from a practical viewpoint but are harder to analyse, as we now explain. It turns out that, even when working unit-mass Langevin dynamics with nonlinear spatial CG maps $\xi$, constraining the phase-space via $\Xi_2$~\eqref{eq:genXi} is a good choice, see related discussions in~\cite{LelievreRoussetStoltz12,DLPSS18}. 

Following the discussion about affine constraints,  consider the Langevin dynamics~\eqref{eq:Lang-SC} with $K=A-J$ and $\Xi_2$ where $\xi:\R^d\to\R$ is scalar-valued (for simplicity of discussion), smooth, and non-degenerate i.e\ $|\nabla\xi|^2>0$. Using It\^o's lemma, the behaviour of the constrained variables $(\bar Q_t,\bar P_t)=\Xi_2(Q_t,P_t)\in\R^{2k}$ is given by
    \begin{align*}
        d\bar Q_t &= \bar P_t dt + \frac{1}{\eps} |\nabla\xi(Q_t)|^2\bar P_t dt \\
        d\bar P_t &= \Bigl( P_t^\top \nabla^2\xi(Q_t)P_t - (\nabla\xi^\top\nabla V)(Q_t) - \gamma \bar P_t\Bigr)dt  - \frac1\eps|\nabla \xi(Q_t)|^2 \Bigl( \bar Q_t + \gamma\bar P_t \Bigr)dt + \sqrt{2\gamma\beta^{-1}} \nabla\xi^\top(Q_t) dW_t.
    \end{align*}
    Since $|\nabla\xi(q)|>0$ for any $q\in \R^{d}$, we can introduce the time-rescaling 
    \begin{equation*}
        \tau(t) = \int_0^t |\nabla\xi(Q_s)|^2ds,
    \end{equation*}
    using which the constrained dynamics above can be rewritten as
    \begin{align*}
        \begin{pmatrix}
            d\bar Q_\tau \\ d\bar P_\tau 
        \end{pmatrix}
        = \begin{pmatrix}
            0 & \frac1\eps \\ -\frac1\eps & -\frac{\gamma}{\eps}
        \end{pmatrix} 
        \begin{pmatrix}
        \bar Q_\tau \\ \bar P_\tau      
        \end{pmatrix} d\tau + F(Q_\tau,P_\tau)dt + \begin{pmatrix}
             0 \\  \sqrt{2\gamma\beta^{-1}}
        \end{pmatrix}d\bar W_\tau,
    \end{align*}
    where $\bar W_t$ is a standard $2$-d Brownian motion and $F:\R^{2d}\to \R^{2k}$ is defined as 
    \begin{align}\label{eq:nonLinCG-rem}
        F(q,p) = 
        \begin{pmatrix}
            |\nabla\xi(q)|^{-2} \\
            |\nabla\xi(q)|^{-2} \Bigl( p^\top \nabla^2\xi(q)p - (\nabla\xi^\top\nabla V)(q) - \gamma \nabla\xi^\top(q)p \Bigr) 
        \end{pmatrix}.
    \end{align}
Note that the linear part of the coarse-grained dynamics is exactly the same as the linear part for (coordinate-projection) phase-space constrained dynamics~\eqref{eq:intro-Phase-Cons-Fast} studied in Section~\ref{sec:phase-space}. Recall that the proof of the soft-constrained limit employs variation of constants (see Appendix~\ref{App:PS-cons-proof}) and uses the fact that $\nabla_{q^1} V$ is bounded (recall assumption~\eqref{ass:Vbound}). In this nonlinear setting this boundedness assumption would correspond to assuming that $F$ defined in~\eqref{eq:nonLinCG-rem} is bounded. This is, however, not possible since one of the terms is of the form $\R^{2d}\ni(q,p)\mapsto p^\top \nabla ^2\xi(q)p$ which scales quadratically in $p$ and therefore cannot be bounded. Consequently, dealing with nonlinear CG maps would require a careful analysis of the asymptotic behaviour of $|\nabla\xi(Q_t)|^{-2}P_t^\top \nabla^{2}\xi(Q_t)P_t$ as $\eps\to 0$. This study is left to future work.

\section*{Acknowledgments}
The authors thank Chetan Pahlajani for pointing out an error in the handling of stochastic integrals appearing in an earlier draft of this paper. 
This research has been partially funded by the German Federal Government, the
Federal Ministry of Education and Research and the State of Brandenburg within the framework of the joint
project EIZ: Energy Innovation Center (project numbers 85056897 and 03SF0693A) and by the Deutsche
Forschungsgemeinschaft (DFG) through the grant CRC 1114: Scaling Cascades in Complex Systems (project
no.\ 235221301). US acknowledges support from CRC 1114 for hosting the visit to FU Berlin and BTU Cottbus-Senftenberg. 

\begin{appendices}
\addtocontents{toc}{\protect\setcounter{tocdepth}{1}}

\section{Variation of constants}\label{app:VarConst}

Throughout the proofs in this article we will make use of the following simple result which summarises an explicit solution for a class of SDEs with linear and nonlinear drift terms.

\begin{prop}[Variations of constants] \label{prop:varofconst}
For $t>0$, let $(x_t,y_t)\in\R^\ell\times\R^m$ be a strong unique solution to a coupled SDE system, where $x_t$ evolves according to  
\begin{align} \label{eq:SDEmatrixnot}
dx_t = A x_t dt+ f(x_t,y_t,t) dt + C dW_t.
\end{align} 
Here $A \in \R^{\ell \times \ell}$ and $C \in \R^{\ell \times k}$ are constant matrices, $W_t$ is a standard Brownian motion in $\R^k$ and  $x_0\in\R^\ell$ is the initial condition. Furthermore, assume that $f:\R^{\ell+m}  \times \R_{\geq 0} \to \R^\ell$ satisfies 
\begin{enumerate}
    \item uniformly Lipschitz continuous in space, i.e.\ there exists $L_f>0$ such that for any $z,z'\in\R^{\ell+m}$ we have  $|f(z,t)-f(z',t)|\leq L_f|z-z'|$;
    \item sublinear growth, i.e.\ there exists a constant $c>0$ such that for any $z\in\R^{\ell+m}$ and $t>0$ we have $|f(z,t)|\leq c(1+|z|)$. 
\end{enumerate}
Then $x_t$ can be explicitly written as
\begin{equation}\label{eq:VarofCon-sol}
    x_t = e^{At}x_0 + \int_0^t e^{A(t-s)} f(x_s,y_s,s) ds + \int_0^t e^{A(t-s)} C dW_s\,.
\end{equation}
\end{prop}

The existence and uniqueness of the strong solution is standard (see for instance~\cite[Theorem 26.8]{Klenke13}). The integral form of the solution follows by using variations of constants along with integration by parts for It\^o integrals. 
\begin{proof}
Using the stochastic integration by parts formula (see for instance~\cite[Theorem 18.16]{Kallenberg21}) and $W_0=0$ almost surely we find
\begin{align}\label{eq:stoc-IntByParts}
    \int_0^t e^{A(t-s)} C dW_s = C W_t + \int_0^t A e^{A(t-s)} C W_s ds 
\end{align}
almost surely. Hence $x_t - CW_t \eqqcolon g_t$, where $x_t$ is given by~\eqref{eq:VarofCon-sol}, is differentiable and satisfies 
\begin{equation*}
    \dot g_t = A \biggl( e^{At}x_0 + \int_0^t e^{A(t-s)} f(x_s,y_s,s) ds +  \int_0^t A e^{A(t-s)} C W_s ds \biggr) + f(x_t,y_t,t) + ACW_t = A x_t + f(x_t,y_t,t).
\end{equation*}
Therefore, using the definition of $g_t$ we have
\begin{align*}
    x_t = g_0 + \int_0^t  \frac{dg_s}{ds} ds + CW_t 
        = x_0 + \int_0^t \bigl[Ax_s + f(x_s,y_s,s)\bigr]ds + \int_0^t C\, dW_s,
\end{align*}
i.e.\ $x_t$ solves~\eqref{eq:SDEmatrixnot}. 
\end{proof}

In what follows, we will apply Proposition~\ref{prop:varofconst} to the following Langevin dynamics in $\R^2$
\begin{equation} \label{eq:lemsollangevin}
\begin{aligned} 
    dq_t &= ap_t dt, \\
    dp_t &= \bigl(- \nabla_{q} V(q_t,\tilde q_t) - b q_t - c p_t\bigr) dt + \tilde c  dB_t,
\end{aligned} 
\end{equation} 
where $(q_t,p_t)\in\R\times\R$, $V:\R^{1+m}\to \R$, $B_t$ is a standard one-dimensional Brownian motion and $a,b,c,\tilde c>0$ are positive constants. Here $\tilde q_t \in \R^{m}$ is a latent variable which itself solves an SDE such that the system $(q_t,p_t,\tilde q_t)$ is assumed to have a unique strong solution. We will often use the variables $B_t$ and $W_t$ interchangeably to denote the Brownian motion.

Now, Proposition~\ref{prop:varofconst} applies to~\eqref{eq:lemsollangevin} with 
\begin{align} \label{eq:choiceforfuncoupled}
    x = (q,p)^\top\,, \  y_t=\tilde q_t, \ \  A = - \begin{pmatrix}
0 & -a \\ b & c 
\end{pmatrix}\,,  \ \ f(x_t,y_t,t) =\begin{pmatrix} 0 \\ -\nabla_q V(q_t,\tilde q_t) \end{pmatrix}\,, \ \ C = \begin{pmatrix}   0 \\  \tilde c\end{pmatrix} \,. 
\end{align}

The following three results use Proposition~\ref{prop:varofconst} to arrive at explicit solutions to~\eqref{eq:choiceforfuncoupled} in three cases: $\theta\neq 0$, $0 \neq  \theta  \in \mathbb C$  and $\theta=0$, where 
\begin{equation}\label{def:theta}
    \theta \coloneqq \sqrt{c^2-4ab}.
\end{equation}
Here $a,c>0$ and $b\neq 0$ are the constants appearing  in~\eqref{eq:lemsollangevin}.

\begin{prop}[Solutions for $\theta\neq 0$] \label{prop:sollangevin}
Let $c^2 \neq 4ab$ with $\theta$ defined in~\eqref{def:theta}. Then~\eqref{eq:lemsollangevin} admits the solution 
\begin{subequations}
\begin{align}
\begin{split}\label{eq:q-sol}
\begin{aligned}
    q_t &=  \Bigl[ \mfrac{1}{2}\Bigl(1-\mfrac{c}{\theta}\Bigr)e^{-\frac{1}{2}(c+\theta) t} + \mfrac{1}{2}\Bigl(1+\mfrac{c}{\theta}\Bigl)e^{-\frac{1}{2}(c-\theta) t} \Bigr] q_0 - \mfrac{a}{\theta} \Bigl(e^{-\frac{1}{2}(c+\theta) t} - e^{-\frac{1}{2}(c-\theta) t}\Bigr)p_0 \\ 
    & \ + \mfrac{a}{\theta} \int_0^t\bigl( e^{-\frac{1}{2}(c+ \theta)(t-s)} -e^{-\frac{1}{2}(c - \theta)(t-s)}\bigr) \nabla V(q_s,\tilde q_s) ds + \mfrac{a \tilde c }{\theta} \int_0^t \bigl(e^{-\frac{1}{2}(c - \theta)(t-s)} - e^{-\frac{1}{2}(c + \theta)(t-s)}\bigr) dB_s  \,,   
\end{aligned}
\end{split}\\
\begin{split}\label{eq:p-sol}
    \begin{aligned}
    p_t &= \mfrac{b}{\theta} \Bigl(e^{-\frac{1}{2}(c+\theta) t} - e^{-\frac{1}{2}(c-\theta) t}\Bigr)q_0  + \mfrac{a}{\theta} \Bigl[\mfrac{2b}{c-\theta}e^{-\frac{1}{2}(c+\theta) t} - \mfrac{2b}{c+\theta}e^{-\frac{1}{2}(c-\theta) t}\Bigr] p_0  \\
    & \quad - \mfrac{2ab}{\theta} \int_0^t \Bigl( \mfrac{1}{c- \theta}e^{-\frac{1}{2}(c+\theta) (t-s)} - \mfrac{1}{c + \theta} e^{-\frac{1}{2}(c-\theta) (t-s)}\Bigr) \nabla V(q_s,\tilde q_s) ds  \\ 
    & \qquad \quad  +\mfrac{2a b\tilde c}{\theta} \int_0^t \Bigl( \mfrac{1}{c- \theta}e^{-\frac{1}{2}(c+\theta) (t-s)} - \mfrac{1}{c + \theta} e^{-\frac{1}{2}(c-\theta) (t-s)}\Bigr) dB_s.
    \end{aligned}
\end{split}
\end{align}
\end{subequations}
\end{prop} 
\begin{proof}
For $c^2 \neq 4ab$ and $\theta=\sqrt{c^2 - 4ab}$, the matrix $A$ in \eqref{eq:choiceforfuncoupled} admits the eigenvalues and eigenvectors
\begin{align*}
    \lambda_{1} = - \frac{1}{2}(c + \theta), \ v_1 = 
    \begin{pmatrix}\frac{\lambda_2}{b} \\  1\end{pmatrix};
    \quad 
    \lambda_{2} = - \frac{1}{2}(c - \theta), \ v_2 = 
    \begin{pmatrix}\frac{\lambda_1}{b} \\  1\end{pmatrix}.
\end{align*}
Writing $S=(v_1 \ v_2) = \begin{pmatrix}
\frac{\lambda_2}{b} & \frac{\lambda_1}{b} \\ 1 & 1
\end{pmatrix}$ and using $S^{-1} = \frac{b}{\theta} \begin{pmatrix} 1 & \frac{- \lambda_1}{b} \\ -1 & \frac{\lambda_2}{b}
\end{pmatrix} $ we calculate
\begin{align} \label{eq:expAt}
    e^{At} = S \begin{pmatrix}
    e^{\lambda_1 t} & 0 \\ 0 & e^{\lambda_2 t}
    \end{pmatrix} S^{-1}  = 
    \begin{pmatrix}
        \frac{1}{2}(-\frac{c}{\theta} + 1)e^{\lambda_1 t} + \frac{1}{2}(\frac{c}{\theta} + 1)e^{\lambda_2 t} &  -\frac{a}{\theta} (e^{\lambda_1 t} - e^{\lambda_2 t})\\
        \frac{b}{\theta} (e^{\lambda_1 t} - e^{\lambda_2 t}) & \frac{a}{\theta} \left(- \frac{b}{\lambda_2}e^{\lambda_1 t} + \frac{b}{\lambda_1}e^{\lambda_2 t}\right)
    \end{pmatrix}.
\end{align}
where we have used $\lambda_1 \lambda_2 = ab$. 
The result then follows from Proposition~\ref{prop:varofconst} along with~\eqref{eq:choiceforfuncoupled}. 
\end{proof}
The solution above can be further simplified when $\theta\in\mathbb C$ as  discussed below.  
\begin{corollary}[Solutions for $0\neq\theta \in \mathbb C \backslash \mathbb{R}$]  \label{cor:sol_underdamped} 
Let $c^2 < 4ab$, i.e.\ $\theta$ defined in~\eqref{def:theta} satisfies $\theta= i|\theta| = i \sqrt{4ab-c^2}$. Then~\eqref{eq:lemsollangevin} admits the solution
\begin{subequations}
\begin{align}
    \begin{split} \label{eq:q-sol-complextheta}
       q_t &= \Bigl[ \cos\Bigl(t\mfrac{|\theta|}{2}\Bigr)+ \frac{c}{|\theta|}\sin\Bigl(t\mfrac{|\theta|}{2}\Bigr)  \Bigr] e^{-\frac{c}{2}t}q_0 + \frac{2a}{|\theta|} \sin\Bigl(t\mfrac{|\theta|}{2}\Bigr)e^{-\frac{c}{2}t}p_0 \\ 
       & \qquad - \mfrac{2a}{|\theta|} \int_0^t \sin\Bigl((t-s)\mfrac{|\theta|}{2}\Bigr)e^{-\frac{c}{2}(t-s)}  \nabla V(q_s,\tilde q_s) ds   +  \frac{2a\tilde c}{|\theta|} \int_0^t \sin\Bigl((t-s)\mfrac{|\theta|}{2}\Bigr)e^{-\frac{c}{2}(t-s)}  dB_s,
    \end{split}\\
    \begin{split}\label{eq:p-sol-complextheta}
    p_t &=  - \frac{2b}{|\theta|}  \sin\Bigl(t\mfrac{|\theta|}{2}\Bigr)e^{-\frac{c}{2}t}q_0 + \Bigl[\cos\Bigl(t\mfrac{|\theta|}{2}\Bigr) - \frac{c}{|\theta|}\sin\Bigl(t \mfrac{|\theta|}{2}\Bigr)\Bigr]e^{-\frac{c}{2}t}p_0 \\
    & \qquad - \int_0^t \Bigl[ \cos\Bigl((t-s)\frac{|\theta|}{2}\Bigr) - \frac{c}{|\theta|} \sin\Bigl((t-s)\frac{|\theta|}{2}\Bigr) \Bigr] \nabla V(q_s,\tilde q_s) e^{-\frac{c}{2}(t-s)} ds \\
    &\qquad+ \  \tilde c \int_0^t \left( \cos\Bigl((t-s)\mfrac{|\theta|}{2}\Bigr) - \frac{c}{|\theta|} \sin\Bigl((t-s)\mfrac{|\theta|}{2}\Bigr) \right) e^{-\frac{c}{2}(t-s)} dB_s.
    \end{split}
\end{align}
\end{subequations}
\end{corollary}
\begin{proof}
The solution for $q_t$ follows by applying 
\begin{align*}
\mfrac{1}{\theta}\Bigl[e^{ \frac{\theta}{2}(t-s)} - e^{-\frac{\theta}{2}(t-s)}\Bigr] = \mfrac{1}{i|\theta|}\Bigl[ e^{ i(t-s)\frac{|\theta|}{2}} - e^{-i(t-s)\frac{|\theta|}{2}} \Bigr]= \mfrac{2}{|\theta|}\sin\Bigl((t-s)\mfrac{|\theta|}{2}\Bigr)\,,
 \end{align*}
to $q_t$ in Proposition~\ref{prop:sollangevin}. Using $|\theta|^2 + c^2 = 4ab$ we note that
\begin{align*}
     \frac{1}{\theta(c-\theta)} &= \frac{1}{|\theta|i(c-|\theta|i)} = \frac{1}{|\theta|^2 + |\theta|c i} = \frac{|\theta|^2 - |\theta|ci}{|\theta|^4 + |\theta|^2 c^2} = \frac{1}{|\theta|^2+c^2} - \frac{c}{|\theta|(|\theta|^2+c^2)}i = \frac{1}{4ab} - \frac{c}{4ab|\theta|}i, \\
     \frac{1}{\theta(c+\theta)} &= -\overline{\frac{1}{\theta(c-\theta)}}=-\frac{1}{4ab} - \frac{c}{4ab|\theta|}i\,,  
 \end{align*}
which gives
\begin{align*}
     \frac{2a}{\theta} \left(\frac{1}{c- \theta}e^{-\frac{1}{2}\theta t} - \frac{1}{c + \theta } e^{\frac{1}{2}\theta t}\right) &= \frac{1}{2b} (e^{-it\frac{|\theta|}{2}} + e^{it\frac{|\theta|}{2}} ) - \frac{c}{2b|\theta|} i (e^{-it\frac{|\theta|}{2}} - e^{it\frac{|\theta|}{2}})\\
     &=\frac{1}{b} \cos\biggl(\frac{|\theta|t}{2}\biggr) - \frac{c}{b|\theta|} \sin\biggl(\frac{|\theta|t}{2}\biggr).
\end{align*}
The solution for $p_t$ then follows by applying this identity to $p_t$ in Proposition~\ref{prop:sollangevin}. 
\end{proof}

The following result states the explicit solution for $\theta=0$.
\begin{prop}[Solution for $\theta=0$] \label{prop:sollangevin_Adeg}
 Let $c^2 = 4ab$ with $\theta$ defined in~\eqref{def:theta}. Then~\eqref{eq:lemsollangevin} admits the solution 
\begin{align*}
    q_t &= e^{-\frac{c}{2}t}\Bigl[\Bigl(1+\mfrac{c}{2}t\Bigr)q_0 + atp_0  \Bigr] -\int_0^t e^{- \frac{c}{2} (t-s) }a(t-s) \nabla V(q_s,\tilde q_s) ds + \tilde c\int_0^t e^{- \frac{c}{2} (t-s) }a(t-s) dB_s,\\
    p_t &=  e^{-\frac{c}{2}t}\Bigl[-btq_0 + \Bigl(1-\mfrac{c}{2}t\Bigr)p_0  \Bigr] - \int_0^t e^{- \frac{c}{2} (t-s) }\Bigl[1-\mfrac{c}{2}(t-s)\Bigr] \nabla V(q_s,\tilde q_s) ds \\
    &\qquad + \tilde c\int_0^t e^{- \frac{c}{2} (t-s) }\Bigl[1-\mfrac{c}{2}(t-s)\Bigr] dB_s.
\end{align*}
\end{prop}
\begin{proof} Here $A$ is not diagonalizable since it's only eigenvalue $\lambda= - \frac{c}{2}$ (with eigenvector $(-\frac{2a}{c},1)^\top$) has algebraic multiplicity $2$. The generalised eigenvector corresponding to $\lambda$ is $(-\frac{1}{b},0)^\top$ and therefore using $S\coloneqq\left(\begin{smallmatrix}
    -\frac{2a}{c} & -\frac{1}{b} \\ 1 & 0
\end{smallmatrix}\right)$, 
$L\coloneqq\left(\begin{smallmatrix}
    \lambda & 1 \\ 0 & \lambda
\end{smallmatrix}\right)$, 
$A$ admits the decomposition $A=SLS^{-1}$ and we calculate 
\begin{equation}\label{eq:expA-theta0}
    e^{At} = e^{-\frac{c}{2}t} \begin{pmatrix}
1+\frac{c}{2}t & at \\ -bt & 1- \frac{c}{2}t
\end{pmatrix}.
\end{equation}
The required result then follows by substituting this matrix exponential  along with~\eqref{eq:choiceforfuncoupled} into Proposition~\ref{prop:varofconst}. 
\end{proof}

\section{Proofs of asymptotic limits}\label{app:LimitProofs}

In this section we provide the proofs for various limits discussed in Section~\ref{sec:InfConf} and Section~\ref{sec:PhysicalLimits}. All the proofs follow a similar strategy which we now outline.

\subsection{Proof strategy for the constrained variables}\label{sec:Strategy-constrained}
For the constrained variables, denoted by $x^1$ in the following, we investigate the $\eps\to 0$ behaviour of  $x_t=(x^1_t,x^2_t)\in\R^m$ where $x^1$  solves the SDE of the type
\begin{equation*}
    x^1_t = f_0(t,x_0,\eps) + \underbrace{\int_0^t f_1(s,t,\eps) \nabla_{x^1} U(x^1_s,x^2_s) ds}_{=:I_1(t,\eps)} + \underbrace{\int_0^t f_2(s,t,\eps) dB_s}_{=:I_2(t,\eps)}.
\end{equation*}
Here 
$B_t$ is a standard Brownian motion, $V,f_0,f_1,f_2$ are sufficiently smooth functions with appropriate growth conditions to ensure well-posedness of solutions and $|\nabla_{x^1} U|<C_U$ (recall~\eqref{ass:Vbound}).

Using Young's inequality along with the  It\^o isometry for $I_2$ we find
\begin{align}\label{eq:genPathBound}
    \E\Bigl[ |x^1_t|^2\Bigr] 
    \leq 3 \biggl\{ \E\Bigl[|f_0(t,x_0,\eps)|^2\Bigr] + C_U^2 \biggl(\int_0^t|f_1(s,t,\eps)| ds\biggr)^2+  \int_0^t \bigl|f_2(s,t,\eps)\bigr|^2ds \biggr\} 
\end{align}
where we have used $|\nabla_{x^1} U|\leq C_U$ to bound the $I_1$ term. 
\subsection{Proof strategy for the unconstrained variables and auxiliary results}
In the following we study the behaviour of the unconstrained variables, which often contain   $(Q^{2}_t,P^{2}_t)$ (recall~\eqref{eq:abc-slow}) 
which solve 
\begin{align}\label{eq:Q2P2}
\begin{aligned}
    dQ^{2}_t &= P^{2}_t dt    \\
    dP^{2}_t &= -\nabla_{q^2} V(Q^{1}_t,Q^{2}_t) dt  - \gamma P^{2}_t dt + \sqrt{2\gamma\beta^{-1}} \,dW^2_t,
\end{aligned}
\end{align}
where $\nabla_{q^2}$ is the gradient with respect to $q^2\in \R^{d-k}$. Recall that the potential $V$ is assumed to be of the form $V(q) = \frac{\alpha}{2}|q|^2 + U(q)$ (see \eqref{eq:Valpha}), where $U$ satisfies \eqref{ass:VLip}-\eqref{ass:Vbound}. In the following two lemmas we derive bounds on $\|e^{Kt}\|_F^2$, where $K \in \R^{(d-k) \times (d-k)}$ is the matrix which captures the linear terms of the dynamics. These bounds will be used in Theorem \ref{cor:gronwall} below to derive estimates for the unconstrained variables, which turn out to crucially depend on $\alpha$ (recall the particular form~\eqref{eq:Valpha} of the potential). As we shall see later, in case $\alpha =0$, (i.e.\ $V$ does not admit a quadratic part in $q$) the error estimate for unconstrained variables (when dealing with stiff potentials) will grow exponentially in $T$ as opposed to linear growth in $T$ if $\alpha>0$ (i.e.\ $V$ admits a quadratic term, i.e. $\alpha>0$).

\begin{lemma}[Matrix exponential I]\label{lem:matrixexp}
Let $K \coloneqq \begin{pmatrix}
        0_{(d-k) \times (d-k)} & I_{(d-k) \times (d-k)} \\ 0_{(d-k) \times (d-k)} & - \gamma  I_{(d-k) \times (d-k)} 
    \end{pmatrix}$ with $\gamma>0$. We then have 
    \begin{align}\label{eq:matrixexp}
        e^{Kt} = \begin{pmatrix}
            I & -\gamma^{-1} e^{-\gamma I t} + \gamma^{-1}I \\ 0 & e^{-\gamma It}
        \end{pmatrix},
    \end{align}
    where the identity matrices above are all $d-k$ dimensional. 
    Furthermore, for any $t\geq 0$ we have the bound 
    \begin{align*}
        \bigl\|e^{Kt}\bigr\|^2_F \leq (d-k) (2+\gamma^{-2}).
    \end{align*}
\end{lemma}
\begin{proof}
The proof of~\eqref{eq:matrixexp} follows by a straightforward expansion of the matrix exponential. Using~\eqref{eq:matrixexp} we calculate
\begin{align*}
     \bigl\| e^{ K t} \bigr\|_F^2 &= \tr\bigl( e^{Kt} (e^{Kt})^\top\bigr)
     = \tr \begin{pmatrix}
         I + \gamma^{-2} I - 2\gamma^{-2} e^{-2\gamma I t}  + \gamma^{-2} e^{-2\gamma I t} & \cdot \\
         \cdot & e^{-2\gamma I t} 
     \end{pmatrix}\\
     & = \tr\bigl( I + \gamma^{-2}(I - e^{-2\gamma I t}) \bigr) + \tr\bigl( e^{-2\gamma I t}\bigr) 
     = (1+\gamma^{-2})\tr(I) + (1-\gamma^{-2})\tr(e^{-2\gamma I t})\\
     & = (1+ \gamma^{-2})\tr(I) + (1-\gamma^{-2}) \tr(I)  e^{-2\gamma t} \leq (d-k) \bigl( 2+ \gamma^{-2} \bigr)
\end{align*}
where the final inequality follows since $-\gamma^{-2}e^{-2\gamma t}\leq 0$, $e^{-2\gamma t}\leq 1$ for any $t\geq 0$ and $\tr(I)=d-k$. 
\end{proof}
\begin{lemma}[Matrix exponential II]\label{lem:matrixexp-b} 
Let $K \coloneqq \begin{pmatrix}
        0 & I \\ -\alpha I & - \gamma I 
    \end{pmatrix} \in \R^{2(d-k) \times 2(d-k)}$ with $\alpha,\gamma>0$, and $\theta= \sqrt{\gamma^2 - 4\alpha}$. Then for any $t\geq 0$
    \begin{align*}
        \bigl\|e^{Kt}\bigr\|^2_F \leq C (d-k) e^{-\eta t}, \text{ where } \eta \in (0,\gamma]\,, \  \eta = \begin{cases}
        \gamma - \sqrt{\gamma^2 - 4\alpha}, &\text{ if } \gamma^2 > 4\alpha \\ \gamma , &\text{ if } \gamma^2 < 4\alpha \\ \gamma - \delta , \text{ for any } \delta \in (0,\gamma) &\text{ if } \gamma^2 = 4\alpha .\end{cases} 
    \end{align*}
\end{lemma}
\begin{proof}
First note that in the proof we consider $\tilde K = \begin{pmatrix}
    0 & 1 \\ - \alpha & - \gamma
\end{pmatrix} \in \R^{2 \times 2}$ instead of $K \in \R^{2(d-k)\times 2(d-k)}$ as given above. We do this since $\| e^{Kt}\|^2_F = (d-k) \|e^{\tilde K t }\|^2_F$, which follows by re-ordering the variables (arrange $Q^i$'s and $P^i$'s together) thereby leading to block-diagonal structure with $\tilde K$ on each of the diagonal blocks. Here onwards $K=\tilde K$.

The proof follows by using~\eqref{eq:expAt}-\eqref{eq:expA-theta0}, which calculates the matrix exponential for $A$ defined in~\eqref{eq:choiceforfuncoupled}; our case follows with the choice $a=1, c= \gamma$. Introduce $\theta = \sqrt{\gamma^2 - 4\alpha}$. The proof consists of three cases: first we deal with non-zero $\theta$, in particular we treat $\theta \in \R$, i.e. $\gamma^2-4\alpha >0$ and the case when $\theta$ is purely complex, i.e. $\gamma^2-4\alpha <0$, separately. Finally, we consider $\theta=0$.

\noindent\textbf{Case $\theta\neq 0$.}
Using the expression for $e^{Kt}$ in ~\eqref{eq:expAt}, and , $\, \lambda_1 = -\frac{1}{2}(\gamma+\theta), \, \lambda_2 = -\frac{1}{2}(\gamma-\theta) $ and $\lambda_1 \lambda_2 = \alpha$ we have
\begin{align*}
    e^{Kt} = \frac{1}{\theta}\begin{pmatrix}
        \lambda_2 e^{\lambda_1 t} - \lambda_1 e^{\lambda_2 t} & e^{\lambda_2 t } - e^{\lambda_1 t} \\ \alpha(e^{\lambda_1 t} - e^{\lambda_2 t}) & \lambda_2 e^{\lambda_2 t} - \lambda_1 e^{\lambda_1 t}
    \end{pmatrix}
\end{align*}
and compute
\begin{align*}
     \bigl\| e^{ K t} \bigr\|_F^2 &= \tr\bigl( e^{Kt} (e^{Kt})^*\bigr)
     = \frac{1}{|\theta|^2} \Bigl\{ \Bigl[e^{(\lambda_1 + \lambda_1^*)t} + e^{(\lambda_2 + \lambda_2^*)t} \Bigr]\left( |\lambda_1|^2 + |\lambda_2|^2 + 1 + \alpha^2 \right)  \\
     &  \quad -  \Bigl[ e^{(\lambda_1 + \lambda_2^*)t} + e^{(\lambda_2 + \lambda_1^*)t} \Bigr] \left( \lambda_1^* \lambda_2 + \lambda_2^* \lambda_1 + 1 + \alpha^2 \right)
    \Bigr\} .
\end{align*}

\begin{itemize}
    \item Now let $\theta > 0, $ i.e. $\gamma^2 - 4\alpha > 0.$ In this case $\lambda_1 + \lambda_2^* = - \gamma = \lambda_1^* + \lambda_2, \ \lambda_1 \lambda_2^* = \lambda_1^* \lambda_2 = \alpha$ and \[0 > \lambda_2 + \lambda_2^* = - \gamma + \theta \geq - \gamma  = \lambda_2 + \lambda_1^* = \lambda_1 + \lambda_2^*  \geq  - \gamma - \theta = \lambda_1 + \lambda_1^*\,.\]
    Using that $|\lambda_1|^2 + |\lambda_2|^2 + 1+ \alpha^2 = \gamma^2 + (1-\alpha)^2$ and omitting the terms with negative sign, we find
    \begin{align*}
        \bigl\| e^{ K t} \bigr\|_F^2 \leq 2e^{(-\gamma + \theta) t} \frac{\gamma^2 + (1-\alpha)^2}{|\theta|^2} \leq C e^{-(\gamma-\theta)t},
    \end{align*}
    where $C>0$. 
    \item Next, consider the case when $\gamma^2-4\alpha <0$, i.e.\ $\theta = i|\theta|$ is purely complex. First, compute $\lambda_1 + \lambda_1^* = \lambda_2 + \lambda_2^* = - \gamma, \ \lambda_1 + \lambda_2^* = - \gamma - 2i|\theta|,\ \lambda_1^* + \lambda_2 = - \gamma + 2i|\theta|, \ \lambda_1 \lambda_2^* = - \alpha + \frac{1}{2}\gamma^2 + \frac{1}{2} i \gamma |\theta|, \ \lambda_1^* \lambda_2 = - \alpha + \frac{1}{2}\gamma^2 - \frac{1}{2} i \gamma |\theta|, \ |\lambda_1|^2 = |\lambda_2|^2 = \alpha $ so that $|\lambda_1|^2 + |\lambda_2|^2 + 1 + \alpha^2 = (1+\alpha)^2$ and $ \lambda_1^* \lambda_2 + \lambda_2^* \lambda_1 + 1 + \alpha^2 = - |\theta|^2 + (1+\alpha)^2.$ This yields that
    \begin{align*}
        \|e^{Kt}\|_F^2 = e^{-\gamma t} \frac{2(1+\alpha)^2}{|\theta|^2} + e^{-(\gamma + 2i|\theta|)t } + e^{-(\gamma - 2i|\theta|)t } \leq 2 e^{-\gamma t} \left(1+ \frac{(1+\alpha)^2}{|\theta|^2}  \right)
    \end{align*}
\end{itemize}

\noindent\textbf{Case $\theta=0$.}
In this case according to \eqref{eq:expA-theta0} we have
\begin{align*}
    e^{Kt} = e^{-\frac{\gamma}{2}t}\begin{pmatrix}
        1 + \frac{\gamma}{2}t & t \\ -\alpha t & 1-\frac{\gamma}{2} t 
    \end{pmatrix}
\end{align*}
and therefore
\begin{align*}
    \| e^{Kt}\|_F^2 = e^{-\gamma t}\biggl[2+ \left(\frac{\gamma^2}{2} + 1 + \alpha^2\right)t^2\biggr] \leq C_\delta e^{-(\gamma-\delta)t},
\end{align*}
where $\delta>0$ is fixed and  $C_\delta >0$ depends only on $\delta.$ The last inequality uses similar arguments as in the proof of  \cite[Corollary 3.7]{HartmannNeureitherSharma25a}. 
 
Summarising the three cases above, we yield the claimed result. 
\end{proof}


In the following we combine the results of Lemma \ref{lem:matrixexp} and \ref{lem:matrixexp-b} to be able to compute explicit convergence rates for $Q^2_t, P^2_t$ defined in \eqref{eq:Q2P2} to the limit dynamics $q_t,p_t$ defined in~\eqref{eq:limit-gen}.

\begin{theorem} \label{cor:gronwall}
    Given $\eps>0$, let $(Q^2_t,P^2_t)$  solve~\eqref{eq:Q2P2} with initial datum $(Q^2_0,P^2_0)$, and let  
    $(q_t,p_t)$ solve~\eqref{eq:limit-gen} with initial datum $(q_0,p_0)$. Recall the definition of the potential $V$ from~\eqref{eq:Valpha} and the parameter $\alpha\geq0$.  
    We have the $L^2$-estimate
    \begin{align}\label{eq:gronwall-L2}
        \E \biggl[ \sup\limits_{t \in [r,T]} \biggl| 
        \begin{pmatrix}
           q_t - Q^{2}_t \\ p_t - P^{2}_t
        \end{pmatrix} \biggr|^2 \biggr] &\leq 
        \begin{dcases}
        C e^{\tilde C T}  \biggl( \left| \begin{pmatrix}
            q_0 - Q^2_0 \\ p_0 - P^2_0
        \end{pmatrix} \right|^2 + T\int_0^T  \E\bigl[|\hat q - Q^{1}_s|^2 \bigr]ds \biggr) 
        \ \ &\text{if $\alpha=0$},\\
        \tilde C \biggl(e^{-\eta r} \biggl|\begin{pmatrix}
           q_0 - Q^{2}_0 \\ p_0 - P^{2}_0
        \end{pmatrix} \biggr|^2+ \eta^{-1} \int_0^T  \E\left[|  \hat q - Q^{1}_s |^2\right] ds \biggr) &\text{if $\alpha>0$}. 
        \end{dcases}
    \end{align}
    Additionally  for $\alpha >0$  we have $L^1$ estimates 
    \begin{align} \label{eq:grownall-alphapos}
   \E\biggl[ \sup\limits_{t \in [r,T]} \biggl| \begin{pmatrix}
        q_t - Q^2_t \\ p_t - P^2_t    \end{pmatrix} \biggr| \biggr] &\leq C  \biggl(e^{- \frac{\eta}{2} r} \biggl| \begin{pmatrix}
            q_0 - Q^2_0 \\ p_0 - P^2_0
        \end{pmatrix} \biggr| +  \int_0^T e^{-\frac{\eta}{2}(T-s)} \E\left[|\hat q - Q^{1}_s|\right] ds \biggr) \,,
    \end{align}
Here $\eta\in (0,\gamma]$ defined in Lemma~\ref{lem:matrixexp-b} and $C, \tilde C>0$ is independent of $t,\eps$. 
\end{theorem}
\begin{proof}
We will proceed in two steps. In the first step, we will show that 
    \begin{equation*}
        \biggl|
        \begin{pmatrix}
           q_t - Q^{2}_t \\ p_t - P^{2}_t
        \end{pmatrix} \biggr| \leq \bigl\|e^{Kt}\bigr\|_F\left|\begin{pmatrix}
           q_0 - Q^{2}_0 \\ p_0 - P^{2}_0
        \end{pmatrix} \right|+L_U \int_0^t \bigl\|e^{K(t-s)}\bigr\|_F  |  \hat q - Q^{1}_s | ds + \int_0^t \bigl\|e^{K(t-s)}\bigr\|_F  \left|
        \begin{pmatrix}
           q_s - Q^{2}_s \\ p_s - P^{2}_s
        \end{pmatrix} \right| ds,
    \end{equation*}    
    where $L_U$ is the Lipschitz constant for $\nabla U(q)$ (recall~\eqref{ass:VLip}) and $K$ characterises the linear part of the dynamics (see below).
    In the second step, we will use the estimates for $\|e^{Kt}\|_F^2 $ derived in Lemma~\ref{lem:matrixexp} and Lemma~\ref{lem:matrixexp-b} for the case $\alpha = 0$ and $\alpha >0$ respectively, which together with the Gronwall's inequality yield the final result. 
    
    \noindent\textbf{Step 1.} 
    Using $K = \begin{pmatrix}
        0 & I \\ -\alpha I  & - \gamma  I\end{pmatrix}$, where $0,I\in \R^{(d-k) \times (d-k)}$, and  Proposition~\ref{prop:varofconst} we find 
    \begin{align*}
        \begin{pmatrix}
            q_t \\ p_t
        \end{pmatrix} &= e^{Kt}\begin{pmatrix}
            q_0 \\ p_0        \end{pmatrix} + \int_0^t e^{K(t-s)}\begin{pmatrix}
           0 \\ \nabla_{q^2} U(\hat q,q_s)
        \end{pmatrix} ds  + \sqrt{2 \gamma\beta^{-1} }\int_0^t e^{K(t-s)}
        \begin{pmatrix}
            0 \\ dW^2_s
        \end{pmatrix} , \\ 
        \begin{pmatrix}
            Q^{2}_t \\ P^{2}_t
        \end{pmatrix} &= e^{Kt}\begin{pmatrix}
            Q^{2}_0 \\ P^{2}_0        
            \end{pmatrix} + \int_0^t e^{K(t-s)}\begin{pmatrix}
           0 \\ \nabla_{q^2} U(Q^{1}_s,Q^{2}_s)
        \end{pmatrix} ds  + \sqrt{2 \gamma\beta^{-1} }\int_0^t e^{K(t-s)}\begin{pmatrix}
           0 \\ dW^2_s
        \end{pmatrix} ,
    \end{align*}
    which leads to 
    \begin{align}
        \left|
        \begin{pmatrix}
           q_t - Q^{2}_t \\ p_t - P^{2}_t
        \end{pmatrix} \right| 
        &\leq  
        \left|e^{Kt}\begin{pmatrix}
           q_0 - Q^{2}_0 \\ p_0 - P^{2}_0
        \end{pmatrix} \right|+
        \int_0^t \left|e^{K(t-s)}\begin{pmatrix}
           0 \notag \\ \nabla_{q^2} U(\hat q,q_s)  - \nabla_{q^2} U(Q^{1}_s,Q^{2}_s) 
        \end{pmatrix}  \right|ds 
        \\ 
        & \leq  \bigl\|e^{Kt}\bigr\|_F\left|\begin{pmatrix}
           q_0 - Q^{2}_0 \\ p_0 - P^{2}_0
        \end{pmatrix} \right|+L_U \int_0^t \bigl\|e^{K(t-s)}\bigr\|_F \, \bigl|(\hat q,q_s)^\top - (Q^{1}_s,Q^{2}_s)^\top\bigr| ds \notag \\
        & \leq  \bigl\|e^{Kt}\bigr\|_F\left|\begin{pmatrix}
           q_0 - Q^{2}_0 \\ p_0 - P^{2}_0
        \end{pmatrix} \right|+L_U \int_0^t \bigl\|e^{K(t-s)}\bigr\|_F  |  \hat q - Q^{1}_s | ds + \int_0^t \bigl\|e^{K(t-s)}\bigr\|_F  \left|
        \begin{pmatrix}
           q_s - Q^{2}_s \\ p_s - P^{2}_s
        \end{pmatrix} \right| ds , \label{eq:step1proofgron}
    \end{align} 
    where we used the Lipschitz continuity of $\nabla U$ to arrive at the second inequality and the third inequality
    follows by adding $p_s - P^2_s$. 

\noindent\textbf{Step 2a ($\alpha = 0$).} 
By Lemma \ref{lem:matrixexp} we have that $\|e^{Kt}\|_F^2 \leq C$, where $C > 0$ depends on the dimensions $d,k$ and the friction coefficient $\gamma$, but not on $\eps,t$. Therefore, in this case~\eqref{eq:step1proofgron} leads to
\begin{align*}
        \left|
        \begin{pmatrix}
           q_t - Q^{2}_t \\ p_t - P^{2}_t
        \end{pmatrix} \right| 
         \leq  C \left(\left|\begin{pmatrix}
           q_0 - Q^{2}_0 \\ p_0 - P^{2}_0
        \end{pmatrix} \right|+L_U \int_0^t   |  \hat q - Q^{1}_s | ds + \int_0^t   \left|
        \begin{pmatrix}
           q_s - Q^{2}_s \\ p_s - P^{2}_s
        \end{pmatrix} \right| ds \right),
    \end{align*}     
   We set 
   \[\delta(t) \coloneqq C \left(\left|\begin{pmatrix}
           q_0 - Q^{2}_0 \\ p_0 - P^{2}_0
        \end{pmatrix} \right|+L_U \int_0^t   |  \hat q - Q^{1}_s | ds \right) 
        \]
        and note that $t\mapsto \delta(t)$ is monotonically increasing. Then by Gronwall's lemma we arrive at 
        \begin{align*}
            \left|
        \begin{pmatrix}
           q_t - Q^{2}_t \\ p_t - P^{2}_t
        \end{pmatrix} \right| 
         \leq  \delta(t)e^{Ct} = C e^{Ct}\biggl(\left|\begin{pmatrix}
           q_0 - Q^{2}_0 \\ p_0 - P^{2}_0
        \end{pmatrix} \right|+L_U \int_0^t   |  \hat q - Q^{1}_s | ds\biggr).
        \end{align*}
        Squaring this inequality and applying Young's inequality to the right-hand side we find
        \begin{align*}
            \left|
        \begin{pmatrix}
           q_t - Q^{2}_t \\ p_t - P^{2}_t
        \end{pmatrix} \right|^2 
         \leq 2C e^{2Ct}\biggl(\left|\begin{pmatrix}
           q_0 - Q^{2}_0 \\ p_0 - P^{2}_0
        \end{pmatrix} \right|^2+L_U^2 t \int_0^t   |  \hat q - Q^{1}_s |^2 ds\biggr)
        \end{align*}
        where we have used the Cauchy Schwarz inequality to bound the square of the time-integral on the right-hand side. 
        Finally, taking the supremum over $t \in [0,T]$, followed by the expectation yields 
        the required identity~\eqref{eq:gronwall-L2}. 

        \noindent\textbf{Step 2b ($\alpha > 0$).} Lemma~\ref{lem:matrixexp-b} states that $\|e^{Kt} \|^2_F \leq C e^{-\eta t}$ for $\alpha >0$, where 
        $\eta\in(0,\gamma]$ and $C>0$ does not depend on $t,\eps$.
        Then for $t\geq r\geq 0$,~\eqref{eq:step1proofgron} leads to
        \begin{align*}
        \left|
        \begin{pmatrix}
           q_t - Q^{2}_t \\ p_t - P^{2}_t
        \end{pmatrix} \right| 
        &\leq  
         C \biggl( e^{-\frac{\eta}{2} r } \left|\begin{pmatrix}
           q_0 - Q^{2}_0 \\ p_0 - P^{2}_0
        \end{pmatrix} \right|+L_U \int_0^t e^{-\frac{\eta}{2} (t-s)} |  \hat q - Q^{1}_s | ds + \int_0^t e^{-\frac{\eta}{2}(t-s)} \left|
        \begin{pmatrix}
           q_s - Q^{2}_s \\ p_s - P^{2}_s
        \end{pmatrix} \right| ds \biggr).
    \end{align*} 
    We set 
    \begin{align*}
        \delta(t) = C \biggl( e^{-\frac{\eta}{2} r } \left|\begin{pmatrix}
           q_0 - Q^{2}_0 \\ p_0 - P^{2}_0
        \end{pmatrix} \right|+L_U \int_0^t e^{-\frac{\eta}{2} (t-s)} |  \hat q - Q^{1}_s | ds \biggr) \text{ and } \nu(s) = e^{-\frac{\eta}{2}(t-s)}
    \end{align*}
    and note that $t\mapsto \delta(t)$ is monotonically increasing. With $\int_0^t e^{-\frac{\eta}{2}(t-s)} ds = 2\eta^{-1} (1- e^{-\frac{\eta}{2} t})\leq 2\eta^{-1}$ and using Gronwall's inequality we find
     \begin{align}\label{eq:L1-gronwall}
            \left|
        \begin{pmatrix}
           q_t - Q^{2}_t \\ p_t - P^{2}_t
        \end{pmatrix} \right| 
         \leq  \delta(t)e^{\int_0^t \nu(s) ds} \leq C e^{2\eta^{-1}}\biggl(e^{-\frac{\eta}{2} r}\left|\begin{pmatrix}
           q_0 - Q^{2}_0 \\ p_0 - P^{2}_0
        \end{pmatrix} \right|+L_U \int_0^t e^{-\frac{\eta}{2}(t-s)}  |  \hat q - Q^{1}_s | ds\biggr).
        \end{align}
        Taking the supremum over $t \in [r,T]$, followed by the expectation, the required $L^1$-estimate~\eqref{eq:grownall-alphapos} follows by using Fubini's Theorem. The $L^2$-estimate follows by squaring the above inequality, which using Young's and Cauchy-Schwarz inequality yields that 
        \begin{align*}
            \left|
        \begin{pmatrix}
           q_t - Q^{2}_t \\ p_t - P^{2}_t
        \end{pmatrix} \right|^2 
         \leq C \biggl(e^{-\eta r}\left|\begin{pmatrix}
           q_0 - Q^{2}_0 \\ p_0 - P^{2}_0
        \end{pmatrix} \right|^2+L_U^2 \int_0^t e^{-\eta(t-s)} ds \int_0^t |  \hat q - Q^{1}_s |^2 ds\biggr).
        \end{align*}
        Computing the exponential integral term, taking the supremum of $t \in [r,T]$ and finally taking the expectation (and using Fubini) yields the result \eqref{eq:gronwall-L2}. 
        Let us mention the following issue at this point: squaring ~\eqref{eq:L1-gronwall} and taking the supremum for $t \in [0,T]$ suggests we can also get a uniform in time estimate in $L^2$. This does not work since taking the expectation and squaring the integral term doesn't commute.
\end{proof}

\subsection{Spatially confined Langevin dynamics}

In this section we study the spatially confined Langevin dynamics.

\begin{proof}[Proof of Proposition~\ref{prop:SpatCons}.] \label{proof:spat}
Recall the spatially-confined Langevin dynamics~\eqref{eq:abc-slow},\eqref{eq:intro-Spat-Cons-Fast}  
with $Q^{1}_t,P^{1}_t\in\R^k$ and $Q^{2}_t,P^{2}_t \in \R^{d-k}$. 
Using the notation $Q^{1}_t = \bigl( (Q^{1}_t)^1,\ldots,(Q^{1}_t)^k\bigr)$, $P^{1}_t = \bigl( (P^{1}_t)^1,\ldots,(P^{1}_t)^k\bigr)$ for the coordinates of $Q^{1}_t \in\R^k$, $P^{1}_t\in\R^k$, it is easily checked that each pair $((Q^{1}_t)^i,(P^{1}_t)^i)$ in $\R^2$ for $i\in\{1,\ldots,k\}$ satisfies the auxiliary Langevin form~\eqref{eq:lemsollangevin} where $\tilde q_t \in \R^{d-1}$ are the remaining spatial variables and the constants are given by 
\begin{equation*}
    a= 1, \ \ b = \alpha +
    \frac{1}{\eps}, \ \ c=\gamma, \ \ \tilde c= \sqrt{2\gamma\beta^{-1}}.
\end{equation*}
In the following we present the proof for the case that $V(q)$ does not have a quadratic term in $q^1$, i.e.\ $\alpha=0$ in~\eqref{eq:Valpha}. We point out that the case $\alpha>0$ follows similarly, as the only difference in the parameters above is that $b$ changes  which will not change the final results, in particular since $\theta = \sqrt{\gamma^2 - \frac{4(\eps+\alpha)}{\eps}}$ has the same scaling behaviour as $\theta = \sqrt{\gamma^2 - 4\eps^{-1}}$ in the presentation below. The only other difference in the proof $\alpha>0$ would be that $\nabla V$ is replaced by $\nabla U$ in the following. 

Note that for $\eps$ small enough (specifically $\eps<4\gamma^{-2}$), $\theta =\sqrt{c^2-4ab}$ is imaginary and the solution is given by Corollary~\ref{cor:sol_underdamped}. Using $\theta = i|\theta| = i\sqrt{\frac{4}{\eps} -\gamma^2}$ and collecting these solutions in vectorial form we find    \begin{subequations}
\begin{align}
\begin{split}\label{eq:Q1-sol}
    \begin{aligned}
    Q^{1}_t &= \Bigl[\cos\Bigl(t\mfrac{|\theta|}{2}\Bigr)+ \mfrac{\gamma}{|\theta|} \sin\Bigl(t\mfrac{|\theta|}{2}\Bigr) \Bigr]e^{-\frac{\gamma}{2}t}Q^{1}_0
    + \mfrac{2}{|\theta|} \sin\Bigl(t\mfrac{|\theta|}{2}\Bigr)e^{-\frac\gamma2 t}P^{1}_0 \\
    &\qquad -\mfrac{2}{|\theta|}\int_0^t  \sin\Bigl( (t-s)\mfrac{|\theta|}{2}\Bigr) e^{-\frac{\gamma}{2}(t-s)}\nabla_{q^1} V(Q_s)\, ds + \mfrac{2\sqrt{2\gamma\beta^{-1}}}{|\theta|} \int_0^t  \sin\Bigl((t-s) \mfrac{|\theta|}{2}\Bigr) e^{-\frac{\gamma}{2}(t-s)} dW^1_s,
    \end{aligned}
\end{split}\\
\begin{split}\label{eq:P1-sol}
    \begin{aligned}
    P^{1}_t &= -\mfrac{2}{\eps|\theta|} \sin\Bigl(t\mfrac{|\theta|}{2}\Bigr)e^{-\frac{\gamma}{2}t}Q^{1}_0+ \Bigl[ \cos\Bigl(t\mfrac{|\theta|}{2}\Bigr) - \frac{\gamma}{|\theta|}\sin\Bigl(t \mfrac{|\theta|}{2}\Bigr) \Bigr]e^{-\frac{\gamma}{2}t}P^{1}_0 \\
    &\qquad - \int_0^t \Bigl[ \cos\Bigl((t-s)\mfrac{|\theta|}{2}\Bigr) -\mfrac{\gamma}{|\theta|} \sin\Bigl((t-s)\mfrac{|\theta|}{2}\Bigr) \Bigr] \nabla_{q^1} V(Q_s) e^{-\frac{\gamma}{2}(t-s)}ds \\ 
    &\qquad + \sqrt{2\gamma\beta^{-1}} \int_0^t \Bigl[ \cos\Bigl((t-s)\mfrac{|\theta|}{2}\Bigr) -\mfrac{\gamma}{|\theta|} \sin\Bigl((t-s)\mfrac{|\theta|}{2}\Bigr) \Bigr] e^{-\frac{\gamma}{2}(t-s)}dW^1_s.
    \end{aligned}
\end{split}
\end{align}
\end{subequations}
We now prove part~\ref{item:spat-Q1} of Proposition~\ref{prop:SpatCons}. Following the procedure outlined at the start of this section (recall~\eqref{eq:genPathBound}), which involves using Young's inequality, It\^o isometry, and the bound $|\nabla_{q^1} V|=|\nabla_{q^1} U|\leq C_U$ (recall~\eqref{ass:Vbound} and note that $\alpha=0$), we find for any $t\geq 0$ 
\begin{align}
    \E\Bigl[&|Q_t^{1}|^2\Bigr] \leq C \biggl\{ \E\Bigl[ \Bigl|\Bigl[\cos\Bigl(t\mfrac{|\theta|}{2}\Bigr)+ \mfrac{\gamma}{|\theta|} \sin\Bigl(t\mfrac{|\theta|}{2}\Bigr) \Bigr]e^{-\frac{\gamma}{2}t}Q^{1}_0
    + \mfrac{2}{|\theta|} \sin\Bigl(t\mfrac{|\theta|}{2}\Bigr)e^{-\frac\gamma2 t}P^{1}_0\Bigr|^2 \Bigr]\notag\\ 
    & \quad + \biggl(\mfrac{1}{|\theta|} \int_0^t \Bigl|\sin\Bigl( (t-s)\mfrac{|\theta|}{2}\Bigr)\Bigr| e^{-\frac{\gamma}{2}(t-s)}ds\biggr)^2 
    + \frac{\beta^{-1}\gamma}{|\theta|^2} \int_0^t \Bigl|\sin\Bigl( (t-s)\mfrac{|\theta|}{2}\Bigr) e^{-\frac{\gamma}{2}(t-s)}\Bigr|^2 ds
    \biggr\} \notag\\
     &\leq C \biggl\{ e^{-\gamma r}\Bigl(1+ \mfrac{\gamma}{|\theta|}\Bigr)^2 \E\bigl[|Q_0^{1}|^2\bigr] + \mfrac{e^{-\gamma r}}{|\theta|^2}  \E\bigl[|P_0^{1}|^2\bigr] 
     + \mfrac{1}{|\theta|^2} \biggr\}, \label{eq:Spat-Q1-Aux1}
\end{align}
where the bound on the initial data terms follows by using Young's inequality and the last two integrals were bounded using $|\sin(\cdot)|\leq 1$ and explicitly computing the integrals, which gives
\begin{align*}
    &\biggl(\mfrac{1}{|\theta|} \int_0^t \Bigl|\sin\Bigl( (t-s)\mfrac{|\theta|}{2}\Bigr)\Bigr| e^{-\frac{\gamma}{2}(t-s)}ds\biggr)^2 
    + \mfrac{\beta^{-1}\gamma}{|\theta|^2} \int_0^t \Bigl|\sin\Bigl( (t-s)\mfrac{|\theta|}{2}\Bigr) e^{-\frac{\gamma}{2}(t-s)}\Bigr|^2 ds \\
    &\leq \biggl(\mfrac{1}{|\theta|} \int_0^t e^{-\frac{\gamma}{2}(t-s)}ds\biggr)^2 
    + \frac{\beta^{-1}\gamma}{|\theta|^2} \int_0^t  e^{-\gamma(t-s)} ds  \leq C \biggl\{ \mfrac{1}{|\theta|^2\gamma^2} (1-e^{-\frac{\gamma}{2}t})^2 + \mfrac{1}{|\theta|^2}(1-e^{-\gamma t})\biggr\} \leq \mfrac{C}{|\theta|^2}.
\end{align*}

Since $|\theta|= \bigO  (\eps^{-\frac12})$ as $\eps\to 0$, for small enough $\eps>0$ we arrive at
\begin{equation*}
\E\Bigl[|Q_t^{1}|^2\Bigr]   \leq  C \biggl( e^{-\gamma t} \Bigl[\E\bigl[|Q_0^{1}|^2\bigr] + \eps \E\bigl[|P_0^{1}|^2\bigr]\Bigr]+ \eps \biggr) \xrightarrow{\eps\to 0} 0 \ \ \text{if} \ \ \lim_{\eps\to 0}\E[|Q^{1}_0|^2]=0 . 
\end{equation*}
where $C=C(C_U,\gamma,\beta)>0$ is independent of $\eps$ and $t$. 

Next we discuss~\ref{item:spat-P1}.  
Integrating~\eqref{eq:P1-sol} we have
\begin{align}\label{eq:Int-p}
\begin{aligned}
    \int_0^r P_t^{1} dt &= -\mfrac{2}{\eps|\theta|} \int_0^r  \sin\Bigl(t\mfrac{|\theta|}{2}\Bigr)e^{-\frac{\gamma}{2}t}Q^{1}_0 dt+ \int_0^r \Bigl[ \cos\Bigl(t\mfrac{|\theta|}{2}\Bigr) - \frac{\gamma}{|\theta|}\sin\Bigl(t \mfrac{|\theta|}{2}\Bigr) \Bigr]e^{-\frac{\gamma}{2}t}P^{1}_0 dt \\
    &\qquad - \int_0^r  \Bigr\{ \int_0^t \Bigl[ \cos\Bigl((t-s)\mfrac{|\theta|}{2}\Bigr) -\mfrac{\gamma}{|\theta|} \sin\Bigl((t-s)\mfrac{|\theta|}{2}\Bigr) \Bigr]  e^{-\frac{\gamma}{2}t} \Bigr\}\nabla_{q^1} V(Q_s) e^{\frac{\gamma}{2}s}ds dt \\ 
    &\qquad + \sqrt{2\gamma\beta^{-1}} \int_0^r \Bigl\{\int_0^t  \Bigl[ \cos\Bigl((t-s)\mfrac{|\theta|}{2}\Bigr) -\mfrac{\gamma}{|\theta|} \sin\Bigl((t-s)\mfrac{|\theta|}{2}\Bigr) \Bigr] e^{-\frac{\gamma}{2}t}  \Bigr\}e^{\frac{\gamma}{2}s} dW^1_s dt\\ &\eqqcolon I^\eps_1+I^\eps_2+I^\eps_3+I^\eps_4,
\end{aligned}
\end{align}
where $I^\eps_i$ refers to the $i$-th integral on the right-hand side of~\eqref{eq:Int-p}.

We will make use of the following integral identities in our proof 
\begin{align}\label{eq:aux-int}
\begin{aligned}
    \int_u^r \cos\Bigl((t-s)\mfrac{|\theta|}{2} \Bigr) e^{-\frac{\gamma}{2}t} dt &= \mfrac{2}{\gamma ^2+|\theta| ^2}
    \Bigl\{-e^{-\frac{\gamma}{2}r } \Bigl[\gamma  \cos \Bigl((s-r)\mfrac{|\theta|}{2}\Bigr)+|\theta|  \sin \Bigl((s-r)\mfrac{|\theta|}{2}\Bigr)\Bigr]\\ 
    & \qquad\qquad\qquad+ e^{-\frac{\gamma}{2}u}\Bigl[\gamma  \cos \Bigl((s-u)\mfrac{|\theta|}{2}\Bigr)+|\theta|  \sin \Bigl((s -u) \mfrac{|\theta|}{2}\Bigr) \Bigr]\Bigr\} \\
    \int_u^r \sin\Bigl((t-s)\mfrac{|\theta|}{2} \Bigr) e^{-\frac{\gamma}{2}t} dt &=
    \mfrac{2}{\gamma ^2+|\theta|^2} \Bigl\{e^{-\frac{\gamma}{2}r } \Bigl[\gamma  \sin \Bigl((s-r)\mfrac{|\theta|}{2}\Bigr)-|\theta|  \cos \Bigl((s-r)\mfrac{|\theta|}{2}\Bigr)\Bigr]\\ 
    & \qquad\qquad\qquad + e^{-\frac{\gamma}{2}u}\Bigl[|\theta|  \cos \Bigl((s-u)\mfrac{|\theta|}{2}\Bigr) -\gamma  \sin \Bigl((s-u)\mfrac{|\theta|}{2} \Bigr]\Bigr)\Bigr\}
\end{aligned}
\end{align}
We then have
\begin{align*}
    |I_1^\eps + I_2^\eps| &= \biggl|-\frac{4 |\theta|-4 e^{-\frac{\gamma}{2}   r} \bigl[\gamma  \sin \left(r\mfrac{|\theta|}{2}\right)+|\theta| \cos \left(r\mfrac{|\theta|}{2}\right)\bigr]}{\eps|\theta|   (|\theta|^2+\gamma ^2)} Q_0^{1} + \frac{2 }{|\theta|}e^{-\frac{\gamma}{2}  r}\sin \Bigl(r\mfrac{|\theta|}{2}\Bigr) P^{1}_0\biggr| \\
    &\leq \bigO(\eps) |Q_0^{1}| + \bigO(\eps^{\frac{1}{2}}) |P_0^{1}|
\end{align*}

Next consider $I_3^\eps$ where we interchange the integrals to arrive at
\begin{align*}
    I_3^\eps &=   \int_0^r  \Bigr\{ \int_s^r \Bigl[ \cos\Bigl((t-s)\mfrac{|\theta|}{2}\Bigr) -\mfrac{\gamma}{|\theta|} \sin\Bigl((t-s)\mfrac{|\theta|}{2}\Bigr) \Bigr]  e^{-\frac{\gamma}{2}t} dt \Bigr\} \, \nabla_{q^1}  V(Q_s) e^{\frac{\gamma}{2}s} ds  \\
    &=   \int_0^r \Bigl( \mfrac{2}{\gamma ^2+|\theta| ^2}
    \Bigl\{-e^{-\frac{\gamma}{2}r } \Bigl[\gamma  \cos \Bigl((s-r)\mfrac{|\theta|}{2}\Bigr)+|\theta|  \sin \Bigl((s-r)\mfrac{|\theta|}{2}\Bigr)\Bigr] + e^{-\frac{\gamma}{2}s}\gamma  \Bigr\} \\ & \qquad -\mfrac{\gamma}{|\theta|}    \mfrac{2}{\gamma ^2+|\theta|^2} \Bigl\{e^{-\frac{\gamma}{2}r } \Bigl[\gamma  \sin \Bigl((s-r)\mfrac{|\theta|}{2}\Bigr)-|\theta|  \cos \Bigl((s-r)\mfrac{|\theta|}{2}\Bigr)\Bigr] + e^{-\frac{\gamma}{2}s}|\theta| \Bigr\}\Bigr)  
    \, \nabla_{q^1}  V(Q_s) e^{\frac{\gamma}{2}s} ds  \\
    & = \int_0^r \mfrac{2}{\gamma^2 + |\theta|^2} e^{-\frac{\gamma}{2}r} \sin \Bigl((s-r) \mfrac{|\theta|}{2} \Bigr)\left[-|\theta| - \frac{\gamma^2}{|\theta|} \right] \nabla_{q^1} V(Q_s) ds \\
    &= - \mfrac{2}{|\theta|} \int_0^r e^{-\frac{\gamma}{2}r} \sin \Bigl((s-r) \mfrac{|\theta|}{2} \Bigr) \nabla_{q^1} V(Q_s) ds, 
\end{align*}
where we used the explicit expressions for the integrals above. Again using $|\nabla V|=|\nabla U|\leq C_U$ we find 
\begin{equation*}
    |I_3^\eps| \leq \mfrac{2}{|\theta|} \int_0^r e^{-\frac{\gamma}{2}r} ds = \frac{2}{|\theta|} r e^{-\frac{\gamma}{2}r} \leq C\eps^{-\frac12}, 
\end{equation*}
and therefore $I_3^\eps \to 0$ as $\eps \to 0$.

Let us continue with $I_4^\eps$. 
We first interchange the order of integration
and then use \eqref{eq:aux-int} 
\begin{align*}
    I_4^\eps 
    &=  \sqrt{2\gamma\beta^{-1}} \int_0^r \int_s^r \Bigl[ \cos\Bigl((t-s)\mfrac{|\theta|}{2}\Bigr) -\mfrac{\gamma}{|\theta|} \sin\Bigl((t-s)\mfrac{|\theta|}{2}\Bigr) \Bigr] e^{-\frac{\gamma}{2}t}  dt \ e^{\frac{\gamma}{2}s} dW^1_s \\
    &= \sqrt{2\gamma\beta^{-1}} \int_0^r   \mfrac{2}{\gamma ^2+|\theta| ^2}
    \Bigl\{-e^{-\frac{\gamma}{2}(r-s) } \Bigl[\gamma  \cos \Bigl((s-r)\mfrac{|\theta|}{2}\Bigr)+|\theta|  \sin \Bigl((s-r)\mfrac{|\theta|}{2}\Bigr)\Bigr] + \gamma   \Bigr\}  \\ 
    & \qquad\qquad\qquad -\mfrac{\gamma}{|\theta|} \mfrac{2}{\gamma ^2+|\theta|^2} \Bigl\{e^{-\frac{\gamma}{2}(r-s) } \Bigl[\gamma  \sin \Bigl((s-r)\mfrac{|\theta|}{2}\Bigr)-|\theta|  \cos \Bigl((s-r)\mfrac{|\theta|}{2}\Bigr)\Bigr] + |\theta|\Bigr)\Bigr\} \ dW^1_s  \\
    &= \sqrt{2\gamma\beta^{-1}} \int_0^r   \mfrac{2}{\gamma ^2+|\theta| ^2}
    e^{-\frac{\gamma}{2}r }  \sin \Bigl((s-r)\mfrac{|\theta|}{2}\Bigr) (- |\theta| - \frac{\gamma^2}{|\theta|}) dW^1_s \\
    & = - \sqrt{2\gamma\beta^{-1}} \int_0^r  \mfrac{2}{|\theta|} e^{-\frac{\gamma}{2}r}  \sin \Bigl((s-r)\mfrac{|\theta|}{2}\Bigr) dW^1_s
\end{align*}
and using dominated convergence theorem for stochastic integrals~\cite[Corollary 18.13]{Kallenberg21} it follows that $|I_4^\eps| \to 0$ in probability.

\end{proof}
    
\begin{proof}[Proof of Proposition~\ref{prop:Behav-Q2P2}] \label{proof:spatslow}
This result makes use of Theorem \ref{cor:gronwall}.
For the case $\alpha = 0$ in $V(q)$ (recall definition~\eqref{eq:Valpha}), we use \eqref{eq:gronwall-L2} with $\hat q = 0$, which in turn requires the following bound which uses~\eqref{eq:spat-point-bound} 
\begin{align*}
        \int_0^T \E\bigl[\bigl| Q^{1}_s\bigr|^2 \bigr] ds   
        &\leq C\int_0^T \Bigl[ e^{-\gamma s}\Bigl(\E\bigl[|Q_0^{1}|^2\bigr] + \eps \E\bigl[|P_0^{1}|^2\bigr] \Bigr) + \eps \Bigr]ds
        \leq C\Bigl[\gamma^{-1}\Bigl( \E\bigl[|Q_0^{1}|^2\bigr] + \eps \E\bigl[|P_0^{1}|^2\bigr] \Bigr)+ \eps T \Bigr] ,
\end{align*}
which yields the required result~\eqref{eq:Spat-fast}.
For the case $\alpha > 0$ we use \eqref{eq:grownall-alphapos}. First note that by Jensen's inequality  $ \E\bigl[|Q_t^{1}|\bigr]^2  \leq \E\bigl[|Q_t^{1}|^2\bigr]$ and hence \eqref{eq:spat-point-bound} implies 
\begin{align*}
    \E\bigl[|Q_t^{1}|\bigr] \leq \sqrt{ C_1 \eps +  C_2 e^{- \gamma t} \left(|Q^1_0|^2 + \eps|P^1_0|^2\right) } \leq  \sqrt{C_1} \sqrt{ \eps} +  \sqrt{C_2} \left(|Q^1_0| + \sqrt{\eps} |P^1_0|\right),
\end{align*}
which leads to (note $\hat q=0$)
\begin{align*}
    \int_0^t e^{-\frac{\eta}{2}(t-s)}  \E\left[|Q^1_s|\right]  ds &\leq C\left(  \sqrt{ \eps} +  |Q^1_0| + \sqrt{\eps} |P^1_0|\right) \int_0^t e^{-\frac{\eta}{2} (t-s)} ds\leq  \frac{2C}{\eta} \bigl( \sqrt{\eps}  + |Q^1_0| + \sqrt{\eps}|P^1_0|\bigr).
\end{align*}
Inserting the last estimate into \eqref{eq:grownall-alphapos} the claim \eqref{eq:Spat-fast-2} follows. 
    The proof of \eqref{eq:Spat-fast} follows again with Theorem \ref{cor:gronwall}, specifically equation~\eqref{eq:gronwall-L2} together with the time-$t$-bound as given in~\eqref{eq:spat-point-bound}.
\end{proof}

\subsection{Phase-space confined Langevin dynamics}\label{App:PS-cons-proof}

We now study the phase-space confined Langevin dynamics. 

\begin{proof}[Proof of Proposition~\ref{prop:PS-Cons}] \label{proof:phase-space}
Recall that the phase-space confined Langevin dynamics~\eqref{eq:abc-slow},\eqref{eq:intro-Phase-Cons-Fast} with 
$Q^{1}_t,P^{1}_t\in\R^k$ and $Q^{2}_t,P^{2}_t \in \R^{d-k}$. 
Using the notation $Q^{1}_t = \bigl( (Q^{1}_t)^1,\ldots,(Q^{1}_t)^k\bigr)$, $P^{1}_t = \bigl( (P^{1}_t)^1,\ldots,(P^{1}_t)^k\bigr)$ for the coordinates of $Q^{1}_t \in\R^k$, $P^{1}_t\in\R^k$, it is checked that each pair $((Q^{1}_t)^i,(P^{1}_t)^i)$ in $\R^2$ for $i\in\{1,\ldots,k\}$ satisfies the auxiliary Langevin form~\eqref{eq:lemsollangevin} where $\tilde q_t \in \R^{d-1}$ are the remaining spatial variables and the constants are given by 
\begin{equation}\label{eq:PSC-lin}
    a= 1+\frac{1}{\eps}, \ \ b = \alpha+ \frac{1}{\eps}, \ \ c=\gamma\Bigl(1+\frac{1}{\eps}\Bigr), \ \ \tilde c= \sqrt{2\gamma\beta^{-1}}.
\end{equation}
We present the proof for the case that $V(q)$ does not have a quadratic term in $q^1$, i.e.\ $\alpha=0$ in~\eqref{eq:Valpha}. If $\alpha>0$ the constants $b$ above changes and $\theta$ changes accordingly. However, the scaling behaviour of $a,b,\theta$ remains the same independent of $\alpha \geq 0$ and is given in  \eqref{eq:phase-scaling}, \eqref{eq:phase-scaling2} below.

Recall from Appendix~\ref{app:VarConst} that $\theta=\sqrt{c^2-4ab}$ (defined in~\eqref{def:theta}) determines the form of the solution. Using the values of $a,b,c$ given above we have
\begin{equation*}
    \theta = \Bigl(1+\frac{1}{\eps} \Bigr)^{\frac12} \Bigl( \gamma^2\Bigl(1+\frac{1}{\eps}\Bigr)-\frac4\eps\Bigr)^{\frac12}.
\end{equation*}
Note that for $\eps>0$ we have $\theta\in\R$ if $\gamma=2$. Furthermore for $\gamma\neq 2$  
\begin{equation*}
    \theta \in \R \Longleftrightarrow \eps\geq \dfrac{4-\gamma^2}{\gamma^2} \ \ \text{and} \ \ \theta\in\mathbb C\backslash\R \Longleftrightarrow \eps<  \dfrac{4-\gamma^2}{\gamma^2}.
\end{equation*}    
In the case when $\gamma>2$, $4-\gamma^2<0$ and therefore $\theta\in \mathbb \R$ for any $\eps>0$. For $0< \gamma< 2$, $\theta\in \mathbb C\backslash\R$ if $\eps\in(0,\frac{1-\gamma^2}{\gamma^2})$. In conclusion, for small enough $\eps>0$
\begin{equation}\label{eq:PS-Cons-ThetaChoices}
    \theta\in\R \Longleftrightarrow \gamma\geq 2  \ \ \text{ and } \ \ \theta\in \mathbb C\backslash\R \Longleftrightarrow \gamma\in(0,2).
\end{equation}
Using Proposition~\ref{prop:sollangevin}, which is allowed since this solution applies for any $\theta\neq 0$ and $\eps$ small enough (specifically any $\eps>0$ when $\gamma^2\geq1$ or $\eps<\gamma^{-2}-1$ when $\gamma^2<1$), the solutions for $((Q^{1}_t)^i,(P^{1}_t)^i)$ in a vectorial form is
\begin{subequations}
\begin{align}
\begin{split}\label{eq:q-sol-PS}
\begin{aligned}
    Q^{1}_t &=  \Bigl[ \mfrac{1}{2}\Bigl(1-\mfrac{c}{\theta}\Bigr)e^{-\frac{1}{2}(c+\theta) t} + \mfrac{1}{2}\Bigl(1+\mfrac{c}{\theta}\Bigl)e^{-\frac{1}{2}(c-\theta) t} \Bigr] Q^{1}_0 - \mfrac{a}{\theta} \Bigl(e^{-\frac{1}{2}(c+\theta) t} - e^{-\frac{1}{2}(c-\theta) t}\Bigr)P^{1}_0 \\ 
    & + \mfrac{a}{\theta} \int_0^t\bigl( e^{-\frac{1}{2}(c+ \theta)(t-s)} -e^{-\frac{1}{2}(c - \theta)(t-s)}\bigr) \nabla_{q^1} V(Q_s) ds + \mfrac{a \tilde c }{\theta} \int_0^t \bigl(e^{-\frac{1}{2}(c - \theta)(t-s)} - e^{-\frac{1}{2}(c + \theta)(t-s)}\bigr) dW^1_s  \,,
\end{aligned}
\end{split}\\
\begin{split}\label{eq:p-sol-PS}
    \begin{aligned}
    P^{1}_t &= \mfrac{b}{\theta} \Bigl(e^{-\frac{1}{2}(c+\theta) t} - e^{-\frac{1}{2}(c-\theta) t}\Bigr)Q^{1}_0  + \mfrac{a}{\theta} \Bigl[\mfrac{2b}{c-\theta}e^{-\frac{1}{2}(c+\theta) t} - \mfrac{2b}{c+\theta}e^{-\frac{1}{2}(c-\theta) t}\Bigr] P^{1}_0  \\
    & \quad - \mfrac{2ab}{\theta} \int_0^t \Bigl( \mfrac{1}{c- \theta}e^{-\frac{1}{2}(c+\theta) (t-s)} - \mfrac{1}{c + \theta} e^{-\frac{1}{2}(c-\theta) (t-s)}\Bigr) \nabla_{q^1} V(Q_s) ds  \\ 
    & \qquad \quad  +\mfrac{2a b\tilde c}{\theta} \int_0^t \Bigl( \mfrac{1}{c- \theta}e^{-\frac{1}{2}(c+\theta) (t-s)} - \mfrac{1}{c + \theta} e^{-\frac{1}{2}(c-\theta) (t-s)}\Bigr) dW^1_s.
    \end{aligned}
\end{split}
\end{align}
\end{subequations}
We now prove part~\ref{item:phase-Q1}. Following the procedure outlined at the start of this section (recall~\eqref{eq:genPathBound}), which involves using Young's inequality, It\^o isometry, and the bound $|\nabla_{q^1} V|=|\nabla_{q^1}U|\leq C_U$ (recall~\eqref{eq:Valpha} and note that $\alpha=0$), we find for any $t\geq 0$
\begin{align}
    \E\Bigl[ |Q^{1}_t|^2\Bigr]    &\leq C \biggl\{  \E \biggl[\Bigl|\Bigl[ \mfrac{1}{2}\Bigl(1-\mfrac{c}{\theta}\Bigr)e^{-\frac{1}{2}(c+\theta) t} + \mfrac{1}{2}\Bigl(1+\mfrac{c}{\theta}\Bigl)e^{-\frac{1}{2}(c-\theta) t} \Bigr] Q^{1}_0 - \mfrac{a}{\theta} \Bigl(e^{-\frac{1}{2}(c+\theta) t} - e^{-\frac{1}{2}(c-\theta) t}\Bigr)P^{1}_0\Bigr|^2 \biggr] \notag\\ 
    & \ + \Bigl|\mfrac{a}{\theta}\Bigr|^2 \biggl( \int_0^t\bigl| e^{-\frac{1}{2}(c+ \theta)(t-s)} -e^{-\frac{1}{2}(c - \theta)(t-s)}\bigr| ds\biggr)^2  + \Bigl|\mfrac{a \tilde c }{\theta}\Bigr|^2 \int_0^t\bigl|e^{-\frac{1}{2}(c - \theta)(t-s)} - e^{-\frac{1}{2}(c + \theta)(t-s)}\bigr|^2 ds \biggr\} \notag\\
    & \eqqcolon C\left\{I_1^\eps+I_2^\eps+I_3^\eps\right\} \label{eq:phase-Q1}
\end{align}
where $I^\eps_i$ refers to the $i$-the term on the right-hand side of~\eqref{eq:phase-Q1} and $C=C(C_U)$ is a constant. As $\eps\to 0$, the asymptotic behaviour of the constants involved in~\eqref{eq:phase-Q1} is given by 
\begin{equation}\label{eq:phase-scaling}
    |\theta| = \bigO\left(\mfrac{1}{\eps}\right), \   \mfrac{a}{|\theta|} =\bigO(1), \ \mfrac{c}{|\theta|} =\bigO(1), \  c+|\theta| = \bigO\left(\mfrac{1}{\eps}\right), \  c-|\theta| = \bigO\left(\mfrac{1}{\eps}\right).
\end{equation}
Note that we write $|\theta|$ since $\theta $ is either purely imaginary (if $c^2 < 4ab$, which are defined in \eqref{eq:PSC-lin}) or real. Moreover, note that if $\theta $ is purely imaginary then $c + \Re(\theta) = c = \bigO(\eps^{-1})$ and if $\theta $ is real then $ c + \Re(\theta)= c + |\theta| .$\\
For the $I_2^\eps$ term in~\eqref{eq:phase-Q1} we find
\begin{align*}
    I_2^\eps &\leq \Bigl|\mfrac{a}{\theta}\Bigr|^2  \biggl( \int_0^t \bigl| e^{-\frac12 (c+\theta)(t-s)} \bigr|ds + \int_0^t \bigl| e^{-\frac12 (c-\theta)(t-s)} \bigr|ds\biggr)^2 \\
    & = \Bigl|\mfrac{a}{\theta}\Bigr|^2  \biggl( \int_0^t  e^{-\frac12 (c+\Re(\theta))(t-s)} ds + \int_0^t  e^{-\frac12 (c-\Re(\theta))(t-s)} ds\biggr)^2 \\
    & = \Bigl|\mfrac{a}{\theta}\Bigr|^2  \biggl( 
    \mfrac{2}{c+\Re(\theta)} \bigl[ 1-e^{-\frac12 (c+\Re(\theta))t}\bigr] + \mfrac{2}{c-\Re(\theta)} \bigl[ 1-e^{-\frac12 (c-\Re(\theta))t}\bigr] \biggl)^2 \\
    & = \Bigl|\mfrac{a}{\theta}\Bigr|^2  \biggl( 
    \mfrac{2}{c+\Re(\theta)} \bigl[ 1-e^{-\frac12 (c+\Re(\theta))t}\bigr] + \mfrac{2}{c-\Re(\theta))} \bigl[ 1-e^{-\frac12 (c-\Re(\theta))t}\bigr] \biggl)^2 \leq C\eps^2
\end{align*}
where $C>0$ is independent of $\eps$ and $t$. Here the final equality follows since $c+\Re(\theta),c-\Re(\theta)>0$ and the inequality follows by using~\eqref{eq:phase-scaling}. Using Young's inequality and then repeating the same arguments as above leads to 
\begin{align*}
    I_3^\eps \leq 2 \Bigl|\mfrac{a\tilde c}{\theta}\Bigr|^2  \biggl( 
    \mfrac{1}{c+\Re(\theta)} \bigl[ 1-e^{-(c+\Re(\theta))t}\bigr] + \mfrac{1}{c-\Re(\theta))} \bigl[ 1-e^{- (c-\Re(\theta))t}\bigr] \biggl) \leq C\eps.
\end{align*}
Finally, applying Young's inequality twice we arrive at 
\begin{align*}
    I_1^\eps 
    &\leq 4 \biggl\{ \E\bigl[|Q_0^{1}|^2\bigr]  \biggl ( \Bigl| \mfrac{1}{2}\Bigl(1-\mfrac{c}{\theta}\Bigr)\Bigr|^2 e^{-(c+\Re(\theta)) t} + \Bigl|\mfrac{1}{2}\Bigl(1+\mfrac{c}{\theta}\Bigl)\Bigr|^2e^{-(c-\Re(\theta)) t}  \biggr) \\
    &\qquad\qquad\qquad\qquad\qquad\qquad
    + \Bigl|\mfrac{a}{\theta}\Bigr|^2\E\bigl[|P_0^{1}|^2\bigr]  \Bigl[e^{-(c+\Re(\theta)) t} + e^{-(c-\Re(\theta)) t}\Bigr]^2  \biggr\}\\
    &\leq Ce^{-\frac{D}{\eps}t} \Bigl(\E\bigl[|Q_0^{1}|^2\bigr]+\E\bigl[|P_0^{1}|^2\bigr] \Bigr),
\end{align*}
where $C,D>0$ are independent of $\eps$, $t$, and the second inequality follows from~\eqref{eq:phase-scaling}. Substituting these bounds back into~\eqref{eq:phase-Q1} we arrive at the required estimate.

To estimate $P^{1}_t$, we need the following scaling behaviour of additional constants in~\eqref{eq:p-sol-PS} as $\eps\to 0$  
\begin{equation}\label{eq:phase-scaling2}
    \frac{b}{|\theta|}=\bigO(1), \  \frac{ab}{|\theta|} =\bigO\Bigl(\frac{1}{\eps}\Bigr), \ \frac{2b}{|c+\theta|}=\frac{2b}{|c-\theta|}=\bigO(1)
\end{equation}
as $\eps\to 0$. 
Using~\eqref{eq:p-sol-PS} and repeating the same arguments as above we find
\begin{align*}
    \E\Bigl[|P_t^{1}|^2\Bigr] &\leq C\biggl\{ \E\biggl[\biggl| \mfrac{b}{\theta} \Bigl(e^{-\frac{1}{2}(c+\theta) t} - e^{-\frac{1}{2}(c-\theta) t}\Bigr)Q^{1}_0  + \mfrac{a}{\theta} \Bigl[\mfrac{2b}{c-\theta}e^{-\frac{1}{2}(c+\theta) t} - \mfrac{2b}{c+\theta}e^{-\frac{1}{2}(c-\theta) t}\Bigr] P^{1}_0\biggr|^2\biggr]  \\
    & \quad +\biggl(\Bigl|\mfrac{2ab}{\theta}\Bigr| \int_0^t \Bigl| \mfrac{1}{c- \theta}e^{-\frac{1}{2}(c+\theta) (t-s)} - \mfrac{1}{c + \theta} e^{-\frac{1}{2}(c-\theta) (t-s)}\Bigr|  ds\biggr)^2   \\ 
    & \qquad \quad  +\Bigl|\mfrac{2a b\tilde c}{\theta}\Bigr|^2 \int_0^t \Bigl| \mfrac{1}{c- \theta}e^{-\frac{1}{2}(c+\theta) (t-s)} - \mfrac{1}{c + \theta} e^{-\frac{1}{2}(c-\theta) (t-s)}\Bigr|^2 ds \biggr\} \\
    &\leq C\Bigl\{ e^{-\frac{D}{\eps}t} \Bigl( \E \bigl[ |Q^{1}_0|^2\bigr]+\E \bigl[ |P^{1}_0|^2\bigr] \Bigr) + \eps^2  + \eps  \Bigr\}
\end{align*}
where the final inequality uses~\eqref{eq:phase-scaling},~\eqref{eq:phase-scaling2} which imply that $|\frac{ab}{\theta}||\frac{1}{(c+\theta)(c-\theta)}|=\bigO(\eps)$. Here $C,D>0$ are independent of $\eps$ and $t$.
This completes the proof of~\ref{item:phase-Q1}.

The proof of~\ref{item:phase-Q2P2} follows by repeating the proof of Proposition~\ref{prop:Behav-Q2P2}. Specifically, we use the estimates of \ref{item:phase-Q1} together with Theorem~\ref{cor:gronwall}. For the uniform in time $L^1$ bound we additionally note that \eqref{eq:PhasePathBound} by Jensen's inequality implies that
\begin{align*}
    \E \left[|Q^1_s|\right] 
    \leq C_1\Bigl( \eps + e^{-\frac{s}{\eps}C_2}\bigl(|Q^1_0| + |P^1_0|\bigr) \Bigr).
\end{align*}
 Assuming $\eps$ small enough, specifically we require $\eps<2C_2/\eta$ below, this leads to
\begin{align*}
    \int_0^t e^{-\frac{\eta}{2}(t-s)} \E\left[|Q^1_s|\right]  ds &\leq C  \int_0^t e^{-\frac{\eta}{2} (t-s)} \Bigl(  \sqrt{ \eps} +  e^{-\frac{s}{\eps} C_2} \left( |Q^1_0| +|P^1_0|\right)\Bigr) ds  \\
     &=  C e^{-\frac{\eta}{2}t} \int_0^t  \Bigl( e^{\frac{\eta}{2}s} \sqrt{ \eps} +  e^{s\left(-\frac{C_2}{\eps}+\frac{\eta}{2}\right) } \left( |Q^1_0| +|P^1_0|\right)\Bigr) ds\\ 
     &=  C e^{-\frac{\eta}{2}t} \int_0^t  \Bigl( e^{\frac{\eta}{2}s} \sqrt{ \eps} +  e^{s\left(\frac{\eta\eps-2C_2}{2\eps}\right) } \left( |Q^1_0| +|P^1_0|\right)\Bigr) ds\\
     &=  C e^{-\frac{\eta}{2}t}  \Bigl( \frac{2}{\eta} \sqrt{ \eps} (e^{\frac{\eta}{2}t}-1) + \frac{2\eps}{\eta\eps-2C_2} \bigl(e^{t\left(\frac{\eta\eps-2C_2}{2\eps}\right)} -1\bigr)   \left( |Q^1_0| +|P^1_0|\right)\Bigr) \\
    &\leq   C\left( \sqrt{\eps}  + e^{-\frac{\eta}{2} t} \frac{\eps}{2C_2-\eta\eps}\left( |Q^1_0| + \sqrt{\eps}|P^1_0|\right) \right) \\
    &\leq  C\left( \sqrt{\eps}  + e^{-\frac{\eta}{2} t} \eps\left( |Q^1_0| + \sqrt{\eps}|P^1_0|\right) \right),
\end{align*}
where we used that by Taylor's Theorem for some $h \in (0,\eps), \eps<C$ we have $\frac{\eps}{C-\eps} = \left(\frac{1}{C-h} + \frac{h}{(C- h)^2}\right)\eps \leq \left(\frac{1}{C} + \frac{\eps}{C^2}\right) \eps$.
\end{proof}

\section{Partial constraint limits}\label{app:PartCons-Limit-Proof}    

\subsection{Zero mass}
    
We now prove the zero mass limit discussed in Proposition~\ref{prop:Light}. 

\begin{proof}[Proof of Proposition~\ref{prop:Light}] \label{proof:zeromass}
Recall the (momentum-confined) Langevin dynamics~\eqref{eq:abc-slow},\eqref{eq:intro-0Mass-Fast}, 
where $Q^{1}_t,P^{1}_t\in\R^k$ and $Q^{2}_t,P^{2}_t \in \R^{d-k}$. 
Using the notation $Q^{1}_t = \bigl( (Q^{1}_t)^1,\ldots,(Q^{1}_t)^k\bigr)$, $P^{1}_t = \bigl( (P^{1}_t)^1,\ldots,(P^{1}_t)^k\bigr)$ for the coordinates of $Q^{1}_t \in\R^k$, $P^{1}_t\in\R^k$, it follows that each pair $((Q^{1}_t)^i,(P^{1}_t)^i)$ in $\R^2$ for $i\in\{1,\ldots,k\}$ satisfies the auxiliary Langevin form~\eqref{eq:lemsollangevin} where $\tilde q_t \in \R^{d-1}$ are the remaining spatial variables and the constants are given by 
\begin{equation*}
    a= 1+\frac1\eps, \ \ b = \alpha, \ \ c=\gamma\biggl(1+\frac1\eps\biggr), \ \ \tilde c= \sqrt{2\gamma\beta^{-1}}.
\end{equation*}
where $\alpha\geq 0$ is the parameter that characterises the quadratic part of potential $V$ (recall~\eqref{eq:Valpha}). In the following we deal with the case $\alpha=0$ and $\alpha>0$ separately, since the solution given by the variation of constants formula in Appendix~\ref{app:VarConst} differs.

\noindent\textbf{Case: $\alpha=0$.}
Recall that the results for the auxiliary Langevin dynamics~\eqref{eq:lemsollangevin} in Section~\ref{app:VarConst} only apply when $b\neq 0$ and to use Proposition~\ref{prop:varofconst} we need to calculate the matrix exponential for $A$ which is 
\begin{equation}\label{eq:LightMatExp}
    A=\begin{pmatrix}
        0 & a \\ 0 & -c
    \end{pmatrix}, \ \ 
    e^{At} = \begin{pmatrix}
        1 & \frac{a}{c} \bigl( 1-e^{-ct}\bigr) \\ 0 & e^{-ct}
    \end{pmatrix}.
\end{equation}
Repeating the same procedure as in Section~\ref{app:VarConst}, i.e.\ using the matrix exponential above we explicitly calculate the solution to~\eqref{eq:intro-0Mass-Fast} 
\begin{align}
    Q^{1}_t &=  Q^{1}_0 + \mfrac{1}{\gamma}\bigl(1-e^{-a\gamma t}\bigr) P^{1}_0 - \mfrac{1}{\gamma}\int_0^t \bigl(1-e^{-a\gamma (t-s)}\bigr) \nabla_{q^1}V(Q_s)ds + \sqrt{\mfrac{2\beta^{-1}}{\gamma}} \int_0^t \bigl( 1 - e^{-a\gamma (t-s)} \bigr) dW^1_s,\label{eq:Light-Q1-explicit} \\
    P^{1}_t &=  e^{-a\gamma t} P^{1}_0 - \int_0^t e^{-a\gamma (t-s)} \nabla_{q^1}V(Q_s)ds + \sqrt{2\gamma\beta^{-1}} \int_0^t e^{-a\gamma (t-s)} dW^1_s. \notag
\end{align}

We now prove part~\ref{item:Light-P1} of Proposition~\ref{prop:Light}. Using the strategy outlined in Section~\ref{sec:Strategy-constrained} we find
\begin{align*}
    \E\Bigl[|P_t^{1}|^2\Bigr] &\leq 3\biggl\{ 
    \E\biggl[|e^{-a\gamma t} P_0^{1}|^2\biggr]
    + C_U^2\biggl(\int_0^t e^{-a\gamma(t-s)}ds\biggr)^2
    +2\gamma\beta^{-1} \int_0^t e^{-2a\gamma(t-s)}ds
    \biggr\}\\
    &\leq C\Bigl( e^{-\frac{D}{\eps}t} \E\bigl[|P_0^{1}|^2\bigr]+ \eps^2+\eps\Bigr) 
\end{align*}
where the second inequality follows by explicitly calculating the integrals and using $a= \frac{\eps+1}{\eps}$. Here, the constants $C,D>0$ are independent of $\eps$ and $t$. This completes the proof of~\ref{item:Light-P1} for $\alpha=0$.

Next, we prove part~\ref{item:LightSlow} by providing bounds on the following terms 
\begin{align*}
    \bigl|(Q^{1},Q^{2},P^{2}_t)^\top - (\hat q_t,q_t,p_t)^\top\bigr|^2 = \bigl|\hat q_t-Q^{1}_t \bigr|^2 + \bigl| q_t-Q^{2}_t\bigr|^2 + \bigl|p_t-P^{2}_t \bigr|^2. 
\end{align*}
Using (note that we use $q_t$ for the limit of $Q^2_t$ as throughout the paper, and in this zero mass limit we use $\hat q_t$ for the limit of $Q^1_t$)
\begin{equation}\label{eq:q1-0mass}
    \hat q_t = \hat q_0 -\mfrac{1}{\gamma}\int_0^t \nabla_{q^1} V(\hat q_s,q_s)ds + \sqrt{\mfrac{2\beta^{-1}}{\gamma}}\int_0^t dW_s^1
\end{equation}
and~\eqref{eq:Light-Q1-explicit} we find
\begin{align*}
    \hat q_t-Q^{1}_t  &\leq \bigl(\hat q_0-Q^{1}_0 \bigr) +\mfrac{1}{\gamma}\bigl(1-e^{-a\gamma t}\bigr)P^{1}_0-\mfrac{1}{\gamma}\int_0^t \bigl( \nabla_{q^1} V(\hat q_s,q_s)-\nabla_{q^1} V(Q_s)\bigr) ds \\
    &\qquad\qquad+ \mfrac{1}{\gamma}\int_0^t e^{-a\gamma (t-s)}\nabla_{q^1} V(Q_s) ds -\sqrt{\mfrac{2\beta^{-1}}{\gamma}} \int_0^t e^{-a\gamma (t-s)}dW_s^1.
\end{align*}
Following the strategy outlined in \ref{sec:Strategy-constrained}, using the Young's inequality, $e^{-a\gamma t}\leq 1$ for any $t\geq 0$ for the second term, the Lipschitz bound on $V=U$ (since $\alpha =0$)  together with the Cauchy Schwartz inequality for the third term, and explicitly calculating the integrals above we arrive at 
\begin{align*}
    \bigl|\hat q_t-Q^{1}_t\bigr|^2 = C\biggl\{ \bigl|\hat q_0-Q^{1}_0\bigr|^2 + \bigl|P_0^{1}\bigr|^2 +t \int_0^t \bigl|(\hat q_s,q_s)^\top-Q_s\bigr|^2 ds + \eps^2 + \biggl|\int_0^t e^{-a\gamma(t-s)}dW_s^1\biggr|^2 \biggr\}
\end{align*}
where the $\eps$ terms arise since $a=1+\frac1\eps$ and $C$ is independent of $\eps$ and $t$. Note that by It\^o isometry the final term above satisfies 
\begin{equation*}
    \E\biggl[ \biggl|\int_0^t e^{-a\gamma(t-s)}dW_s^1\biggr|^2 \biggr] \leq C \eps
\end{equation*}
where $C>0$ is independent of $\eps$, $t$.  
Proceeding similarly, we also have the bounds 
\begin{align}
    \bigl|q_t-Q^{2}_t\bigr|^2 &\leq C\biggl\{ \bigl|q_0-Q^{2}_0\bigr|^2 + t\int_0^t \bigl|p_s-P^{2}_s\bigr|^2ds \biggr\} \label{eq:zero mass-q2-Q^2}\\ 
    \bigl|p_t-P^{2}_t\bigr|^2 &\leq C \biggl\{ \bigl|p_0-P^{2}_0\bigr|^2 + t \int_0^t \bigl|(\hat q_s,q_s)^\top-Q_s\bigr|^2 ds + t\int_0^t \bigl|p_s-P^{2}_s\bigr|^2 ds \biggr\}, \label{eq:zero mass-p2-P^2}
\end{align}
where $C$ is independent of $\eps$ and $t$. Altogether letting \[u(t)\coloneqq  \E\Big[ \bigl|(Q^{1}_t,Q^{2}_t,P^{2}_t)^\top - (\hat q_t,q_t,p_t)^\top\bigr|^2 \Big]\] 
we have by the above estimates and using Fubini's theorem 
\begin{align*}
 u(t) \leq    C\left(u(0)+ \E\big[ |P^1_0|^2\big] + t \int_0^t u(s) ds + \eps + \eps^2 \right).
\end{align*}
Using these bounds and applying Gronwall's lemma with $\delta(t) = C\left(u(0)+ \E[|P^1_0|^2]+  \eps + \eps^2 \right) $ and $\beta(s) \equiv Ct$ we arrive at the required bound
\begin{align*}
    u(t) \leq \delta(t) e^{\int_0^t \beta(s) ds} = C\left(u(0) + \E\left[|P^1_0|^2\right] + \eps + \eps^2\right) e^{Ct^2}.
\end{align*}
\paragraph{Pathwise bound.} Next we brifly discuss a pathwise bound which would make use of Doob's inequality for the stochastic integral term. More precisely, let \[\tilde u(T) \coloneqq  \E\bigg[ \sup\limits_{t \in [0,T]} \bigl|(Q^{1}_t,Q^{2}_t,P^{2}_t)^\top - (\hat q_t,q_t,p_t)^\top\bigr|^2 \bigg]. \]  
Then similarly as above, first taking the supremum over all bounds and then the expectation, we find
\begin{align*}
    \tilde u(T) \leq C \bigg(u(0) + \E(|P^1_0|^2) + T \int_0^T u(s) ds + \eps^2 +  \E\biggl[\sup_{t\in [0,T]} \biggl|\int_0^t e^{-a\gamma(t-s)}dW_s^1\biggr|^2 \biggr]\bigg)
\end{align*}
Next, using Doob's inequality we compute 
\begin{align*}
    \E\biggl[\sup_{t\in [0,T]} \biggl|\int_0^t e^{-a\gamma(t-s)}dW_s^1\biggr|^2 \biggr] 
    &= \E\biggl[\sup_{t\in [0,T]} e^{-a\gamma t} \biggl|\int_0^t e^{a\gamma s}dW_s^1\biggr|^2 \biggr] 
    \leq \E\biggl[\sup_{t\in [0,T]} \biggl|\int_0^t e^{a\gamma s}dW_s^1\biggr|^2 \biggr] \\
     &\leq 4 \int_0^T e^{2a\gamma s}ds  
    \leq \frac{2}{a\gamma} \exp^{2a\gamma T} = \bigO(\eps) e^{ C \frac{T}{\eps} }
\end{align*} 
and therefore
\begin{align*}
    \tilde u(T) \leq C_1 \bigg[u(0) + \E(|P^1_0|^2) + \eps^2 + \eps  e^{ C_2 \frac{T}{\eps} } + T \int_0^T \tilde u(s) ds  \bigg].
\end{align*}
Using Gronwall's lemma we arrive at 
\begin{align*}
    \tilde u(T) \leq C_1 \bigg(u(0) + \E\big[|P^1_0|^2\big] + \eps^2 + \eps  e^{ C_2 \frac{T}{\eps} }  \bigg) e^{C_1 T^2},
\end{align*}
which now is exploding as $\eps \to 0$.

\noindent\textbf{Case: $\alpha>0$.} We now use the general solution given in Proposition~\ref{prop:sollangevin} with parameters $a=1 +\eps^{-1}$, $b = \alpha >0$, $c=(1+\frac{1}{\eps})\gamma$ and $\theta= \frac{\gamma(\eps+1)}{\eps}\sqrt{1-\frac{4\alpha}{\gamma^2(\eps+1)}\eps}.$ In the following, if $\kappa\coloneqq \frac{4\alpha}{\gamma^2}>1 $ let $\eps < \frac{1}{\kappa-1},$ which guarantees that $\theta > 0$ and use that by Taylor's Theorem we can write 
\begin{align}\label{eq:scaling-quadpotzeromass}
    \frac{c}{\theta} = \frac{1}{\sqrt{1 - 4\alpha\gamma^{-2}\eps(\eps+1)^{-1}}} = 1 + \frac{2\alpha}{\gamma^2} \eps + \bigO(\eps^2), \quad \frac{a}{\theta} = \frac{c}{\gamma \theta}= \frac{1}{\gamma} + \bigO(\eps) \,, \quad c - \theta = \frac{2\alpha}{\gamma} + \bigO(\eps)\,.
\end{align}

Then with $\hat q_0=Q^1_0$ we have (where $Q^1_t$ has been calculated using~\eqref{eq:q-sol} and $\hat q_t$ is given in~\eqref{eq:q1-0mass})
\begin{align*}
    Q^1_t- \hat q_t &= \left[ \frac{1}{2}\left(1-\mfrac{c}{\theta}\right)e^{-\frac{1}{2}(c+\theta)t} + \frac{1}{2}\left(1+ \mfrac{c}{\theta}\right)e^{-\frac{1}{2}(c-\theta)t} - e^{-\frac{\alpha}{\gamma}t} \right]\hat q_0 - \mfrac{a}{\theta} \biggl( e^{-\frac{1}{2}(c+\theta)t} -e^{-\frac{1}{2}(c-\theta)t} \biggr)P^1_0 \\
    &+ \frac{a}{\theta}\int_0^t \left(e^{-\frac{1}{2}(c+\theta)(t-s)} - e^{-\frac{1}{2}(c-\theta)(t-s)}\right) \nabla_{q^1} U(Q_s) ds + \mfrac{1}{\gamma}\int_0^t e^{-\frac{\alpha}{\gamma}(t-s)} \nabla_{q^1} U(\hat q_s, q_s)ds \\
    &+\frac{a}{\theta}\sqrt{\mfrac{2\gamma}{\beta }} \int_0^t \Bigl( e^{-\frac{1}{2}(c-\theta)(t-s)} - e^{-\frac{1}{2}(c+\theta)(t-s)} \Bigr)  dW^1_s - \sqrt{\mfrac{2\beta^{-1}}{\gamma}}\int_0^t e^{-\frac{\alpha}{\gamma}(t-s)} dW_s^1 \,.
\end{align*}
Now, using \eqref{eq:scaling-quadpotzeromass} and adding and subtracting $\gamma^{-1} e^{-\frac{\alpha}{\gamma}(t-s)} \nabla_{q^1}U(Q_s)$ in the second line above we find
\begin{align*}
    Q^1_t- \hat q_t &= \left[ \bigO(\eps)e^{-\frac{1}{2} (c+\theta)t} + \bigl(1+ \bigO(\eps)\bigr)e^{-\left(\frac{\alpha}{\gamma} + \bigO(\eps)\right)t} - e^{-\frac{\alpha}{\gamma}t} \right]q^1_0 - \mfrac{1}{\gamma}\left(1 + \bigO(\eps) \right) \left( e^{-\frac{1}{2}(c+\theta)t} -e^{-\frac{\alpha}{\gamma}(1+ \bigO(\eps))t} \right) P^1_0  \\ 
   &+ \int_0^t \left( \frac{a}{\theta}e^{-\frac{1}{2}(c+\theta)(t-s)} + \frac{1}{\gamma} e^{-\frac{\alpha}{\gamma}t} - \frac{a}{\theta}e^{-\frac{1}{2}(c-\theta)(t-s)} \right) \nabla_{q^1} U(Q_s) ds \\
   & +\mfrac{1}{\gamma}\int_0^t e^{-\frac{\alpha}{\gamma}(t-s)} \left(\nabla_{q^1} U(\hat q_s,q_s) - \nabla_{q^1} U(Q_s) \right)ds \\
    &+\frac{a}{\theta}\sqrt{\mfrac{2\gamma}{\beta }} \int_0^t \Bigl(e^{-\frac{\alpha}{\gamma}(c-\theta)(t-s)} - e^{-\frac{1}{2}(c+\theta)(t-s)}\Bigr)  dW^1_s - \sqrt{\mfrac{2}{\beta\gamma}}\int_0^t e^{-\frac{\alpha}{\gamma}(t-s)} dW_s^1\,.
\end{align*}
Now, use that \eqref{eq:scaling-quadpotzeromass} implies 
\begin{align*}
 \left| e^{-\frac{\alpha}{\gamma}\tau} - \mfrac{\gamma a}{\theta} e^{-\frac{1}{2}(c-\theta)\tau} \right|
    &= \left| e^{-\frac{\alpha}{\gamma}\tau}\left( 1- (1+ \gamma\bigO(\eps)) e^{-\bigO(\eps)\tau}\right)\right| 
    \leq e^{-\frac{\alpha}{\gamma}\tau}  \left( \left| 1-e^{-\bigO(\eps)\tau} \right| + \gamma \bigO(\eps) e^{-\bigO(\eps)\tau} \right) = \bigO(\eps\tau).
\end{align*}
and similarly
\begin{align*}
 \left| \Big(1+ \frac{c}{\theta}\Big) e^{-\frac{1}{2}(c-\theta)\tau} - e^{-\frac{\alpha}{\gamma}\tau}  \right|
    &= \bigO(\eps\tau).
    \end{align*}
Squaring the difference $Q^1_t - \hat q_t$ as given above and taking the expectation, we arrive at
\begin{align*}
    \E \Bigl[|Q^1_t - \hat q_t|^2 \Bigr] \leq C \biggl(\eps^2 t^2 \E\left[|Q^1_0|^2 \right] + \E\left[|P^1_0|^2 \right] + \eps t^2+\eps^2 t^4 +  \int_0^t \E \biggl[|Q_s - (\hat q_s, q_s)^\top |^2 \biggr] ds \biggr).
\end{align*}
The $\eps t^2$ term arises from the stochastic integral terms (via It\^o isometry) and the $\eps^2 t^4$ arises due to the term in the second line.
The bounds for $|q^2-Q^2|^2$ and $|p^2-P^2|^2$ remain unchanged and are given in \eqref{eq:zero mass-q2-Q^2}, \eqref{eq:zero mass-p2-P^2}. Hence, the overall result only changes by the term $\eps^2 |Q^1_0|^2$ which was absent in the case $\alpha=0$.

For $P^1$ we now use the solution as given in \eqref{eq:p-sol} together with the scaling behaviour as given in \eqref{eq:scaling-quadpotzeromass}. This yields following the usual steps that
\begin{align*}
   \E\Big[ \bigl|P^1_t\bigr|^2 \Big]  &\leq C \bigg(\eps \E\left[|Q^1_0|^2 \right]  + (e^{-\frac{D}{\eps}t} + \eps)  \E\left[|P^1_0|^2 \right]   + \bigg(C_U\int_0^t  \mfrac{1}{c- \theta}e^{-\frac{1}{2}(c+\theta) (t-s)} + \mfrac{1}{c + \theta} e^{-\frac{1}{2}(c-\theta) (t-s)} ds \bigg)^2  \\ 
    & \qquad \quad  + \left. \int_0^t \Bigl( \mfrac{1}{c- \theta}e^{-(c+\theta) (t-s)} - \mfrac{1}{c + \theta} e^{-(c-\theta) (t-s)}\Bigr) ds \right) \\
    &\leq C \left(\eps \E\left[|Q^1_0|^2 \right]  + (e^{-\frac{D}{\eps}t} + \eps)  \E\left[|P^1_0|^2 \right]  + \eps^2 + \eps \right), 
\end{align*}
where we used that $(c-\theta)(c+\theta) = c^2 - \theta^2 = 4ab =4\alpha \eps^{-1}.$
\end{proof}

\subsection{Infinite mass}
    
We now prove the infinite mass limit discussed in Proposition~\ref{prop:heavy}. 
\begin{proof}[Proof of Proposition~\ref{prop:heavy}] \label{proof:impetus}
As in the previous proof, we discuss the cases $\alpha=0$ and $\alpha>0$ in $V(Q) = \alpha |Q|^2 + U(Q)$ separately, since the solutions differ (compare initial data for the limit). This proof follows the same arguments as above in the zero mass case, with the coefficients $a,c$ replaced by $a=\eps$ and $c=\eps \gamma$. We only present the difference in the calculations and intermediate results and refer to the proof above for details.

\noindent\textbf{Case: $\alpha=0.$ }
The choice $a=\eps, \ b=\alpha = 0$ and $c=\eps \gamma$ leads to 
the explicit solutions  
\begin{align}
    Q^{1}_t &=  Q^{1}_0 + \mfrac{1}{\gamma}\bigl(1-e^{-\eps \gamma t}\bigr) P^{1}_0 - \mfrac{1}{\gamma}\int_0^t \bigl(1-e^{-\eps \gamma (t-s)}\bigr) \nabla_{q^1}V(Q_s)ds + \sqrt{\mfrac{2\beta^{-1}}{\gamma}} \int_0^t \bigl( 1 - e^{-\eps \gamma (t-s)} \bigr) dW^1_s,\label{eq:Imp-Q1-explicit} \\
    P^{1}_t &=  e^{-\eps \gamma t} P^{1}_0 - \int_0^t e^{-\eps \gamma (t-s)} \nabla_{q^1}V(Q_s)ds + \sqrt{2\gamma\beta^{-1}} \int_0^t e^{-\eps \gamma (t-s)} dW^1_s. \notag
\end{align}

We now prove part~\ref{item:heavy-Q1}. Using the strategy outlined in Section~\ref{sec:Strategy-constrained} we arrive at
\begin{align*}
    \E\Bigl[\bigl|Q_t^{1}-Q^{1}_0\bigr|^2\Bigr] &\leq 3\biggl\{ 
    \E\biggl[\Bigl| \mfrac{1}{\gamma}\bigl(1-e^{-\eps \gamma t}\bigr) P^{1}_0\Bigr|^2\biggr]
    + \mfrac{C_U^2}{\gamma^2}\biggl(\int_0^t (1-e^{-\eps \gamma (t-s)}) ds\biggr)^2 +\mfrac{2\beta^{-1}}{\gamma} \int_0^t \bigl( 1 - e^{-\eps \gamma (t-s)} \bigr)^2 ds
    \biggr\}.
\end{align*}
By the mean value theorem we have $1-e^{-\eps \gamma \tau} = \eps \tau \gamma e^{- \eps \gamma x},$ for some $ x \in [0,\gamma\tau].$ This implies $1-e^{-\eps \gamma \tau} \leq \eps \tau \gamma$ which we use to bound the integrands as well as the first term. Computing the integrals we arrive at the overall bound
\begin{align*}
    \E\Bigl[|Q_t^{1}-Q^{1}_0|^2\Bigr] 
    &\leq C \eps^2 t^2 \Bigl( \E\bigl[|P_0^{1}|^2\bigr]+ (t + t^2)\Bigr), 
\end{align*}
where  $C$ is independent of $\eps$ and $t$.

For~\ref{item:heavySlow} which states the result for the convergence of the unconstrained variables we note that the above derived bound on $Q^1_t - Q^1_0$ gives
\begin{align*}
    \int_0^T \E\Bigl[\bigl|Q_s^{1}-Q^{1}_0\bigr|^2\Bigr]  ds
    &\leq C \int_0^T \Bigl(\eps^2 s^2 \E\bigl[|P_0^{1}|^2\bigr]+ \eps^2(s^3 + s^4)\Bigr) ds \leq C \eps^2 T^3\Bigl\{ \E\bigl[|P_0^{1}|^2\bigr] + (T+T^2 ) \Bigr\}
\end{align*}
and applying Theorem \ref{cor:gronwall} yields the result.

Next we prove part~\ref{item:heavyP1} which compares the momentum $P^{1}_t$ to the solution $p^1_t$ given by 
\begin{align}\label{eq:impetusp1}
    p^1_t =  p^1_0  - \int_0^t  \nabla_{q^1} V(Q^1_0,q_s) ds + \sqrt{2\gamma \beta^{-1}} \int_0^t  dW_s^1,
\end{align}
Adding a zero in the $\nabla V$ term, we find using Young's inequality and the Cauchy Schwarz inequality that 
\begin{align*}
    \E\Bigl[|P_t^{1} - p^1_t|^2\Bigr] &\leq 3\biggl\{ 
    \E\biggl[|e^{-\eps \gamma t} P_0^{1} - p^1_0|^2\biggr]
    + C_U^2\biggl(\int_0^t 1- e^{-\eps \gamma(t-s)}ds\biggr)^2  \\ & \qquad + L_U^2 t \int_0^t \E\left[ |Q_s - (Q^1_0,q_s)^\top|^2 \right] ds +\mfrac{2\gamma}{\beta} \int_0^t \left(1 - e^{-\eps \gamma(t-s)}\right)^2ds.
    \biggr\}
\end{align*}
Next, observe that by successive integration by parts we have for any $k \in \mathbb{N}$
\begin{align}\label{eq:intexptimespoly}
    \int_0^t e^{cs}s^k ds \leq C (t^k e^{ct} + \bar C).
\end{align}
Hence, using the results of \ref{item:heavy-Q1}-\ref{item:heavySlow} and \eqref{eq:intexptimespoly} we have 
\begin{align*}
    \int_0^t\E\left[|Q_s - q_s|^2  \right] ds &\leq C \eps^2 \int_0^t e^{\tilde C s} s^4 \left( \E\bigl[|P^1_0|^2\bigr] + s + s^2 \right) ds \leq C \eps^2 e^{\tilde C t} t^4 \left( \E\bigl[|P^1_0|^2\bigr] + t + t^2 \right).
\end{align*} 
Noting that
\begin{align*}
    \bigl|e^{- \eps \gamma t} P^{1}_0 - p^1_0 \bigr|^2 
    \leq e^{- 2\eps \gamma t} |P^{1}_0 - p^1_0 |^2  + |p^1_0|^2 (1-e^{- \eps \gamma t})^2 \leq e^{- 2\eps \gamma t} |P^{1}_0 - p^1_0 |^2  + \gamma|p^1_0|^2 \eps^2 t^2 
\end{align*}
we arrive at the final result 
\begin{align*}
     \E\Bigl[|P_t^{1} - p^1_t|^2\Bigr] &\leq C\biggl\{ e^{-2\eps \gamma t}
    \E\Bigl[| P_0^{1} - p^1_0|^2  \Bigr]
    +  \E\bigl[|p^1_0|^2\bigr]\eps^2 t^2  + \eps^2 (t^3+t^4)
     + \eps^2 e^{\tilde C t} t^4 \left( \E\bigl[|P^1_0|^2\bigr] + t + t^2 \right)\biggr\} \\
    &\leq C\biggl\{ e^{-2\eps \gamma t}
    \E\Bigl[| P_0^{1} - p^1_0|^2  \Bigr]
    +  \E\bigl[|p^1_0|^2\bigr]\eps t^2  +  \eps^2 t^4 e^{\tilde C t}  \E\left[ |P^{1}_0|^2\right]   + \eps^2 (t^3+t^4) +  \eps^2 (t^6 + t^5) e^{\tilde C t}
    \biggr\}.
\end{align*}

\paragraph{Case: $\alpha >0.$ } First note that $\theta = \sqrt{\eps^2 \gamma^2 - 4\eps \alpha} \in \mathbb{C} \backslash \R $ for small enough $\eps$, which is why we use the general solutions for $Q^1,P^1$  given in Corollary \ref{cor:sol_underdamped}.
 By l'Hospital it follows that $|\theta| \to 0$ as $\eps \to 0$ and observe that $|\theta|^2 = \bigO(\eps).$ Further observe that 
\begin{align*}
    \frac{c}{|\theta|} = \frac{\eps \gamma}{\eps \gamma \sqrt{\frac{\eps \gamma^2 - 4 \alpha }{\eps \gamma^2}}} = \frac{\sqrt{\eps \gamma^2}}{\sqrt{4\alpha - \eps \gamma^2}} = \bigO(\sqrt{\eps}) \ \text{ and } \ \frac{a}{|\theta|} = \frac{c}{\gamma|\theta|} = \bigO(\sqrt{\eps}).
\end{align*}
Using the explicit solution for $Q^1_t$ given in \eqref{eq:q-sol-complextheta} we find that similar to the case $\alpha=0$, we have
\begin{align*}
    \E \Bigl[ \left|Q^1_t -Q^1_0 \right|^2 \Bigr] & \leq  \left(\left|1- \cos\left(\frac{t |\theta|}{2} \right) \right| + \left|\frac{c}{|\theta|} \sin \left( \frac{t |\theta|}{2} \right) \right|^2 \right)\left|e^{-\frac{\eps \gamma}{2}t}\right|^2 \E\bigl[ |Q^1_0|^2 \bigr]    \\
    \qquad & + \frac{a^2}{|\theta|^2} \left|\sin \left( \frac{t |\theta|}{2} \right) \right|^2 \left|e^{-\frac{\eps \gamma}{2}t}\right|^2 \E\bigl[ |P^1_0|^2 \bigr]+ C \left(\frac{a^2C_U^2}{|\theta|^2} t + \frac{a^2}{|\theta|^2}t\right).
\end{align*}
By Taylor's Theorem we have $ f(x)\coloneqq 1- \cos(x) = \frac{1}{2}f''(y)x^2$ for some $y \in [0,x]$ and hence $|\cos(x)-1| \leq \frac{1}{2}|x|^2.$ Similarly $|\sin(x)| \leq |x|$ and we use both estimates for the first term.
Estimating the other $|\sin|$ term and 
the exponentials by one, we arrive at the final bound
\begin{align*}
    \E \Bigl[ \left|Q^1_t -Q^1_0 \right|^2 \Bigr] & \leq C\left( \eps^2 \left( t^2 + t^4 \right)   \E\bigl[|Q^1_0|^2\bigr] + \eps^2  \E\bigl[|P^1_0|^2\bigr] + \eps t \right)\,,
    \end{align*}
    which implies - in the same manner as in the case $\alpha=0$ - using Theorem \ref{cor:gronwall} now for $\alpha > 0$    that
    \begin{align*}
        \E\biggl[\left|\begin{pmatrix}
            q_t - Q^2_t \\ p_t - P^2_t
        \end{pmatrix} \right|^2\biggr] \leq \E\biggl[\sup\limits_{t \in [0,T]}\left|\begin{pmatrix}
            q_2 - Q^2_t \\ p_t - P^2_t
        \end{pmatrix} \right|^2 \biggr] \leq C   \left( \eps^2 (T^3 + T^5) \E\bigl[|Q^1_0|^2\bigr] + \eps^2 T \E\bigl[|P^1_0|^2\bigr] + \eps T^2 \right).
    \end{align*}
    The difference $P^1_t- p^1_t$, where $p^1_t$ is the solution to \eqref{eq:impetusp1} and the solution of $P^1_t$ is given in \eqref{eq:p-sol-complextheta} reads 
     \begin{align*}
    &\left|P^1_t - p^1_t  \right|   \leq
          \left| \frac{2\alpha}{|\theta|}  \sin\Bigl(t\mfrac{|\theta|}{2}\Bigr)Q^1_0\right| + \left|\cos\Bigl(t\mfrac{|\theta|}{2}\Bigr) - \frac{c}{|\theta|}\sin\Bigl(t \mfrac{|\theta|}{2}\Bigr) -1 \right| |P^1_0| \\
    & +  \int_0^t \left|\Bigl[ \cos\Bigl((t-s)\frac{|\theta|}{2}\Bigr) - 1 - \frac{c}{|\theta|} \sin\Bigl((t-s)\frac{|\theta|}{2}\Bigr) \Bigr] e^{-\frac{c}{2}(t-s)} - 1 + e^{-\frac{c}{2}(t-s)} \right| \left| \nabla U(Q_s) \right| ds  \\ &  - \int_0^t \left|\nabla U(Q^1_0,q_s) - \nabla U(Q_s) \right|ds+ \  \left| \tilde c \left( \int_0^t  \left(\cos\Bigl((t-s)\mfrac{|\theta|}{2}\Bigr) - \frac{c}{|\theta|} \sin\Bigl((t-s)\mfrac{|\theta|}{2}\Bigr) \right) e^{-\frac{c}{2}(t-s)} -1\right)   dW^1_s \right|
    \\ &\leq C \left\{ \alpha t  |Q^1_0| + \left(t+t^2\right)\eps|P^1_0| + \int_0^t \eps (s^2+ s) ds + \int_0^t |(Q^1_0,q_s)^\top - Q_s| ds \right\} \\
    &\quad + \left| \tilde c \int_0^t  \left(\left(\cos\Bigl((t-s)\mfrac{|\theta|}{2}\Bigr) -1 - \frac{c}{|\theta|} \sin\Bigl((t-s)\mfrac{|\theta|}{2}\Bigr) \right) e^{-\frac{c}{2}(t-s)} - 1 + e^{-\frac{c}{2}(t-s)} \right)   dW^1_s \right|,
    \end{align*}
    where we have added a zero $\nabla U(Q_s) - \nabla U(Q_s)$ in the first step and used the assumption on $U$ in the second step together with Taylor's Theorem to estimate the $\sin(x) $, $1- \cos(x)$ and $e^{x}$ terms.  Squaring the above inequality and taking the expectation, we can apply our general strategy (see Section \ref{sec:Strategy-constrained}) and compute for the stochastic integral term
    \begin{align*}
      \E \biggl[ & \left| \tilde c \int_0^t  \left(\cos\Bigl((t-s)\mfrac{|\theta|}{2}\Bigr) -1 - \frac{c}{|\theta|} \sin\Bigl((t-s)\mfrac{|\theta|}{2}\Bigr) \right) e^{-\frac{c}{2}(t-s)} -1 + e^{-\frac{c}{2}(t-s)}  dW^1_s \right|^2 \biggr] \\
      & \leq C \int_0^t \left(\left| \cos\Bigl((t-s)\mfrac{|\theta|}{2}\Bigr) - 1 \right|^2+  \left|\frac{c}{|\theta|} \sin\Bigl((t-s)\mfrac{|\theta|}{2}\Bigr)\right|^2 \right)\left|e^{-\frac{c}{2}(t-s)} \right|^2 + \left| e^{-\frac{c}{2}(t-s)} -1  \right|^2 ds \\
      & \leq C \int_0^t \eps^2 (s^2 +s^4) ds \leq C \eps^2 (t^3 +t^5)
    \end{align*}
    so that, altogether, we have
    \begin{align*}
       \E \Bigl[  |p^1_t - P^1_t|^2 \Bigr] & \leq C \left\{ t^2 \E\bigl[ |Q^1_0|^2\bigr] + \eps^2  \left( t^2 + t^4\right) \E\bigl[|P^1_0|^2 \bigr] + \eps^2 (t^3+t^5)  + t \int_0^t | (Q^1_0,q_s) -  Q_s|^2 ds \right\}.
       \end{align*}    
    Using the $Q^1$ and $Q^2$ estimates from above in $\int_0^t |( Q^1_0,q_s)^\top - Q_s|^2 ds = \int_0^t |Q^1_0-Q^1_s|^2 + |q_s - Q^2_s|^2 ds$ we arrive at the final result
   \begin{align*}
       \E \Bigl[ |p^1_t - P^1_t|^2 \Bigr] 
       & \leq C \left\{ t^2\left( 1+ \eps^2 (t^3+t^5)\right) \E\bigl[|Q_0^1|^2\bigr] + \eps^2 (  t^2 +  t^4) \E\bigl[|P^1_0|^2\big] + \eps^2( t^3 + t^5)+ \eps t^4 \right\}.
   \end{align*}
\end{proof}

\subsection{Infinite friction with/without fluctuation-dissipation}
We now prove the infinite-friction limit discussed in Proposition~\ref{prop:PartialMom}.

\begin{proof}[Proof of Proposition~\ref{prop:PartialMom}]\label{proof:inf-fric}
This proof closely mirrors the proof above of Proposition~\ref{prop:Light} and therefore we only point out the essential differences. 

Recall the pre-limit Langevin dynamics~\eqref{eq:abc-slow},\eqref{eq:intro-InfFriction-withFD-Fast} where $Q^{1}_t,P^{1}_t\in\R^k$ and $Q^{2}_t,P^{2}_t \in \R^{d-k}$. Using the notation $Q^{1}_t = \bigl( (Q^{1}_t)^1,\ldots,(Q^{1}_t)^k\bigr)$, $P^{1}_t = \bigl( (P^{1}_t)^1,\ldots,(P^{1}_t)^k\bigr)$ for the coordinates of $Q^{1}_t \in\R^k$, $P^{1}_t\in\R^k$. Each pair $((Q^{1}_t)^i,(P^{1}_t)^i)$ in $\R^2$ for $i\in\{1,\ldots,k\}$ satisfies the auxiliary Langevin form~\eqref{eq:lemsollangevin} where $\tilde q_t \in \R^{d-1}$ are the remaining spatial variables and the constants are given by 
\begin{equation*}
    a= 1, \ \ b = \alpha, \ \ c=\frac{\gamma}{\eps}, \ \ \tilde c= \sqrt{\frac{2\gamma\beta^{-1}}{\eps}},
\end{equation*}
where $\alpha\geq 0$ is the parameter that characterises the quadratic part of potential $V$ (recall~\eqref{eq:Valpha}). We first consider the case $\alpha=0$ and then discuss $\alpha>0$.

\textbf{Case: $\alpha=0$.}
Using the matrix exponential~\eqref{eq:LightMatExp}, but now with the values for $a,c$ given above we arrive at the corresponding solutions 
\begin{align*}
    Q^{1}_t &=  Q^{1}_0 + \mfrac{\eps}{\gamma}\bigl(1-e^{-\frac{\gamma}{\eps} t}\bigr) P^{1}_0 - \mfrac{\eps}{\gamma}\int_0^t \bigl(1-e^{-\frac{\gamma}{\eps} (t-s)}\bigr) \nabla_{q^1}V(Q_s)ds + \sqrt{\mfrac{2\beta^{-1}\eps}{\gamma}} \int_0^t \bigl( 1 - e^{-\frac{\gamma}{\eps} (t-s)} \bigr) dW^1_s,\\
    P^{1}_t &=  e^{-\frac{\gamma}{\eps} t} P^{1}_0 - \int_0^t e^{-\frac{\gamma}{\eps} (t-s)} \nabla_{q^1}V(Q_s)ds + \sqrt{\mfrac{2\gamma\beta^{-1}}{\eps}} \int_0^t e^{-\frac{\gamma}{\eps} (t-s)} dW^1_s.
\end{align*}
Using the standard procedure used in the earlier proofs we find that
\begin{align}
    \E\Bigl[|Q_t^{1} - Q_0^{1}|^2 \Bigr] &\leq
    C\biggl\{  \eps^2 \E\bigl[|P_0^{1}|^2\bigr] + \eps^2\biggl(\int_0^t \bigl( 1 - e^{-\frac{\gamma}{\eps} (t-s)} \bigr)ds  \biggr)^2 + \eps \int_0^t \bigl( 1 - e^{-\frac{\gamma}{\eps} (t-s)} \bigr)^2 ds \biggr\} \notag\\
    & \leq C \Bigl\{  \eps^2 \E\bigl[|P_0^{1}|^2\bigr] + \eps^2 t^2 + \eps t \Bigr\}, \label{eq:Q1-Q0-inffric}
\end{align}
where $C>0$ is independent of $\eps$ and $t$.  The proof of the bound on $P^{1}$ follows exactly as in the proof of Proposition~\ref{prop:Light}, which gives
\begin{align*}
\E\Big[\big|P^{1}_t\big|^2\Big]&\leq C\bigg(e^{-\frac{D}{\eps}t} |P^1_0|^2 +\eps^2 + \frac{2\gamma\beta^{-1}}{\eps} \biggl[\int_0^t e^{-\frac{2\gamma}{\eps}(t-s)}ds\biggr]\bigg)\leq C\bigg(e^{-\frac{D}{\eps}t} |P^1_0|^2 +\eps^2 + \beta^{-1}\bigg),
\end{align*}
for $C,D>0$ independent of $\eps$ and $t$.

The proof of~\eqref{eq:PartialMom-Slow} follows using Theorem \ref{cor:gronwall}  together with~\eqref{eq:Q1-Q0-inffric} 
\begin{align*}
    \E\biggl[\sup_{t\in [0,T]}\biggr|
        \begin{pmatrix}
           q_t - Q^{2}_t \\ p_t - P^{2}_t
        \end{pmatrix} \biggr|^2 \biggr] 
        &\leq C_1 e^{C_2 T}\biggl\{
        \E\biggl[\biggl|
        \begin{pmatrix}
           q_0 - Q^{2}_0 \\ p_0 - P^{2}_0
        \end{pmatrix} \biggr|^2\biggl] + L_U^2 T  \, \E\biggl[\int_0^T \bigl| Q^{1}_s -\hat q\bigr|^2ds \biggr]
        \biggr\} \\
        &\leq C_1 e^{C_2 T}\biggl\{
        \E\biggl[\biggl|
        \begin{pmatrix}
           q_0 - Q^{2}_0 \\ p_0 - P^{2}_0
        \end{pmatrix} \biggr|^2\biggl] + L_U^2 T^2 \biggl[\eps \Bigl(\eps\E\bigl[|P_0^{1}|^2\bigr] + \eps T^2+T)\Bigr) +  \E\bigl[|Q^{1}_0-\hat q|^2\bigr] \biggr]
        \biggr\} 
\end{align*}
where the constants $C_1,C_2>0$ are independent of $\eps$ and $T$. Here the second inequality follows by using the triangle inequality to $|Q_s^{1}-Q^{1}_0+Q^{1}_0-\hat q|$ and then applying  Fubini's theorem. This leads to the required estimate~\eqref{eq:PartialMom-Slow}. 

\textbf{Case: $\alpha>0$.} The proof of the results above changes slightly, in that we resort to the general solution  \eqref{eq:q-sol} and \eqref{eq:p-sol} with $a=1, b=\alpha, c= \gamma \eps^{-1}, \tilde c=\sqrt{2\gamma\beta^{-1}\eps^{-1}}$. To compute the scaling limits as $\eps \to 0$ for the constants involved in the solution we use $\sqrt{1-x} = 1 - \frac{x}{2} - \frac{x^2}{8} -\bigO(x^3)$  and  $(1-x)^{-\frac{1}{2}} = 1 + \frac{x} {2} + \frac{3 x^2}{8} + \bigO(x^3) $ around $x=0$ which yields 
\begin{align*}
    \theta &= \sqrt{c^2 - 4ab} = \sqrt{\gamma^2 \eps^{-2} - 4\alpha} 
    = \bigO(\eps^{-1}), \ \  \frac{c}{\theta} = \frac{1}{\sqrt{1-4\alpha\gamma^{-2} \eps^2}} = 1+ \bigO(\eps^2)\, \\
    c-\theta &= \mfrac{\gamma}{\eps}\biggl(1- \sqrt{1- \mfrac{4\alpha\eps^2}{\gamma^2}} \biggr) = \frac{\gamma}{\eps} \left(1 - (1- 2\alpha \gamma^{-1}\eps^2 - \bigO(\eps^4)) \right)= \bigO(\eps).  
    \end{align*}
Then following our general procedure we arrive at the same bound
\begin{align*}
    \E \Big[ \left| Q^1_t - Q^1_0 \right|^2 \Big] & \leq 
    C \bigg(\bigg|1 - \frac{1}{2}\bigg(1 + \frac{c}{\theta}\bigg)e^{-\frac{1}{2}(c-\theta)t} \bigg|^2 \E\big[ \big|Q^1_0 \big|^2\big] + \bigg|\frac{1}{2}\bigg(1 - \frac{c}{\theta}\bigg)e^{-\frac{1}{2}(c+\theta)t} \bigg|^2 \E\bigl[ \big|Q^1_0 \big|^2\bigr]   \\
     & \qquad  + \bigg|\frac{a}{\theta} \bigg|^2 \bigg(\E\big[ \big| P^1_0\big|^2 \big] + t^2 C_U^2 + \tilde c^2 t  \bigg)\bigg) \\
     & \leq C  \eps \Big( \eps (1 +t^2) \E \big[ \big|Q^1_0\big|^2 \big] + \eps\E \big[ \big|P^1_0\big|^2 \big] + \eps t^2  + t \Big), 
\end{align*}
where we have estimated the exponentials in the  $P^1_0$ term and the two integrals in the solution~\eqref{eq:q-sol} by one  and used $\frac{a\tilde c}{\theta}=\bigO(\sqrt{\eps})$.

To estimate the first term above we have used the following argument. Let $\eps \leq (\kappa + \delta)^{-\frac{1}{2}} $, where  $\kappa=4\alpha\gamma^{-2}>0$ and $ \delta >0$ is fixed. Using Taylor expansion of $(1+\frac{c}{\theta})$ and the explicit formula for $c-\theta$ we find 
\begin{align}
    \bigg|1 - \frac{1}{2}\bigg(1 + \frac{c}{\theta}\bigg)e^{-\frac{1}{2}(c-\theta)t} \bigg| 
    &= \bigg|1 - \frac{1}{2}\bigg(2+ \frac{\kappa\xi}{(1-\kappa\xi^2)^{\frac{3}{2}}}\eps \bigg)e^{-\frac{\gamma}{\eps}(1-\sqrt{\kappa\eps^2})t} \bigg|
    \notag \\ &\leq \big|1-e^{-\frac{\gamma}{\eps}(1-\sqrt{\kappa\eps^2})t}\big| + \bigg| \frac{\kappa\xi\eps}{2(1-\kappa\xi^2)^{\frac{3}{2}}} e^{-\frac{\gamma}{\eps}(1-\sqrt{\kappa\eps^2})t}\bigg|\label{eqInfMass-Q0bound}
\end{align}
where $\xi\in (0,\eps)$. We now estimate each of these terms separately. 

By Taylor expansion around $\eps=0$  we find
\begin{equation*}
    1-\sqrt{\kappa\eps^2} =  \bigg(\frac{\kappa}{\sqrt{1-\kappa\zeta^2}} + \kappa^2\zeta^2(1-\kappa\zeta^2)^{-\frac32}\bigg)\frac{\eps^2}{2},
\end{equation*}
for some $\zeta\in(0,\eps).$
Introducing 
\begin{equation*}
    g(\eps)\coloneqq \exp\bigg(-\frac{\gamma}{\eps}(1-\sqrt{\kappa\eps^2})t\bigg) = \exp\bigg(-\frac{\gamma\eps t}{2}\bigg(\frac{\kappa}{\sqrt{1-\kappa\zeta^2}} + \kappa^2\zeta^2(1-\kappa\zeta^2)^{-\frac32}\bigg)\bigg)
\end{equation*}
and using mean-value theorem along with $g(0)=1$ (as $\zeta=0$ as well) we find
\begin{align*}
    1-e^{-\frac{\gamma}{\eps}(1-\sqrt{\kappa\eps^2})t} = g'(\nu)\eps,
\end{align*}
for some $\nu\in (0,\eps)$. By definition
\begin{equation*}
    g'(\eps)=-\frac{\gamma t}{2}g(\eps) \bigg(\frac{\kappa}{\sqrt{1-\kappa\zeta^2}} + \kappa^2\zeta^2(1-\kappa\zeta^2)^{-\frac32}\bigg).
\end{equation*}
Clearly $|g(\eps)|\leq 1$ for small enough $\eps$. Furthermore, observe that $\zeta  \mapsto \frac{\kappa}{\sqrt{1-\kappa\zeta^2}} + \kappa^2\zeta^2(1-\kappa\zeta^2)^{-\frac32}, $ is increasing in $\zeta.$ Since $\zeta \in (0,\eps)$ and $\eps< (\kappa + \delta)^{-\frac{1}{2}}$ it follows that  
\begin{equation*}
    |g'(\eps)|\leq \frac{\gamma t}{2} \bigg(\sqrt{\frac{\kappa}{\delta}}+\sqrt{\frac{\kappa}{\delta^3}}\bigg) 
\end{equation*}
where the right-hand side is bounded uniformly in $\eps$, and so we can conclude that 
\begin{equation*}
    |1-e^{-\frac{\gamma}{\eps}(1-\sqrt{\kappa\eps^2})t}| \leq \bar C t \eps
\end{equation*}
for some $\bar C$ independent of $\eps$ and $t$. Using a similar argument as above where we choose $\eps^2<\frac1\kappa-\delta$ for some $\delta>0$ small enough and since $|e^{-\frac{\gamma}{\eps}(1-\sqrt{\kappa\eps^2})t}|\leq 1$ it also follows that
\begin{equation*}
    \bigg| \frac{\kappa\xi\eps}{2(1-\kappa\xi^2)^{\frac{3}{2}}} e^{-\frac{\gamma}{\eps}(1-\sqrt{\kappa\eps^2})t}\bigg| \leq  \hat C\eps, 
\end{equation*}
for some $\hat C>0$ independent of $\eps$ and $t$. Combining these bounds and substituting back into~\eqref{eqInfMass-Q0bound} we arrive at the claimed bound for the $\E[|Q^1_0|^2]$ term.

By Theorem \ref{cor:gronwall} with $\alpha > 0$, i.e. \eqref{eq:gronwall-L2}, this implies that 
\begin{align*}
    \E\biggl[\sup_{t\in [0,T]}\biggr|
        \begin{pmatrix}
           q_t - Q^{2}_t \\ p_t - P^{2}_t
        \end{pmatrix} \biggr|^2 \biggr] 
        &\leq C \biggl\{
       \E\biggl[\biggl|
        \begin{pmatrix}
           q_0 - Q^{2}_0 \\ p_0 - P^{2}_0
        \end{pmatrix} \biggr|^2\biggl] + T  \eps \Big(  \E \big[ \big|Q^1_0\big|^2 \big] + \eps\E \big[ \big|P^1_0\big|^2 \big] + \eps T^2  + T \Big)
        \biggr\}  \\
        & \qquad +  C\E\bigl[|Q^{1}_0-\hat q|^2\bigr]. 
\end{align*}
Similarly, for $P^1_t$ using \eqref{eq:p-sol}, the scaling behaviour of the constants involved, and repeating the same  procedure as above,  we find 
\begin{align*}
    \E \Big[ \left| P^1_t \right|^2 \Big] & \leq
    C_1 \bigg( \bigg|\frac{b}{\theta}\bigg|^2 \E\left[|Q^1_0|^2\right] + \bigg|\frac{a}{\theta}\bigg|^2 \left(\frac{4b^2}{(c-\theta)^2}e^{-(c+\theta)t} + \frac{4b^2}{(c+\theta)^2}e^{-(c-\theta)t} \right) \E \left[ |P^1_0|^2\right]  \bigg)\\
    & \qquad + C_2\bigg(\frac{2ab}{\theta} C_U \int_0^t \bigg[\frac{1}{c-\theta}e^{-\frac{1}{2}(c+\theta)(t-s)} + \frac{1}{c+\theta}e^{-\frac{1}{2}(c-\theta)(t-s)} \bigg] ds  \bigg)^2 \\
    & \qquad + C_3 \bigg(\frac{ab\tilde c}{\theta}\bigg)^2 \int_0^t \bigg(\frac{1}{(c+\theta)^{2}}e^{-(c-\theta)(t-s)} + \frac{1}{(c-\theta)^{2}}e^{-(c+\theta)(t-s)}\bigg) ds \\
     &\leq C   \left( \eps^2 \E \big[ \big|Q^1_0\big|^2 \big] +  (e^{-\frac{\gamma}{\eps}r} +\eps^4)\E \big[ \big|P^1_0\big|^2 \big] + \eps + \beta^{-1} \right), 
\end{align*}
where $C>0$ is independent of $\eps$ and $t$. The bound on the $P^1_0$ term follows by using  
\begin{equation*}
    \bigg|\frac{a}{\theta}\bigg|^2 \frac{4b^2}{(c-\theta)^2}=\bigO(1), \ \ e^{-(c+\theta)r} \leq e^{-\frac{\gamma}{\eps}r}, \ \ \frac{4b^2}{(c+\theta)^2} \leq D\eps^2
\end{equation*}
for some constant $D>0$ independent of $\eps$ and $t$. 
The bound on the first integral follow since $\frac{ab}{\theta^2} \sim \bigO(\eps^2)$ and
\begin{align*}
    \int_0^t \frac{1}{c-\theta} e^{-(c+\theta)(t-s)} ds = \frac{1}{c^2-\theta^2} (1- e^{-(c+\theta)t}) \leq \frac{1}{4ab}, \\
    \int_0^t \frac{1}{c+\theta} e^{-(c-\theta)(t-s)} ds = \frac{1}{c^2-\theta^2} (1- e^{-(c-\theta)t}) \leq \frac{1}{4ab} \,.
\end{align*}
The bound on the final (stochastic) integral also follows using the identities above along with 
\begin{align*}
    \biggl|\frac{ab\tilde c^2}{4\theta^2} \frac{1}{c-\theta} (1- e^{-(c+\theta)t})\biggr| &\leq \frac{ab\tilde c^2}{4\theta^2} = \frac{\alpha}{4} \frac{2\gamma}{\beta\eps} \frac{\eps^2}{\gamma^2-4\alpha\eps^2} \frac{\gamma}{\eps(2\alpha+\bigO(\eps^2))} \leq \frac{1}{2\beta}\frac{1}{2+\bigO(\eps)} \leq \frac{1}{4\beta}+\bigO(\eps),\\
    \biggl|\frac{ab\tilde c^2}{4\theta^2} \frac{1}{c+\theta} (1- e^{-(c+\theta)t})\biggr| &\leq \frac{ab\tilde c^2}{4\theta^2} = \frac{\alpha}{4} \frac{2\gamma\eps}{\gamma^2-4\alpha\eps^2}\frac{\eps}{\gamma(2+\bigO(\eps))} = \bigO(\eps^2).
\end{align*}
\end{proof}
\paragraph{Infinite friction without fluctuation-dissipation.}
We now comment on the setting of infinite friction without fluctuation discussed in Remark~\ref{rem:InfFric-withouFD}. The proof of the corresponding estimates follows as in the proof of infinite friction with fluctuation dissipation discussed above.  The key difference is that the constant $\tilde c = \sqrt{2\gamma\beta^{-1}}$ is independent of $\eps$. The final estimates in this setting read (note that $P^1_t$ vanishes below as opposed to the fluctuation-dissipation discussed above as the noise evolves on the slower scale) 
\begin{align*}
    \E\biggl[|Q_t^{1} - Q_0^{1}|^2 \biggr] 
    &\leq C_1 \Bigl\{  \eps^2\mathds{1}_{\{\alpha>0\}} (1+t^2) \E\bigl[|Q_0^{1}|^2\bigr] +\eps^2 \E\bigl[|P_0^{1}|^2\bigr] + \mathds{1}_{\{\alpha=0\}}\eps^2 t^2 + \mathds{1}_{\{\alpha>0\}}\eps t^2 +\eps^2 t \Bigr\},\\
    \E\biggl[|P_t^{1} |^2 \biggr] &\leq C_2   \left( \eps^2 \mathds{1}_{\{\alpha>0\}}\E \big[ \big|Q^1_0\big|^2 \big] +  (e^{-\frac{t}{\eps}} +\eps^4 \mathds{1}_{\{\alpha>0\}})\E \big[ \big|P^1_0\big|^2 \big] + \eps^2 \right), \\
    \E\biggl[\sup_{t\in [0,T]}\biggr|
        \begin{pmatrix}
           q_t - Q^{2}_t \\ p_t - P^{2}_t
        \end{pmatrix} &\biggr|^2 \biggr] 
        \leq \begin{cases}
        \eps^2 C_3 e^{C_4 T} T^2
        \Bigl(\E\bigl[|P_0^{1}|^2\bigr] + (T^2+T)\Bigr) \ \ \text{if $\alpha=0$}, \\
        \eps^2  C_5 \eta^{-1} T
        \Bigl(\eps^2 \E\bigl[|P_0^{1}|^2\bigr]+\E\bigl[|P_0^{1}|^2\bigr] + (T^2+T)\Bigr) &\text{if $\alpha>0$},
        \end{cases}
\end{align*}
where $C_i>0$ for $i\in\{1,\ldots,5\}$ is independent of $\eps$, $t$ and $T$, and $\eta\in (0,\gamma]$. We have assumed that $q_0=Q^2_0$, $p_0=P^2_0$, and $\hat q=Q^1_0$ to arrive at the final estimate.

\section{Pathwise estimates for constrained variables}\label{app:AlmostSure}
So far in this article we have provided pointwise in time $L^2$ estimates for the constrained variables, for instance see~\eqref{eq:PhasePathBound} which provides estimates on $\E[|Q^1_t|^2]$. Below we demonstrate how these results can be extended to pathwise estimates on $\sup_{t\in [0,T]}|Q^1_t|$ in the almost sure sense and on $\E[\sup_{t\in [0,T]}|Q^1_t|]$.  To simplify the discussion we focus on the setting of phase-space confined Langevin dynamics (see Proposition~\ref{prop:PS-Cons}) with $\alpha=0$,  using the solution to $Q^{1}_t$ given by~\eqref{eq:q-sol}. We choose $\gamma>2$ which ensures that $\theta\in \R$ (see~\eqref{eq:PS-Cons-ThetaChoices}) and simplifies the following discussion. Finally, we assume that the initial datum is deterministic. The computations below generalise straightforwardly when these assumptions are dropped.  

\subsection{Almost-sure estimates}
Since we are interested in estimates for any $\omega\in \Omega$ (where $\Omega$ is the underlying probability space), using~\eqref{eq:q-sol} we find (note that $\alpha=0$)
\begin{align*}
    |Q^{1}_t(\omega)| &\leq  \Bigl| \mfrac{1}{2}\Bigl(1-\mfrac{c}{\theta}\Bigr)e^{-\frac{1}{2}(c+\theta) t} + \mfrac{1}{2}\Bigl(1+\mfrac{c}{\theta}\Bigl)e^{-\frac{1}{2}(c-\theta) t} \Bigr| \bigl|Q^{1}_0\bigr| + \Bigl|\mfrac{a}{\theta}\Bigr| \Bigl|e^{-\frac{1}{2}(c+\theta) t} - e^{-\frac{1}{2}(c-\theta) t}\Bigr|\bigl|P^{1}_0\bigr| \\ 
    & \ + \Bigl|\mfrac{a}{\theta}\Bigr|C_U \int_0^t\bigl| e^{-\frac{1}{2}(c+ \theta)(t-s)} -e^{-\frac{1}{2}(c - \theta)(t-s)}\bigr|  ds + \Bigl|\mfrac{a \tilde c }{\theta}\Bigr| \biggl|\int_0^t \bigl(e^{-\frac{1}{2}(c - \theta)(t-s)} - e^{-\frac{1}{2}(c + \theta)(t-s)}\bigr) dW^1_s \biggr| \\
    & \ \eqqcolon K_1^\eps+K_2^\eps + K_3^\eps + K_4^\eps(\omega),
\end{align*}
where the first inequality follows from the uniform bound $|\nabla_{q^1} U|\leq C_U$ and we have assumed that the initial datum is deterministic for simplicity (this is straightforwardly generalised).

Using the scaling behaviour~\eqref{eq:phase-scaling} of the constants involved it is easily checked that 
\begin{equation*}
 K_1^\eps +K_2^\eps \leq Ce^{-\frac{t}{2\eps}}\Bigl\{ \bigl|Q^{1}_0\bigr| + \bigl|P^{1}_0\bigr| \Bigr\}.   
\end{equation*}
For $K_3^\eps$ we have the following estimate for any $\omega\in \Omega$
\begin{align*}
    K_3^\eps &\leq \Bigl|\mfrac{a}{\theta}\Bigr|C_U \int_0^t\bigl[ e^{-\frac{1}{2}(c+ \theta)(t-s)} +e^{-\frac{1}{2}(c - \theta)(t-s)}\bigr]  ds 
    = \Bigl|\mfrac{a}{\theta}\Bigr|C_U \biggl[ \mfrac{2}{c+\theta}\bigl(1-e^{-\frac12(c+\theta)t}\bigr) +  \mfrac{2}{c-\theta}\bigl(1-e^{-\frac12(c-\theta)t}\bigr)\biggr]\\
    & \qquad\qquad \leq C\eps (1-e^{-\frac{t}{\eps}}) \leq C\eps \xrightarrow{\eps\to 0} 0.
\end{align*}
Note that so far the bounds are reminiscent of the bounds used to prove Theorem~\ref{prop:PS-Cons}. Nevertheless, the final stochastic integral term $K_4^\eps$ requires a different treatment. Using $\lambda_{\pm}\coloneqq c\pm \theta$ and applying integration by parts, we can rewrite $K^\eps_4(\omega)$ as 
\begin{equation}\label{eq:K4}
    K^\eps_4(\omega) = \Bigl|\mfrac{a \tilde c }{\theta}\Bigr| \biggl| - \int_0^t \bigl(\lambda_{-}e^{-\lambda_{-}(t-s)} + \lambda_+ e^{-\lambda_+(t-s)}\bigr) (W^1_t-W^1_s) ds+ W^1_t \bigl(e^{-\lambda_+ t}-e^{-\lambda_- t}\bigr)\biggr|.
\end{equation}
Note that $|\frac{a \tilde c}{\theta}|=O(1)$ and $\lambda_+,\lambda_->0$ with $\lambda_+,\lambda_=O(\frac1\eps)$. Therefore, we can estimate 
\begin{equation*}
    \bigl|W^1_t \bigl(e^{-\lambda_+ t}-e^{-\lambda_- t}\bigr)\bigr| \leq D(\omega)e^{-\frac{t}{\eps}}, 
\end{equation*}
where $\Omega\ni\omega\mapsto D(\omega)$ with $\sup_{s\in [0,t]}|W^1_s|\leq D(\omega)$, where $D(\omega)$ is almost surely finite. 

So to bound $K_4^\eps$, we only need to bound the two integral terms in~\eqref{eq:K4}, both of which are of the form
\begin{align}\label{eq:AuxStocInt}
    I(t,\omega) = \bigg|\int_0^t c(\eps)e^{-c(\eps)(t-s)}(W^1_t-W^1_s)ds\bigg|,
\end{align}
with $c(\eps)>0$ and $c(\eps)=O(\frac1\eps)$. Note that $s\mapsto W_s$ is uniformly continuous on $[0,t]$, and therefore  there exists a modulus of continuity given by
\begin{equation*}
    \eta(\tau)\coloneqq \sup_{|s-s'|\leq \tau} \big|W_s(\omega)-W_{s'}(\omega)\big|, \ \text{ with } \ |\eta(\tau)|\leq K(\omega)|\tau|^{\zeta}
\end{equation*}
for $\zeta<\frac12$, where the bound follows since Brownian motion is $\zeta$-H\"older continuous with $\zeta<\frac12$ and $K(\omega)$ is almost surely finite. Now we estimate $I(t,\omega)$ in~\eqref{eq:AuxStocInt} which in turn will provide bounds on $K^\eps_4$. A change of variables via $u=c(\eps)(t-s)$ followed by introducing (yet to be specified) $R(\eps)\leq c(\eps)t$ leads to 
\begin{align*}
    I(t,\omega) \leq \int_0^{c(\eps)t} e^{-u}\eta\biggl(\frac{u}{c(\eps)}\biggr)du &= \int_0^{R(\eps)}e^{-u}\eta\biggl(\frac{u}{c(\eps)}\biggr)du  + \int_{R(\eps)}^{c(\eps)t} e^{-u}\eta\biggl(\frac{u}{c(\eps)}\biggr)du \\
    &\leq \eta\bigg(\frac{R(\eps)}{c(\eps)}\bigg) \int_0^{R(\eps)}e^{-u}du+2D(\omega)\int_{R(\eps)}^\infty e^{-u}du\\
    &\leq  \eta\bigg(\frac{R(\eps)}{c(\eps)}\bigg) + 2D(\omega)e^{-R(\eps)}
\end{align*}
where the second inequality follows since $\tau\mapsto \eta(\tau)$ is non-decreasing and $|\eta(\tau)|\leq 2D(\omega)$. 

With the choice  $R(\eps)=\log(\frac1\eps)$, which clearly satisfies $R(\eps)\leq c(\eps)t$ for $t$ fixed and $\eps>0$ small enough, we find
\begin{align*}
    I(t,\omega) \leq (K(\omega)+2D(\omega)) \Big[ \big(\eps\log(\eps^{-1})\big)^\zeta+\eps \Big].
\end{align*}

Substituting this bounds back into~\eqref{eq:K4} and combining all bounds for $K^\eps_i$ with $i\in\{1,\ldots,4\}$ we arrive at the pathwise almost sure bounds
\begin{equation*}
    |Q^1_t(\omega)| \leq C\big[ e^{-\frac{t}{2\eps}}\big\{|Q_0^1|+|P_0^1|\big\} +  \eps + (K(\omega)+2D(\omega)) \big(\eps\log(\eps^{-1})\big)^\zeta  \big] \xrightarrow{\eps\to 0} 0 \ \text{almost surely},
\end{equation*}
for any $\zeta<\frac12$. Note that this estimate also leads to a pathwise estimate, i.e.\ with $\sup_{t\in [r,T]}$ in front with the constants $D,K$ appropriately adjusted and the exponential on the right-hand side replaced by $e^{-\frac{r}{2\eps}}$.  

\subsection[Pathwise estimate in L1 sense]{Pathwise $L^1$-estimate}
Repeating computations as in the proof of Proposition~\ref{prop:PS-Cons} (see Appendix~\ref{App:PS-cons-proof}), for any $r\in [0,T]$ it follows that
\begin{align}
    &\E\biggl[\sup_{t\in [r,T]}|Q^{1}_t|\biggr]    \leq C \biggl\{  \E \biggl[\sup_{t\in [r,T]}\Bigl|\Bigl[ \mfrac{1}{2}\Bigl(1-\mfrac{c}{\theta}\Bigr)e^{-\frac{1}{2}(c+\theta) t} + \mfrac{1}{2}\Bigl(1+\mfrac{c}{\theta}\Bigl)e^{-\frac{1}{2}(c-\theta) t} \Bigr] Q^{1}_0 - \mfrac{a}{\theta} \Bigl(e^{-\frac{1}{2}(c+\theta) t} - e^{-\frac{1}{2}(c-\theta) t}\Bigr)P^{1}_0\Bigr| \biggr] \notag\\ 
    & + \Bigl|\mfrac{a}{\theta}\Bigr|^2 \biggl( \int_0^T\bigl| e^{-\frac{1}{2}(c+ \theta)(t-s)} -e^{-\frac{1}{2}(c - \theta)(t-s)}\bigr| ds\biggr)^2  + \Bigl|\mfrac{a \tilde c }{\theta}\Bigr|^2 \E\bigg[\sup_{t\in[0,T]}\biggl|\int_0^te^{-\frac{1}{2}(c - \theta)(t-s)} - e^{-\frac{1}{2}(c + \theta)(t-s)} dW^1_s \biggr|\bigg] \biggr\} \notag\\
    & \qquad\qquad\qquad\eqqcolon C\left\{I_1^\eps+I_2^\eps+I_3^\eps\right\} 
\end{align}
By repeating the calculations from the proof of Proposition~\ref{prop:PS-Cons}, the bounds for $I_1^\eps$ and $I_2^\eps$ are given by
\begin{equation*}
    I_1^\eps+I_2^\eps \leq C\Big\{e^{-\frac{D}{\eps}r}\big(|Q^1_0|+|P^1_0|\big)+\eps\Big\}.
\end{equation*}
As in the almost sure proof, the harder term to estimate is the stochastic integral, which using $\lambda_{\pm}=c\pm \theta >0$ can be rewritten as
\begin{equation*}
    I_3^\eps \leq \Bigl|\mfrac{a \tilde c }{\theta}\Bigr|^2\bigg\{
    \E\bigg[\sup_{t\in[0,T]}\biggl|\int_0^te^{-\frac{1}{2}\lambda_-(t-s)}dW^1_s\bigg] + \E\bigg[\sup_{t\in[0,T]}\biggl|\int_0^te^{-\frac{1}{2}\lambda_+(t-s)}dW^1_s\bigg]
    \bigg\},
\end{equation*}
Note that the stochastic integral $\int_0^t e^{-\frac12 \lambda_\pm (t-s)}dW_s^1$ is an Ornstein-Uhlenbeck process and therefore by~\cite{GraversenPeskir00} there exist constants $C_1,C_2$ such that 
\begin{equation*}
    C_1 \sqrt{\lambda_{\pm}^{-1}}\log\bigg(1+\frac{T\lambda_\pm}{2}\bigg)\leq\E\bigg[\sup_{t\in[0,T]}\biggl|\int_0^te^{-\frac{1}{2}\lambda_\pm(t-s)}dW^1_s\bigg] \leq C_2 \sqrt{\lambda^{-1}_{\pm}}\log\bigg(1+\frac{T\lambda_\pm}{2}\bigg).
\end{equation*}
Since $\lambda_\pm \sim \eps^{-1}$, it follows that 
\begin{equation*}
     I_3^\eps \leq 2C_2\sqrt{\eps} \log\bigg(1+\frac{T}{2\eps}\bigg).
\end{equation*}
Combining all these bounds together, we arrive at the $L^1$-pathwise estimate 
\begin{equation*}
    \E\biggl[\sup_{t\in [r,T]}|Q^{1}_t|\biggr] \leq C\biggl\{e^{-\frac{D}{\eps}r}\big(|Q^1_0|+|P^1_0|\big)+\eps + \sqrt{\eps}\log\biggl(1+\frac{T}{2\eps}\biggr)\biggr\} \xrightarrow{\eps\to 0} 0, 
\end{equation*}
for any fixed $T>0$. Here the constant $C>0$ independent of $\eps$ and $T$.

\section{Example of spatial confinement and steady states}\label{sec:spatialEx}

In this appendix we discuss a few auxiliary results: firstly, a simple linear example with spatial confinement which clearly shows how oscillations arise in the constrained momentum, and secondly, we discuss the steady states in the strong confinement setting.   

\subsection[Oscillating constrained-momentum in spatial-confinement]{Oscillating constrained momentum in spatial-confinement} \label{app:oscP1}
Observe that Proposition~\ref{prop:SpatCons} does not provide results for the limiting behaviour of $P_t^{1}$ directly. As we now discuss, no well-defined limit exists for the fast momentum variable $P^1_t$.

In the explicit solution for $P_t^{1}$ (given in \eqref{eq:P1-sol}) none of the  cosine terms vanishes as $\eps \to 0$. Instead, we can prove that the time integrated $P_t^{1}$ vanishes as $\eps \to 0$. This, together with the first observation suggests, that $P_t^{1}$ undergoes increasingly fast oscillations as $\eps \to 0$. To support this claim, we provide the calculation for the simplest case in which everything is explicitly computable. 

Consider the following spatially constrained model for $(Q_t,P_t)\in \R\times\R$: 
\begin{align*}
    dQ^1_t &= P^1_t dt \\
    dP^1_t &=  - \biggl(1+\frac1\eps\biggr) Q^{1}_t  dt - \gamma P^1_t dt + \sqrt{2\gamma\beta^{-1}} \,dW^1_t,
\end{align*}
which corresponds to $V(q)=\frac{1}{2}|q|^2$, i.e.\ $(Q^1,P^1)$ and $(Q^2,P^2)$ are decoupled.  

In what follows we compute the statistics of $P^1_t$ to better understand the results of Proposition~\ref{prop:SpatCons}. 
To this end consider the explicit solution of $P^1_t$ which we compute using Proposition~\ref{cor:sol_underdamped} 
which leads to 
\begin{align*}
    P^1_t &=  - \frac{2(1+\eps)}{\eps|\theta|}  \sin\Bigl(t\mfrac{|\theta|}{2}\Bigr)e^{-\frac{\gamma}{2}t}Q^1_0 + \Bigl[\cos\Bigl(t\mfrac{|\theta|}{2}\Bigr) - \frac{\gamma}{|\theta|}\sin\Bigl(t \mfrac{|\theta|}{2}\Bigr)\Bigr]e^{-\frac{\gamma}{2}t}P^1_0 \\
    &\qquad+ \  \sqrt{2\gamma\beta^{-1}} \int_0^t \biggl( \cos\Bigl((t-s)\mfrac{|\theta|}{2}\Bigr) - \frac{\gamma}{|\theta|} \sin\Bigl((t-s)\mfrac{|\theta|}{2}\Bigr) \biggr) e^{-\frac{\gamma}{2}(t-s)} dW^1_s.
\end{align*}
The mean solves 
\begin{align*}
    m_t^\eps = \E[P^1_t] =  - \frac{2(1+\eps)}{\eps|\theta|}  \sin\Bigl(t\mfrac{|\theta|}{2}\Bigr)e^{-\frac{\gamma}{2}t} \mathbb{E}[Q^1_0] + \Bigl[\cos\Bigl(t\mfrac{|\theta|}{2}\Bigr) - \frac{\gamma}{|\theta|}\sin\Bigl(t \mfrac{|\theta|}{2}\Bigr)\Bigr]e^{-\frac{\gamma}{2}t}\mathbb{E}[P^1_0],
\end{align*}
which follows by noting that the expectation of the It\^o integral term is zero. 
We observe that:
\begin{enumerate}[label=(\roman*)]
    \item for any $t>0$, $m_t^\eps\to 0$ as $\eps\to 0$ if and only if $\E[Q^1_0]=\E[P^1_0]=0$; 
    \item for any $t>0$, $m_t^\eps$ diverges and becomes unbounded as $\eps\to 0$ if $\E[Q^1_0]\neq 0$, and $m_t^\eps$ diverges, but stays bounded if $\E[Q^1_0]=0$, $\E[P^1_0]\neq 0$;
    \item for any $\eps>0$, $m_t^\eps\to 0$ as $t\to\infty$ for any $Q^1_0,P^1_0\in\R$.
\end{enumerate}

Assuming that $\E[Q^1_0]=\E[P^1_0]=0$, which implies that $\E[P^1_t]=0$, the covariance $\Sigma_t^\eps \in \R$ is given by 
\begin{align*}
    \Sigma_t^\eps & = \E\Bigl[ \bigl( P^1_t - \E[P^1_t]\bigr)^2\Bigr] = \E\bigl[(P^1_t)^2\bigr] \\
    &= 2\gamma\beta^{-1} \int_0^t \biggl( \cos\Bigl((t-s)\mfrac{|\theta|}{2}\Bigr) - \frac{\gamma}{|\theta|} \sin\Bigl((t-s)\mfrac{|\theta|}{2}\Bigr) \biggr)^2 e^{-\gamma(t-s)} ds \\
    & = \beta^{-1}\biggl[ 1-\frac{e^{-\gamma t} \bigl[|\theta|^2+\gamma ^2-\gamma ^2 \cos (|\theta| t)-\gamma  |\theta| \sin (|\theta| t)\bigr]}{|\theta|^2}\biggr],
\end{align*}
where the fourth equality follows by It\^o isometry, and the final equality follows by explicit integration. We now observe that: 
\begin{enumerate}[label=(\roman*)]
    \item for any $t>0$, $\Sigma_t^\eps \to \beta^{-1}(1-e^{-\gamma t})$ as $\eps\to 0$;
    \item for any $\eps>0$, $\Sigma_t^\eps\to \beta^{-1}$ as $t\to \infty$.
\end{enumerate}

In conclusion, if $\E[Q^1_0]=\E[P^1_0]=0$
then
\begin{align*}
    \mathrm{law}(P^1_t) \to \begin{cases}
    \mathcal N\bigl(0,\beta^{-1}(1-e^{-t})\bigr)  &\text{ as }  \eps\to 0 ,\\ 
    \mathcal N(0,\beta^{-1}) &\text{ as }  t\to \infty,
    \end{cases} 
\end{align*}
and else if $\E[Q^1_0],\E[P^1_0]\neq 0$ then 
\begin{align*}
    \mathrm{law}(P^1_t)  
    \begin{cases}
    \text{diverges }  &\text{ as } \eps\to 0, \\
    \to \mathcal N(0,\beta^{-1}) &\text{ as } t\to\infty.
    \end{cases}
\end{align*}

The oscillatory behaviour of the momentum variable under spatial confinement is well-known from the classical literature on soft-constrained Hamiltonian systems~\cite{RubinUngar57,bornemann1997homogenization}. A key difference is that in the Hamiltoninan case, the fluctuation in the momentum $P^1_t$ is only seen if $\E[Q^1_0],\E[P^1_0]\neq 0$; in our case the momentum fluctuates (captured by the diverging sine and cosine terms) due to the additional noise term (which injects the required energy to constantly keep the system moving). 

\subsection{Steady state for spatially-confined Langevin dynamics}\label{sec:Steady}

In this section we discuss the $\eps\to 0$ limit of the steady states for the spatially confined Langevin dynamics, recall discussion at the end of Section~\ref{sec:spatial}. 
Throughout this discussion, we assume that $\alpha>0$ in the definition of the potential $V$ (recall~\eqref{eq:Valpha}), which along with the Lipschitz bounds on $V$ (recall~\eqref{ass:VLip}) ensures that the Boltzmann distribution $Z^{-1}e^{-\beta H}$, with normalisation constant $Z$ and corresponding Hamiltonian $H$, is the unique steady state of the (pre-limit) Langevin dynamics~\eqref{eq:intro-genLang}. 

Again, we make use of the coordinate-projection $\xi:\R^d\to\R^k$ defined as $\xi(q)=q^1$ with $\xi^{-1}(0)=\{q\in\R^d:\xi(q)=q^1=0\}$.  
The following proposition discusses the $\eps\to 0$ limit of $\mu^\eps$, defined in~\eqref{eq:SpatPreLimSte}, which is the steady state for the spatially confined Langevin dynamics~\eqref{eq:abc-slow},\eqref{eq:intro-Spat-Cons-Fast}.
\begin{prop} \label{prop:invmeas-spat}
The probability measure $\mu^\eps\in \mathcal P(\R^{2d})$ defined in~\eqref{eq:SpatPreLimSte} converges weakly to $\mu\in \mathcal P(\xi^{-1}(0)\times\R^d)$ given by 
\begin{align*}
    d\mu(q^2,p)\coloneqq \frac{1}{Z} \exp\biggl(-V(0,q^2) -\frac{|p|^2}{2}\biggr) dq^2 dp
\end{align*}
where $Z$ is the normalisation constant. 
In particular, for $f\in C_b(\R^d)$ with $f=f(q)$ we have
\begin{align*}
    \lim_{\eps\to 0} \int_{\R^{2d}}f(q)d\mu^\eps(q,p) = \int_{\R^{d-k}} f(0,q^2) \frac{1}{\hat Z} e^{-V(0,q^2)}dq^2,
\end{align*}
with $\hat Z\coloneqq \int_{\R^{d-k}}e^{-V(0,q^2)}dq^2$. 
Furthermore, the limit~\eqref{eq:limit-gen} of $(Q^2,P^2)$ with $\hat q=0$  admits the $(q^2,p^2)$-marginal of $\mu$, i.e.\ $\displaystyle\int_{\R^k}\mu\,dp^1$, as a steady state (cf.Proposition~\ref{prop:Behav-Q2P2}).
\end{prop}
\begin{proof}
We first note that 
\begin{align*}
    Z^\eps &= \int_{\R^{2d}} \exp\biggl(- V(q) - \frac{1}{2\eps}|q^1|^2 - \frac{1}{2}|p|^2\biggr) dqdp = (2\pi)^\frac{d}{2} \int_{\R^d}\exp\biggl(- V(q) - \frac{1}{2\eps}|q^1|^2 \biggr) dq \\
    & = (2\pi)^\frac{d}{2} \int_{\R^k} \biggl( \int_{\xi^{-1}(0)}e^{-V(q^1,q^2)}dq^2\biggr) e^{-\frac{1}{2\eps}|q_1|^2}dq^1
\end{align*}
where the final equality follows by Fubini's theorem since $e^{-\frac{1}{2\eps}|q^1|^2}\leq 1$ and $e^{-V}\in L^1(\R^{d})$. Therefore, using the variable transformation $\sqrt{\eps}u=q^1$, we have 
\begin{align*}
    \frac{Z^\eps_{\Xi_1}}{(2\pi\eps)^\frac{k}{2}} = \frac{(2\pi)^{\frac{d}{2}}}{(2\pi)^{\frac{k}{2}}} \int_{\R^k} \biggl( \int_{\xi^{-1}(\sqrt{\eps}u)} e^{-V(\sqrt{\eps u},q^2)}dq^2 \biggr) e^{-\frac12 |u|^2}du = \frac{(2\pi)^{\frac{d}{2}}}{(2\pi)^{\frac{k}{2}}} \int_{\R^k} h^{\eps}(u) e^{-\frac12 |u|^2}du,
\end{align*}
where $h^\eps$ is inner integral
\begin{align}\label{eq:SteSta-limit-Spat-DCT}
    h^\eps(u) = \mathds{1}_{\{\eta = \sqrt{\eps u}\}} 
    \int_{\xi^{-1}(\eta)} e^{-V(\eta, q^2)}  dq^2 \xrightarrow{\eps\to 0} \int_{\xi^{-1}(0)} e^{-V(0,q^2)}dq^2.
\end{align}
for any $\eta\in \R^{k}$
The limit above follows by dominated convergence theorem since $e^{-V(\sqrt{\eps} u,q^2)}\to e^{-V(0,q^2)}$ for any $u\in \R^k$ and $q^2\in\R^{d-k}$, and $|\mathds{1}_{\{\eta = \sqrt{\eps u}\}}e^{-V(\eta,q^2)}| \leq C e^{-\frac{\alpha}{2}|q^2|^2}\in L^1(\R^{d-k})$ for a constant $C>0$. The dominating function follows since $\alpha>0$; recall the definition of the potential $V$ is \eqref{eq:Valpha}.

Consequently, we arrive at
\begin{align*}
    \frac{Z^\eps_{1}}{(2\pi\eps)^\frac{k}{2}} \xrightarrow{\eps \to 0} \frac{(2\pi)^{\frac{d}{2}}}{(2\pi)^{\frac{k}{2}}} \int_{\R^k} \biggl( \int_{\xi^{-1}(0)} e^{-V(0,q^2)}dq^2 \biggr) e^{-\frac12 |u|^2}du = (2\pi)^{\frac{d}{2}} \int_{\xi^{-1}(0)} e^{-V(0,q^2)}dq^2.
\end{align*}
By using a similar argument, for any $f\in C_b(\R^{2d})$ we find
\begin{align*}
    \frac{1}{(2\pi\eps)^\frac{k}{2}} \int_{\R^{2d}} f(q,p) d\mu_{1}(q,p) \xrightarrow{\eps\to 0} \int_{\xi^{-1}(0)\times \R^{d}} f(0,q^2,p) e^{-V(0,q^2)-\frac{1}{2}|p|^2} dq^2dp,
\end{align*}
which leads to the required result. 
\end{proof}

\end{appendices}

{\small
\bibliographystyle{alphainitials.bst}
\bibliography{sample}
}
\end{document}